\documentclass{article} 
\usepackage{iclr2025_conference,times}

\usepackage[colorlinks=true,breaklinks=true,bookmarks=true,urlcolor=blue,citecolor=blue,linkcolor=blue,bookmarksopen=false,draft=false]{hyperref}
\usepackage{graphicx}	
\usepackage[lofdepth,lotdepth]{subfig}
\usepackage{amsmath,mathtools,amsthm}	
\usepackage{amssymb,dsfont,bbm}	
\usepackage[ruled,vlined]{algorithm}  
\usepackage{xcolor}  
\usepackage{comment}
\usepackage[textwidth=3cm,textsize=footnotesize]{todonotes}
\usepackage{bm}

\usepackage{xargs}
\usepackage{fancyhdr}
\usepackage{natbib}

\usepackage{setspace}
\usepackage{lastpage}
\usepackage{upgreek}

\usepackage{mathtools}	
\usepackage{enumitem}
\usepackage{mathrsfs}
\usepackage{aliascnt}
\usepackage{cleveref}

\newcommand{\noise}[1]{{\varepsilon_{#1}}}

\newcommandx{\norm}[2][2=]{\Vert#1 \Vert_{{#2}}}
\def\thetas{\theta^{\star}}
\def\noisecov{\Sigma_\varepsilon^{\star}}
\def\trace{\operatorname{Tr}}
\def\prtheta{\bar{\theta}}

\def\Rem{\mathcal{R}}

\newcommand{\ConstRR}[1]{\mathsf{C}_{\operatorname{RR},#1}}
\newcommand{\ConstPR}[1]{\mathsf{C}_{#1}}
\newcommand{\Constlast}[1]{\mathsf{D}_{\operatorname{last},#1}}
\newcommand{\Conststep}[1]{\mathsf{C}_{\operatorname{step},#1}}

\newcommand{\barpsi}{\bar{\psi}}
\newcommand{\MKQ}{\mathrm{Q}}
\newcommand{\PP}{\mathbb{P}}

\def\Id{\mathrm{I}}
\def\rmd{\mathrm{d}}
\def\rme{\mathrm{e}}
\def\rset{\mathbb{R}}

\def\calC{\mathcal{C}}
\def\bgammatheta{\bar{\theta}_{\gamma}}
\def\covgammatheta{\bar{\Sigma}_{\gamma}}
\def\L{\operatorname{L}}
\def\H{\operatorname{H}^{\star}}

\newcommand{\continuation}{??}

\def\eqsp{\,}
\renewcommand{\vec}[1]{\operatorname{vec}\left( #1 \right)}

\newcommand{\indi}[1]{\mathbbm{1}_{#1}}

\def\P{\mathsf{P}}

\def\PE{\mathsf{E}}
\def\E{\PE}

\newcommandx{\CPE}[3][1=]{\mathsf{E}_{#1}\bigl[  #2 | #3 \bigr]}

\def\B{\mathsf{B}}

\def\nset{\mathbb{N}}
\def\rset{\mathbb{R}}
\def\R{\rset}

\DeclareMathAlphabet\mathbfcal{OMS}{cmsy}{b}{n} 

\newtheorem{theorem}{Theorem}
\crefname{theorem}{theorem}{Theorems}
\Crefname{Theorem}{Theorem}{Theorems}

\newaliascnt{lemma}{theorem}
\newtheorem{lemma}[lemma]{Lemma}
\aliascntresetthe{lemma}
\crefname{lemma}{lemma}{lemmas}
\Crefname{Lemma}{Lemma}{Lemmas}

\newaliascnt{corollary}{theorem}
\newtheorem{corollary}[corollary]{Corollary}
\aliascntresetthe{corollary}
\crefname{corollary}{corollary}{corollaries}
\Crefname{Corollary}{Corollary}{Corollaries}

\newaliascnt{proposition}{theorem}
\newtheorem{proposition}[proposition]{Proposition}
\aliascntresetthe{proposition}
\crefname{proposition}{proposition}{propositions}
\Crefname{Proposition}{Proposition}{Propositions}

\newaliascnt{definition}{theorem}

\aliascntresetthe{definition}
\crefname{definition}{definition}{definitions}
\Crefname{Definition}{Definition}{Definitions}

\newaliascnt{definitionProposition}{theorem}

\aliascntresetthe{definitionProposition}
\crefname{Proposition and Definition}{Proposition and Definition}{Proposition and Definition}
\Crefname{Proposition and Definition}{Proposition and Definition}{Proposition and Definition}

\newaliascnt{remark}{theorem}

\aliascntresetthe{remark}
\crefname{remark}{remark}{remarks}
\Crefname{Remark}{Remark}{Remarks}

\crefname{example}{example}{examples}
\Crefname{Example}{Example}{Examples}

\crefname{figure}{figure}{figures}
\Crefname{Figure}{Figure}{Figures}

\newtheorem{assumption}{\textbf{A}\hspace{-3pt}}
\Crefname{assumption}{\textbf{A}\hspace{-3pt}}{\textbf{A}\hspace{-3pt}}
\crefname{assumption}{\textbf{A}}{\textbf{A}}

\newtheorem{assumptionC}{\textbf{C}\hspace{-3pt}}
\Crefname{assumptionC}{\textbf{C}\hspace{-3pt}}{\textbf{C}\hspace{-3pt}}
\crefname{assumptionC}{\textbf{C}}{\textbf{C}}

\def\Zset{\mathsf{Z}}
\def\Zsigma{\mathcal{Z}}

\def\metricz{\mathsf{d}_{\Zset}}

\newcommand{\Wass}[1]{\mathbf{W}_{{#1}}}

\def\mrl\mathrm{L}
\def\F{{\mathcal{F}}}
\def\MK{{\rm Q}}
\def\MKK{{\rm K}}

\def\msa{\mathsf{A}}

\newcommand{\dto}{\overset{d}{\to}}

\def\iid{i.i.d.}
\def\gauss{\mathrm{N}}

\newcommand{\parenthese}[1]{\left( #1 \right)}

\usepackage{hyperref}
\usepackage{url}

\title{Nonasymptotic Analysis of Stochastic Gradient Descent with the Richardson–Romberg Extrapolation}

\author{%
Marina Sheshukova$^{1,5}$ \quad Denis Belomestny$^{2,1}$ \quad Alain Durmus $^{3}$  \quad \textbf{Eric Moulines}$^{3,4}$ \quad \\
\textbf{Alexey Naumov}$^{1,6}$ \quad \textbf{Sergey Samsonov}$^{1}$
\\
$^1$HSE University \quad $^2$Duisburg-Essen University \quad $^3$CMAP, UMR 7641, Ecole Polytechnique \\
$^4$Mohamed Bin Zayed University of AI \quad $^5$ Skolkovo Institute of Science and Technology \\
$^6$ Steklov Mathematical Institute of Russian Academy of Sciences\\ \texttt{\{msheshukova,anaumov,svsamsonov\}@hse.ru}\\ \texttt{\{alain.durmus,eric.moulines\}@polytechnique.edu}\\ \texttt{denis.belomestny@uni-due.de}
}

\iclrfinalcopy 
\begin{document}

\maketitle

\begin{abstract}
We address the problem of solving strongly convex and smooth minimization problems using stochastic gradient descent (SGD) algorithm with a constant step size. Previous works  suggested to combine the Polyak-Ruppert averaging procedure with the Richardson-Romberg extrapolation to reduce the asymptotic bias of SGD at the expense of a mild increase of the variance. We significantly extend previous results by providing an  expansion of the mean-squared error of the resulting estimator with respect to the number of iterations $n$. We show that the root mean-squared error can be decomposed into the sum of two terms: a leading one of order $\mathcal{O}(n^{-1/2})$ with  explicit dependence on a minimax-optimal asymptotic covariance matrix, and a second-order term of order $\mathcal{O}(n^{-3/4})$, where the power $3/4$ is best known. We also extend this result to the higher-order moment bounds. Our analysis relies on the properties of the SGD iterates viewed as a time-homogeneous Markov chain. In particular, we establish that this chain is geometrically ergodic with respect to a suitably defined weighted Wasserstein semimetric.
\end{abstract}

\section{Introduction}
\label{sec:intro}
Stochastic gradient methods are a fundamental approach for solving a wide range of optimization problems, with a broad range of applications including generative modeling \citep{goodfellow2014generative,GoodBengCour16}, empirical risk minimization \citep{van2000asymptotic}, and reinforcement learning \citep{Sutton1998,schulman2015trust,deeprl}. These methods solve the stochastic minimization problem
\begin{equation}
\label{eq:stoch_minimization}
\min_{\theta \in \rset^{d}} f(\theta) \eqsp, \qquad \nabla f(\theta) = \PE_{\xi \sim \PP_{\xi}}[\nabla F(\theta,\xi)] \eqsp, 
\end{equation}
where $\xi$ is a random variable with distribution $\PP_{\xi}$, and the  gradient $\nabla f$ of the function $f$ can be accessed only through (unbiased) noisy estimates $\nabla F$. Throughout this paper, we consider  strongly convex minimization problems admitting a unique solution $\thetas$. Arguably the simplest and one of the most widely used approaches to solve \eqref{eq:stoch_minimization} is the stochastic gradient descent (SGD). This algorithm constructs the sequence of updates 
\begin{equation}
\label{eq:sgd_recursion_main}
\theta_{k+1} = \theta_{k} - \gamma_{k+1} \nabla F(\theta_k,\xi_{k+1})\eqsp, \quad \theta_0 \in \rset^{d}\eqsp, 
\end{equation}
where $\{\gamma_k\}_{k \in \nset}$ are step sizes, either diminishing or constant, and $\{\xi_k\}_{k \in \nset}$ is a sequence of  \iid\ random variables with distribution $\PP_{\xi}$. The algorithm \eqref{eq:sgd_recursion_main} can be viewed as a special instance of the Robbins-Monro procedure \citep{robbins1951stochastic}. While the SGD algorithm remains one of the core algorithms in statistical inference, its performance can be enhanced by means of additional techniques that use e.g., momentum \citep{qian1999momentum}, averaging \citep{ruppert1988efficient,polyak1992acceleration}, or variance reduction \citep{defazio2014saga,nguyen2017sarah}. In particular, the celebrated Polyak-Ruppert algorithm proceeds with a trajectory-wise  averaging of the estimates
\begin{equation}
\label{eq:pr_averaged_est_gen}
\bar{\theta}_{n_0,n} = \frac{1}{n}\sum_{k=n_0+1}^{n+n_0}\theta_{k}\eqsp
\end{equation}
for some \(n_0>0.\) It is known \citep{polyak1992acceleration, fort2015central}, that under appropriate assumptions on $f$ and $\gamma_k$, the sequence of estimates $\{\bar{\theta}_{n_0,n}\}_{n \in \nset}$ is asymptotically normal, that is,
\begin{equation}
\label{eq:CLT_fort}
\sqrt{n}(\bar{\theta}_{n_0,n} - \thetas) \dto \gauss(0,\Sigma_\infty)\eqsp, \quad n\to \infty
\end{equation}
where $\dto$ denotes the convergence in distribution and $\gauss(0,\Sigma_\infty)$ denotes the zero-mean Gaussian distribution with covariance matrix $\Sigma_\infty$, which is asymptotically optimal, see \citet{fort2015central} for a discussion. On the other hand, quantitative counterparts of \eqref{eq:CLT_fort} rely on the root mean squared error (root-MSE) bounds of the form
\begin{equation}
\label{eq:rate_setting0}
\PE^{1/2}[\norm{\bar{\theta}_{n_0,n} - \thetas}^2] \leq \frac{\sqrt{\trace{\Sigma_{\infty}}}}{n^{1/2}} + \frac{C(f,d)}{n^{1/2 + \delta}} + \Rem(\norm{\theta_0 - \thetas},n)\eqsp.
\end{equation}
Here $\Rem(\norm{\theta_0 - \thetas},n)$ is a remainder term which reflects the dependence upon initial condition, $C(f,d)$ is some instance-dependent constant and $\delta > 0$. There are many studies establishing \eqref{eq:rate_setting0} for Polyak-Ruppert averaged SGD under various model assumptions, including \citet{moulines2011,gadat2017optimal}. In particular, \citet{li2022root}  derived the bound \eqref{eq:rate_setting0} with the rate $\delta = 1/4$, which is the best known among first-order methods, but their results apply to a modified two-timescale algorithm with multiple restarts (Root-SGD). In our work, we show that the same upper bound is achieved by a simple modification of the estimate $\bar{\theta}_{n_0,n}$ based on Richardson-Romberg extrapolation. The main contributions of the current paper are as follows:
\begin{itemize}[topsep=0pt]
\item We show that a version of SGD algorithm with constant step size, Polyak-Ruppert averaging, and Richardson-Romberg extrapolation lead to the root-MSE bound \eqref{eq:rate_setting0} with $\delta = 1/4$ when applied to strongly convex minimization problems. We obtain this result by leveraging the analysis of iterates of the constant step-size SGD as a Markov chain. It is important to note that this result is obtained  for a fixed step  size $\gamma$ of order  $ 1/\sqrt{n}$ with $n$ being a total number of iterations. This result requires that the number of samples $n$ is known a priori to optimize the step size $\gamma$.
\item We generalize the above result and obtain high-order moment bounds on the error. Selecting the constant step size $\gamma = 1/\sqrt{n},$ we obtain for $p \geq 2$ bounds of the form
\begin{equation}
\label{eq:rate_setting_rr}
\PE^{1/p}[\norm{\prtheta_{n}^{(RR)} - \thetas}^p] \leq \frac{c_0 p^{1/2}\sqrt{\trace{\Sigma_{\infty}}}}{n^{1/2}} + \frac{C(f,d,p)}{n^{3/4}} + \Rem(\norm{\theta_0 - \thetas},n,p)\eqsp,
\end{equation}
where $c_0$ is a universal constant, and $\prtheta_{n}^{(RR)}$ is a counterpart of $\bar{\theta}_{n_0,n}$ when using Richardson-Romberg extrapolation, see related definitions in \Cref{sec:richardson-romberg}. Our proof is based on a novel version of the Rosenthal inequality, which might be of independent interest. 
\end{itemize}
The rest of the paper is organized as follows. We provide a literature review on the non-asymptotic analysis of the SGD algorithm and its modifications, with an emphasis on constant step-size SGD and the Richardson-Romberg procedure in \Cref{sec:lit_review}. Next, we analyze constant step-size SGD by treating it as a Markov chain and study the properties of the Polyak-Ruppert averaged estimator \eqref{eq:pr_averaged_est_gen} in \Cref{sec:analysis_sgd}. In \Cref{sec:richardson-romberg}, we analyze the properties of Richardson-Romberg extrapolation applied to Polyak-Ruppert averaged SGD and derive bounds for the second-order and higher-order moments of the error.

\paragraph{Notations and definitions.} 
For $\theta_0, \ldots,\theta_k$ being the iterates of stochastic first-order method, we denote $\F_{k} = \sigma(\theta_0,\theta_1,\ldots,\theta_k)$. Let $(\Zset,\metricz)$ be a complete separable metric space equipped with its Borel $\sigma$-algebra $\Zsigma$. We call a function $c: \Zset \times \Zset \rightarrow \mathbb{R}_{+}$ a \emph{distance-like} function, if it is symmetric, lower semi-continuous and
$c(x, y) = 0$ if and only if $x = y$, and there exists $q \in \nset$ such that
$\metricz^q(x, y) \leq c(x, y) $. We denote by $\mathscr{C}(\xi, \xi')$ the set of couplings of probability measures $\xi$ and $\xi'$, that is, a set of probability measures on $(\Zset^{2},\Zsigma^{\otimes 2})$, such that for any $\Pi \in \mathscr{C}(\xi, \xi')$ and any $A \in \Zsigma$ it holds $\Pi(\Zset \times A)= \xi'(A)$ and $\Pi(A \times \Zset)= \xi(A)$. We define the Wasserstein semimetric associated to the distance-like function $c(\cdot,\cdot)$, as 
\begin{equation}
\label{eq:def_wasserstein_distanse}
\Wass{c}(\xi,\xi') = \underset{\Pi \in \mathscr{C}(\xi, \xi')}{\inf}\int_{\Zset \times \Zset} c(z, z')\Pi(\rmd z, \rmd z')\eqsp.
\end{equation}
Note that $\Wass{c}(\xi,\xi')$ is not necessarily a distance, as it may fail to satisfy the triangle inequality. In the particular case of $\Zset = \rset^{d}$, and $c_p(x, y) = \norm{x-y}^p$, $x,y \in \rset^{d}$, $p\geq 1$, we use a short notation for $\Wass{p}(\xi,\xi')$ defined by $\Wass{p}^{p}(\xi,\xi') = \Wass{c_p}(\xi,\xi')$. For $x, y \in \rset^d$ denote by $x \otimes y$ the tensor product of $x$ and $y$ and by $x^{\otimes k}$ the $k$-th tensor power of $x$. In addition, for a function $f: \rset^d \rightarrow \rset$ we denote by $\nabla^k f(\theta)$ the $k$-th differential of $f$, that is  $\nabla^k f(\theta)_{i_1, \cdots, i_k} = \frac{\partial^k f}{\prod_{j=1}^k \partial \theta_{i_j}}$. For any tensor $M \in (\rset^d)^{\otimes (k-1)}$, we define $\nabla^k f(\theta) M \in \rset^{d}$ by the relation $(\nabla^k f(\theta) M)_l = \sum_{i_1, \ldots, i_{k-1}} M_{i_1, \cdots, i_{k-1}}\nabla^k f(\theta)_{i_1, \cdots, i_{k-1}, l}$, where $l \in  \{1, \ldots ,d\}$. Also for any tensor $M \in (\rset^d)^{\otimes (k-1)}$ we define $\norm{M} = \underset{x^l\neq 0, l\in\{1,\ldots,k\}}{\sup}\frac{\sum_{i_1, \ldots, i_k}M_{i_1, \ldots, i_k}x^1_{i_1}\cdot \ldots \cdot x^{k}_{i_{k}}}{\norm{x^1}\cdot\ldots\cdot\norm{x^{k}}}$. For two sequences $\{a_n\}_{n \in \nset}$ and $\{b_n\}_{n \in \nset}$ we write $a_n \lesssim b_n$, if there is an absolute constant $c > 0$, such that $a_n \leq c b_n$ for any $n \in \nset$. Throughout this paper we use $c_0$ for an absolute constant, which values may vary from line to line.

\section{Literature review}
\label{sec:lit_review}
Constant step-size SGD has been widely studied in the literature. Its bias and MSE were studied for strongly convex problems in \citep{durmus2020biassgd}, for both the last iterate and the Polyak-Ruppert averaging. \citet{yu2021analysis} studied the bias and the asymptotic normality of the last iterate of SGD for non-convex problems under the Polyak-Lojasiewicz condition and its generalizations. \citet{vlatakis2023stochastic} considered constant step size methods for solving variational inequalities, characterizing the bias and establishing asymptotic normality. \cite{merad2023convergence} studied the convergence of constant step-size SGD iterates in $\Wass{p}$ distance, $p \geq 1$, and provided concentration bounds for both $\theta_n$ and $\bar{\theta}_{n_0,n}$. \citet{moulines2011} derived the root-MSE bound \eqref{eq:rate_setting0} with $\delta =1/6$  for the case of strongly convex functions $f$. \citet{gadat2017optimal} derived an MSE counterpart to \eqref{eq:rate_setting0} of the form 
\[ 
\textstyle 
\PE[\norm{\bar{\theta}_{n_0,n} - \thetas}^2] \leq \frac{\trace{\Sigma_{\infty}}}{n} + \frac{C(f,d)}{n^{1+\beta}}
\]
with $\beta = 1/4$, translating to the root-MSE bound \eqref{eq:rate_setting0} with $\delta =1/8$. \citet{li2022root} suggested the Root-SGD algorithm combining the ideas of the two-timescale stochastic approximation and multiple restarts, establishing a counterpart of \eqref{eq:rate_setting0} with $\delta = 1/4$. The recent series of papers \citep{huo2023bias,zhang2024constant,zhang2024prelimit} investigated stochastic approximation algorithms with both i.i.d. and Markovian data and constant step sizes. The authors proposed precise characterization of the bias together with an MSE bounds after applying the Richardson-Romberg correction. At the same time, these works only considered the $2$-nd moment of the error and did not aim to identify the leading term of the MSE bound, which aligns with the CLT covariance matrix $\Sigma_{\infty}$.

Richardson-Romberg (RR) extrapolation is a technique used to improve the accuracy of numerical approximations \citep{hildebrand1987introduction}, such as those from numerical differentiation or integration. It involves using approximations with different step sizes and then extrapolating to reduce the error, typically by removing the leading term in the error expansion. The one-step RR extrapolation was introduced to reduce the discretization error induced by an Euler scheme to simulate stochastic differential equation in \citet{talaytubaro90}, and later generalized for non-smooth functions in \citet{bally1996richardsonnonsmooth}. This technique was extended using multistep discretizations in \citet{pages2007}. RR extrapolation have been applied to Stochastic Gradient Descent (SGD) methods in \citet{durmus2016stochastic}, \citet{durmus2020biassgd}, \citet{merad2023convergence} and \citet{huo2024effectiveness},  to improve convergence and reduce error in optimization problems, particularly when dealing with noisy or high-variance gradient estimates. Recent papers \citep{allmeier2024computing,zhang2024constant,huo2024effectiveness,kwon2024two} consider applications of RR to different stochastic approximation problems with constant step-size, including $Q$-learning, and single- and two-timescale stochastic approximation.

\vspace{-2mm}
\section{Finite-time analysis of the SGD dynamics for strongly convex minimization problems}
\label{sec:analysis_sgd}
\subsection{Geometric ergodicity of SGD iterates}

We consider the following assumption on the function $f$ in the minimization problem \eqref{eq:stoch_minimization}.
\begin{assumption}
\label{ass:mu-convex}
The function $f$ is continuously differentiable and $\mu$-strongly convex on $\rset^{d}$, that is, there exists a constant $\mu > 0$, such that for any $\theta,\theta' \in \rset^{d}$, it holds that
\begin{equation}
\label{eq:strong_convex}
\frac{\mu}{2}\norm{\theta-\theta'}^2 \leq f(\theta) - f(\theta') - \langle \nabla f(\theta'), \theta - \theta' \rangle \eqsp.
\end{equation}
\end{assumption}

\begin{assumption}
\label{ass:L-smooth}
The function $f$ is $4$ times continuously differentiable and $\L_{2}$-smooth on $\rset^{d}$, that is, there is a constant $\L_{2} \geq 0$, such that for any $\theta,\theta' \in \rset^{d}$,
\begin{equation}
\label{eq:L-smooth}
\norm{\nabla f(\theta) - \nabla f(\theta')} \leq \L_{2} \norm{\theta - \theta'}\eqsp.
\end{equation}
Moreover, $f$ has bounded $3$-rd and $4$-th derivatives, that is, there exist $\L_3,\L_4\geq 0$ such that
\begin{equation}
\label{eq: bound_3_rd_derivative}
\norm{\nabla^{i}f(\theta)} \leq \L_{i}\eqsp \text{ for } i \in \{3,4\}\eqsp.
\end{equation}
\end{assumption}
 We aim to solve the problem \eqref{eq:stoch_minimization} using SGD with a constant step size, starting from initial distribution $\nu$. That is, for $k \geq 0$ and a step size $\gamma > 0,$ we consider the following recurrent scheme
\begin{equation}
\label{eq:sgd_recursion}
\theta^{(\gamma)}_{k+1} = \theta^{(\gamma)}_{k} - \gamma \nabla F(\theta^{(\gamma)}_k, \xi_{k+1})\eqsp, \quad \theta^{(\gamma)}_{0} = \theta_0 \sim \nu\eqsp, 
\end{equation}
where $\{\xi_{k}\}_{k \in \nset}$ satisfy the following condition: 
\begin{assumption}[p]
\label{ass:rand_noise}
$\{\xi_k\}_{k \in \nset}$ is a sequence of independent and identically distributed (\iid) random variables with distribution $\PP_{\xi}$, such that $\xi_i$ and $\theta_0$ are independent and for any  $\theta \in \rset^{d}$ it holds that
\[
\PE_{\xi \sim \PP_{\xi}}[\nabla F(\theta,\xi)] = \nabla f(\theta)\eqsp.
\]
Moreover, there exists $\tau_p$, such that $\E^{1/p}[\norm{\nabla F(\thetas, \xi)}^p] \leq \tau_p$, and for any $q = 2,\ldots,p$ it holds with some $\L_1 > 0$ that for any $\theta_1, \theta_2 \in \R^d$, 
\begin{equation}
\label{eq:L-co-coercivity-assum}
\L_1^{q-1} \norm{\theta_1-\theta_2}^{q-2} \langle \nabla f(\theta_1) - \nabla f(\theta_2),\theta_1-\theta_2\rangle \geq \PE_{\xi \sim \PP_{\xi}}[\norm{\nabla F(\theta_1, \xi) - \nabla F(\theta_2, \xi)}^{q}]  \eqsp.
\end{equation}
\end{assumption}

Assumption \Cref{ass:rand_noise}($p$) generalizes the well-known $\L_1$-co-coercivity assumption, see \citet{durmus2020biassgd}. A sufficient condition which allows for \Cref{ass:rand_noise}($p$) is to assume that $F(\theta,\xi)$ is $\PP_{\xi}$-a.s. convex with respect to $\theta \in \rset^{d}$. For ease of notation, we set
\begin{equation}
\label{eq:const_L_def}
\L = \max(\L_1, \L_2,\L_3,\L_4)\eqsp,
\end{equation}
and trace only dependence upon $\L$ in our subsequent bounds. In this paper we focus on the convergence to $\thetas$ of the Polyak-Ruppert averaged estimator defined for any $n \geq 0$, by
\begin{equation}
\label{eq:pr_averaged_est} 
\bar{\theta}_{n}^{(\gamma)} = \frac{1}{n}\sum_{k=n+1}^{2n}\theta^{(\gamma)}_{k}\eqsp.
\end{equation}
Many previous studies instead consider $\bar{\theta}_{n_0,n} = \frac{1}{n-n_0} \sum_{k=n_0+1}^{n} \theta_{k}^{(\gamma)}$ rather than $\bar{\theta}_{n}^{(\gamma)}$, where $n \geq n_0 + 1$ and $n_0$ denotes a burn-in period. However, when the sample size $n$ is sufficiently large, the choice of the optimal burn-in size $n_0$ affects the leading terms in the MSE bound of $\bar{\theta}_{n}^{(\gamma)} - \theta^\star$ only by a constant factor. Therefore, we focus on \eqref{eq:pr_averaged_est}, or equivalently, use $2n$ observations and set $n_0 = n$.
\vspace{-3mm}
\paragraph{Properties of $\{\theta_k^{(\gamma)}\}_{k \in \nset}$ viewed as a Markov chain.}
Under assumptions \Cref{ass:mu-convex},  \Cref{ass:L-smooth} and \Cref{ass:rand_noise}($2$), the sequence $\{\theta_k^{(\gamma)}\}_{k \in \nset}$ defined by the relation \eqref{eq:sgd_recursion} is a time-homogeneous Markov chain with the Markov kernel 
\begin{equation}
\label{eq:mk_kernel_theta}
\MK_{\gamma}(\theta,\msa) = \int_{\rset^{d}} \indi{\msa}(\theta - \gamma \nabla F(\theta, z)) \P_{\xi}(\rmd z) \eqsp, \quad \theta \in \rset^{d}\eqsp, \, \msa \in \B(\rset^{d}) \eqsp,
\end{equation}
where $\B(\rset^d)$ is a Borel $\sigma$-field of $\rset^d$. In \citet{durmus2020biassgd} it has been established that, under the stated assumptions, $\MK_{\gamma}$ admits a unique invariant distribution $\pi_{\gamma}$, if $\gamma$ is small enough. Previous studies, such as \citet{durmus2020biassgd} or \citet{merad2023convergence}, studied the convergence of the distributions of $\{\theta_k^{(\gamma)}\}_{k \in \nset}$ to $\pi_{\gamma}$ in the $p$-Wasserstein distance $\Wass{p}$, $p \geq 1$, associated with the Euclidean distance in $\rset^{d}$. Our main results require to switch to the non-standard  distance-like function, which is defined under \Cref{ass:mu-convex} and \Cref{ass:rand_noise}($2$) as follows: 
\begin{equation}
\label{eq:function_c_def}
c(\theta, \theta') = \norm{\theta-\theta'}\left(\norm{\theta - \theta^*} + \norm{\theta' - \theta^*} + \frac{2 \sqrt{2} \tau_2 \sqrt{\gamma}}{\sqrt{\mu}}\right)\eqsp, \qquad \theta,\theta'\in\rset^d \eqsp.
\end{equation}
Here the constants $\tau_2$ and $\mu$ are given in \Cref{ass:rand_noise}($2$) and \Cref{ass:mu-convex}, respectively. This distance-like function is designed to analyze $\{\theta_k^{(\gamma)}\}_{k \in \nset}$ under \Cref{ass:mu-convex} and \Cref{ass:rand_noise}($2$). In particular, it depends on the step size $\gamma$ and $\thetas$. Our first main result establishes \emph{geometric ergodicity} of the Markov kernel $\MK_{\gamma}$ with respect to the distance-like function $c$ from \eqref{eq:function_c_def}. 
\begin{proposition}
\label{lem:Wasserstein_ergodicity}
Assume \Cref{ass:mu-convex},  \Cref{ass:L-smooth}, and \Cref{ass:rand_noise}($2$). Then for any $\gamma \in (0;1/(2\L)],$ the Markov kernel $\MK_{\gamma}$ defined in \eqref{eq:mk_kernel_theta} admits a unique invariant distribution $\pi_{\gamma}$. Moreover, $\MK_{\gamma}$ is geometrically ergodic with respect to the cost function $c$, that is, for any initial distribution $\nu$ on $\rset^d$ and $k \in \nset$,
\begin{equation}
\Wass{c}(\nu\MK_{\gamma}^k, \pi_{\gamma}) \leq 4(1/2)^{k/m(\gamma)}\Wass{c}(\nu, \pi_{\gamma}) \eqsp,
\end{equation}
where $m(\gamma) = \lceil 2\log 4/(\gamma \mu)\rceil$.
\end{proposition}

\paragraph{Discussion.} The proof of \Cref{lem:Wasserstein_ergodicity} is provided in \Cref{sec:proof_wass_ergodicity}. Properties of the invariant distribution $\pi_{\gamma}$ were previously studied in literature, see e.g. \cite{durmus2020biassgd}. It particular, it is known \cite[Lemma~13]{durmus2020biassgd}, that the $2$-nd moment of $\pi_{\gamma}$ scales linearly with $\gamma$:
\begin{equation}
\label{eq:moment_stationary}
\int_{\rset^{d}}\norm{\theta-\thetas}^2 \pi_{\gamma}(\rmd \theta) \lesssim \frac{\gamma \tau_2}{\mu}\eqsp.
\end{equation}
This bound yields, using Lyapunov's inequality, that
\[ 
\int_{\rset^{d} \times \rset^{d}} \norm{\theta-\theta^{\prime}} \pi_{\gamma}(\rmd \theta) \pi_{\gamma}(\rmd \theta^{\prime}) \lesssim \sqrt{\frac{ \gamma \tau_2}{\mu}}\eqsp. 
\]
At the same time, expectation of the cost function $c(\theta,\theta')$ scales linearly with the step size $\gamma$:
\begin{equation}
\label{eq:moment_scaling_cost_function}
\int_{\rset^{d}}c(\theta,\theta^{\prime}) \pi_{\gamma}(\rmd \theta) \pi_{\gamma}(\rmd \theta^{\prime}) \lesssim \frac{ \gamma \tau_2}{\mu}\eqsp.
\end{equation}
The property \eqref{eq:moment_scaling_cost_function} is crucial to obtain tighter  error bounds with respect to the step size $\gamma$ for the Richardson-Romberg estimator, as well as to obtain the improved version of the Rosenthal inequality for additive functional of $\{\theta_k^{(\gamma)}\}_{k \in \nset}$ derived in \Cref{prop:psi_p_moment_bound}. Precisely, the additional $\sqrt{\gamma}$ factor obtained in  \eqref{eq:moment_scaling_cost_function} would allow us to obtain sharper bounds on the remainder terms in \Cref{th:RR_second_moment}. Next, we analyze the error $\theta_{\infty}^{(\gamma)} - \thetas$ where $\theta_{\infty}^{(\gamma)}$ is distributed according to the stationary  distribution $\pi_{\gamma}$. To this end, we consider the following condition.
\begin{assumptionC}[p]
\label{ass:stationary_moments_bounds}
There exist constants $\Constlast{p}, \Conststep{p} \geq 2$ depending only on $p$, such that for any step size $\gamma \in (0, 1/(\L \Conststep{p})]$, and any initial distribution $\nu$ it holds that 
\begin{equation}
\label{eq:arbitrary_distr_p_moment_bound} 
\PE_{\nu}^{2/p}\bigl[\norm{\theta_k^{(\gamma)} - \thetas}^{p}\bigr] \leq (1 - \gamma \mu)^{k} \PE_{\nu}^{2/p}\bigl[\norm{\theta_0 - \thetas}^{p}\bigr] + \Constlast{p} \gamma \tau_{p}^2 / \mu\eqsp.
\end{equation}
Moreover, for the stationary distribution $\pi_{\gamma}$ it holds that 
\begin{equation}
\label{eq:stationary_p_moment_bound}
\PE_{\pi_{\gamma}}^{2/p}\bigl[\norm{\theta_{\infty}^{(\gamma)} - \thetas}^{p}\bigr] \leq \Constlast{p} \gamma \tau_{p}^2 / \mu \eqsp.
\end{equation}
\end{assumptionC}
It is important to note that \Cref{ass:stationary_moments_bounds} is not independent from the preceding assumptions \Cref{ass:mu-convex} - \Cref{ass:rand_noise}($p$). In particular, \citet[Lemma~$13$]{durmus2020biassgd} establishes that, under \Cref{ass:mu-convex},\Cref{ass:L-smooth}, and \Cref{ass:rand_noise}($p$), the bound \eqref{eq:stationary_p_moment_bound} holds for $\gamma \in (0, 1/(\L \Conststep{p})]$ with some constants $\Constlast{p}$ and $\Conststep{p}$ depending only upon $p$. Unfortunately, it is difficult to obtain precise dependence of $\Conststep{p}$ and $\Constlast{p}$ on $p$, and to obtain \eqref{eq:stationary_p_moment_bound} with precise numerical constants. Existing studies \citep{gadat2017optimal, merad2023convergence} either use different set of assumptions or do not explicitly characterize their dependence on $p$. That is why we prefer to state \Cref{ass:stationary_moments_bounds}($p$) as a separate assumption. In the subsequent bounds we use \Cref{ass:stationary_moments_bounds}($p$) together with \Cref{ass:mu-convex},\Cref{ass:L-smooth}, and \Cref{ass:rand_noise}($p$), tracking the dependence of our bounds upon $\Conststep{p}$ and $\Constlast{p}$. We leave the derivation of \Cref{ass:stationary_moments_bounds}($p$) with precise constants $\Constlast{p}, \Conststep{p}$ as an interesting direction for future research.

Under assumption \Cref{ass:stationary_moments_bounds}, we control the fluctuations of $\theta_k^{(\gamma)}$ around $\thetas$. However, unless $f$ is quadratic, it is known that $\int_{\rset^{d}} \theta \pi_{\gamma}(\rmd \theta) \neq \thetas $.
In the next proposition, we quantify this bias under weaker assumptions than those in \citet[Theorem~4]{durmus2020biassgd}.

\begin{proposition}
\label{prop:bias_PR}
Assume \Cref{ass:mu-convex}, \Cref{ass:L-smooth}, \Cref{ass:rand_noise}($6$), and \Cref{ass:stationary_moments_bounds}($6$). Then there exist such $\Delta_{1} \in \rset^{d}, \Delta_2 \in \rset^{d\times d}$, not depending upon $\gamma$, that for any $\gamma \in (0, 1/(\L \Conststep{6})]$, it holds
\begin{equation}
\label{eq:bias_expansion_result}
\bgammatheta := \int_{\rset^{d}} \theta \pi_{\gamma}(\rmd \theta) = \thetas + \gamma \Delta_{1} + B_1 \gamma^{3/2}\eqsp,
\end{equation}
\begin{equation}
\label{eq:bias_expansion_result_2}
\covgammatheta := \int_{\rset^{d}} (\theta-\thetas)^{\otimes 2} \pi_{\gamma}(\rmd \theta) =  \gamma \Delta_2 + B_2 \gamma^{3/2}\eqsp.
\end{equation}
Here $B_1 \in \rset^{d}$ and $B_2 \in \rset^{d \times d}$ satisfy $\norm{B_1} \leq \ConstPR{b,1}$, $\norm{B_2}\leq \ConstPR{b,2}$, where $\ConstPR{b,1}$ and $\ConstPR{b,2}$ are constants independent of $\gamma$ and are defined in \eqref{eq: def_ConstPRb1,2}. Moreover, for any initial distribution $\nu$ on $\rset^{d}$, it holds that 
\begin{equation}
\label{eq:bias_PR_bound}
\PE_{\nu}[\bar{\theta}_{n}^{(\gamma)}] = \thetas + \gamma\Delta_{1} + B_1 \gamma^{3/2} + \Rem_1(\theta_0-\thetas, \gamma, n)\eqsp, 
\end{equation}
where 
\begin{equation}
\textstyle 
\norm{\Rem_1(\theta_0-\thetas, \gamma, n)} \lesssim \frac{\rme^{-\gamma\mu(n+1)/2}}{n\gamma\mu}\parenthese{\PE_{\nu}^{1/2}\bigl[\norm{\theta_0 -\thetas}^2\bigr] + \frac{\sqrt{\gamma}\tau_2}{\sqrt{\mu}}}\eqsp.
\end{equation}
\end{proposition}
The proof is provided in \Cref{sec:proof_prop_bias_PR_and_RR}. Results of this type are known in the literature for stochastic approximation algorithms, see e.g. \citet{huo2024collusion} and \citet{allmeier2024computing}. The additive term $\Delta_1$ vanishes when the function $f$ is quadratic, see \citep{moulines2011}. 

\vspace{-2mm}
\subsection{Analysis of the Polyak-Ruppert averaged estimator $\bar{\theta}_{n}^{(\gamma)}$.}
\vspace{-2mm}
\label{sec:second_moment_PR_no_bias_reduction}
In this section, we analyze the finite-sample properties of the estimator $\bar{\theta}_{n}^{(\gamma)}$ from \eqref{eq:pr_averaged_est}. The analysis is based on techniques previously used in \citet{moulines2011}, as well as in the analysis of the Polyak-Ruppert averaged LSA (Linear Stochastic Approximation) algorithms, see  \citet{mou2020linear, durmus2022finite}. Below we derive the key representation for the error $\bar{\theta}_{n}^{(\gamma)} - \thetas$, following \citet{moulines2011}. Define the $k$-th step noise at the point $\theta \in \rset^{d}$ by:
\begin{equation}
\label{eq:noise_k_theta}
\noise{k}(\theta) = \nabla F(\theta,\xi_{k}) - \nabla f(\theta)\eqsp,
\end{equation}
and note that  $\noise{k+1}(\theta_{k}^{(\gamma)})$ is a martingale-difference sequence w.r.t. $(\F_k)_{k\in\nset}$. Then the recurrence \eqref{eq:sgd_recursion} takes form 
\begin{equation}
\label{eq:recurrence_error}
\theta_{k+1}^{(\gamma)} - \thetas = \theta_{k}^{(\gamma)} - \thetas - \gamma \bigl(\nabla f(\theta_{k}^{(\gamma)}) + \noise{k+1}(\theta_{k}^{(\gamma)})\bigr)\eqsp.
\end{equation}
We set 
\begin{equation}
\label{eq:eta_k_definition}
\eta(\theta) = \nabla f(\theta) - \H (\theta - \thetas)\eqsp, \text{ where } \H = \nabla^{2}f(\thetas) \in \rset^{d \times d}\eqsp.
\end{equation}
We obtain from \eqref{eq:recurrence_error} with simple algebra that  
\begin{equation}
\label{eq:summ_parts_old}
\H (\theta_{k}^{(\gamma)} - \thetas) = \gamma^{-1} (\theta_{k}^{(\gamma)} - \theta_{k+1}^{(\gamma)}) - \noise{k+1}(\theta_{k}^{(\gamma)}) - \eta(\theta_{k}^{(\gamma)})\eqsp.
\end{equation}
Taking average of \eqref{eq:summ_parts_old} for $k = n+1$ to $2n$, we arrive at the final representation:
\begin{equation}
\label{eq:summ_parts}
\H (\bar{\theta}_n^{(\gamma)} - \thetas) = \frac{\theta_{n+1}^{(\gamma)} - \thetas}{\gamma n} - \frac{\theta_{2n+1}^{(\gamma)} - \thetas}{\gamma n} - \frac{1}{n}\sum_{k=n+1}^{2n}\noise{k+1}(\theta_k^{(\gamma)}) - \frac{1}{n}\sum_{k=n+1}^{2n}\eta(\theta_k^{(\gamma)})\eqsp.
\end{equation}
We further introduce the covariance matrix of $\noise{k}(\thetas)$ measured at the optimal point $\thetas$, that is, 
\begin{equation}
\label{eq:noisecov}
\noisecov = \PE_{\xi \sim \PP_{\xi}}[\nabla F(\thetas, \xi)^{\otimes 2}] \eqsp.
\end{equation}
Note that $\noisecov$ does not depend on the step size $\gamma$ and is related to the asymptotically optimal covariance matrix of the Polyak-Ruppert averaged iterates $\bar{\theta}_{n}^{(\gamma)}$, see \citet{fort2015central}. Precisely, under assumptions \Cref{ass:mu-convex}-\Cref{ass:rand_noise}, the asymptotic covariance matrix $\Sigma_{\infty}$ from \eqref{eq:CLT_fort} is given by 
\begin{equation}
\label{eq:asympt_covariance}
\Sigma_{\infty} = \{\H\}^{-1} \noisecov \{\H\}^{-1}\eqsp.
\end{equation}
We state our subsequent results in terms of $\H (\bar{\theta}_n^{(\gamma)} - \thetas)$ and its Richardson-Romberg adjusted counterpart. One can switch to the corresponding results for $\bar{\theta}_n^{(\gamma)} - \thetas$ at the price of the factor $\norm{\{\H\}^{-1}}$, which affects the non-leading terms. In our first result below, we establish the root MSE bound on the error of the Polyak-Ruppert averaged estimator \eqref{eq:pr_averaged_est}.
\begin{theorem}
\label{th:mart_decomposition_pr_error}
Assume \Cref{ass:mu-convex}, \Cref{ass:L-smooth}, \Cref{ass:rand_noise}($4$), and \Cref{ass:stationary_moments_bounds}($4$). Then for any $\gamma \in (0, 1/(\L \Conststep{4})]$, $n \in \nset$, and initial distribution $\nu$ on $\rset^{d}$, the sequence of Polyak-Ruppert estimates \eqref{eq:pr_averaged_est} satisfies
\begin{equation}
\label{eq:th_PR_main}
\PE_{\nu}^{1/2}[\norm{\H(\bar{\theta}_n^{(\gamma)} - \thetas)}^{2}] \leq \frac{\sqrt{\trace{\noisecov}}}{\sqrt{n}} + \frac{\ConstPR{2}}{\gamma^{1/2}n} + \ConstPR{3} \gamma + \frac{\ConstPR{4} \gamma^{1/2}}{n^{1/2}} +  \Rem_{2}(n, \gamma, \norm{\theta_0-\thetas})\eqsp,
\end{equation}
where the constants $\ConstPR{2}$ to $\ConstPR{4}$ are defined in \Cref{sec:proof_th_main_no_rr} (see equation \eqref{eq:Const_2_3_4_def}), and 
\begin{multline*}
\Rem_{2}(n, \gamma, \norm{\theta_0-\thetas}) = \frac{c_0 (1-\gamma\mu)^{(n+1)/2} \L}{\gamma \mu n} \PE_{\nu}^{1/2}\bigl[\norm{\theta_0-\thetas}^2\bigr] \\
+ \frac{c_0 \L (1-\gamma\mu)^{n+1}}{2n\gamma\mu}\PE^{1/2}_{\nu}\bigl[\norm{\theta_0 -\thetas}^4\bigr] \eqsp,
\end{multline*}
where $c_0$ is an absolute (numerical) constant.
\end{theorem} 
The version of \Cref{th:mart_decomposition_pr_error} with explicit constants together with the proof is provided in \Cref{sec:proof_th_main_no_rr}, see \Cref{th:mart_decomposition_pr_error_explicit_constants}. Note that the result of \Cref{th:mart_decomposition_pr_error} is valid for arbitrary $\gamma \in (0, 1/(\L \Conststep{4})]$. At the same time, this bound can be optimized over step size of the form $\gamma = n^{-\beta}$, $\beta \in (0,1)$.

\begin{corollary}
\label{cor:mart_decomposition_pr_error}
Under the assumptions of \Cref{th:mart_decomposition_pr_error}, provided that $n \geq \left( \L \Conststep{4} \right)^{3/2}$, it holds setting $\gamma = n^{-2/3}$ that 
\begin{equation}
\label{eq:2-moment-bound-optimized}
\PE^{1/2}_{\nu}[\norm{\H (\bar{\theta}_n^{(\gamma)} - \thetas)}^{2}] \leq  \frac{\sqrt{\trace{\noisecov}}}{n^{1/2}} + \frac{\mathsf{C}(\L,\mu)}{n^{2/3}} + \Rem_{2}(n, 1 /n^{2/3}, \norm{\theta_0-\thetas})\eqsp,
\end{equation}
where the expression for $\mathsf{C}(\L,\mu)$ can be traced from \Cref{sec:proof_th_main_no_rr}, eq. \eqref{eq:Const_2_3_4_def}. 
\end{corollary} 
\Cref{cor:mart_decomposition_pr_error} implies that, if  $n$ is known in advance and  $\gamma = n^{-2/3}$, then $\bar{\theta}_{n}^{(\gamma)}$ satisfies    \eqref{eq:rate_setting0} with $\delta = 1/6$. A closer inspection of the sum \eqref{eq:summ_parts} reveals that $\PE_{\pi_{\gamma}}[\eta(\theta_k^{(\gamma)})]$ is of order $\gamma$, and we can not expect to provide a better bound for the term $\frac{1}{n}\sum_{k=n+1}^{2n}\eta(\theta_k^{(\gamma)})$ compared to the one coming from Minkowski's inequality. Thus, this is the \emph{bias} of the stationary distribution, which does not allow us to improve scaling of the second-order term w.r.t. the sample size $n$.  
\par 
In case of deterministic problems $\noise{k}(\theta) = 0$ for any $k$ and $\theta$, and \Cref{ass:stationary_moments_bounds}($p$) is satisfied for any $p \geq 2$ with $\Constlast{p} = 0 $. In such a setting, $\noisecov = 0$, and the remainder terms are proportional to $\Constlast{p}$ with $p = 2,4$, or $6$, and also vanishes. Therefore, \Cref{th:mart_decomposition_pr_error} provides exponential convergence bounds, which are embedded in the remainder term. Previous studies in \citep{moulines2011} provides the bound of the same order $\mathcal{O}(n^{-2/3})$ for the second-order term of the root-MSE bound of SGD algorithm with Polyak-Ruppert averaging. This rate is known to be suboptimal for first-order methods. The recent work by \citet{li2022root} shows that the best known second-order error term in the bound \eqref{eq:2-moment-bound-optimized} is of order $\mathcal{O}\bigl(n^{-3/4}\bigr)$ and can be achieved by the Root-SGD algorithm. In the next section we mirror this bound using the constant step-size SGD algorithm combined with the Richardson-Romberg extrapolation technique.

\vspace{-3mm}
\section{Richardson-Romberg extrapolation}
\label{sec:richardson-romberg}
\vspace{-3mm}
Our analysis presented in \Cref{th:mart_decomposition_pr_error} was based on the summation by parts formula \eqref{eq:summ_parts} and Taylor expansion of the gradient $\nabla f(\theta)$ in the neighborhood of $\thetas$, yielding the remainder quantity $\eta(\theta)$. It is important to notice that
\begin{equation}
\int_{\rset^{d}}\eta(\theta)\pi_{\gamma}(\rmd \theta) \neq 0\eqsp,
\end{equation} 
which prevents us from using larger step size $\gamma$ in the optimized bound \eqref{eq:2-moment-bound-optimized}. In this section we show that Richardson-Romberg extrapolation technique is sufficient to significantly reduce the bias associated with $\eta(\theta)$ and improve the second-order term in the MSE bound \eqref{eq:2-moment-bound-optimized}. Instead of considering a single SGD trajectory $\{\theta_k^{(\gamma)}\}_{k \in \nset}$, and then relying on the tail-averaged estimator $\prtheta_{n}^{(\gamma)}$, we construct two parallel chains based on the same sequence $\{\xi_k\}_{k \in \nset}$:
\begin{equation}
\label{eq:RR-chain-coupled}
\begin{split}
\theta_{k+1}^{(\gamma)} &= \theta_{k}^{(\gamma)} - \gamma \nabla F(\theta_{k}^{(\gamma)}, \xi_{k+1})\eqsp, \quad \prtheta_{n}^{(\gamma)} = \frac{1}{n}\sum_{k=n+1}^{2n}\theta_{k}^{(\gamma)}\eqsp, \\
\theta_{k+1}^{(2\gamma)} &= \theta_{k}^{(2\gamma)} - 2\gamma \nabla F(\theta_{k}^{(2\gamma)}, \xi_{k+1})\eqsp, \quad \prtheta_{n}^{(2\gamma)} = \frac{1}{n}\sum_{k=n+1}^{2n}\theta_{k}^{(2\gamma)}\eqsp.
\end{split}
\end{equation}
Based on $\prtheta_{n}^{(\gamma)}$ and $\prtheta_{n}^{(2\gamma)}$ defined above, we construct the Richardson-Romberg estimator:  
\begin{equation}
\label{eq:theta_RR_estimator}
\prtheta_{n}^{(RR)} := 2\prtheta_{n}^{(\gamma)} - \prtheta_{n}^{(2\gamma)}\eqsp.
\end{equation}
Note that it is possible to use different sources of randomness $\{\xi_k\}_{k \in \nset}$ and $\{\xi^{\prime}_k\}_{k \in \nset}$ when constructing the sequences $\{\theta_{k}^{(\gamma)}\}_{k \in \nset}$ and $\{\theta_{k}^{(2\gamma)}\}_{k \in \nset}$, respectively. At the same time, it is possible to show the benefits of using the same sequence of random variables $\{\xi_k\}_{k \in \nset}$ in \eqref{eq:RR-chain-coupled}. Indeed, consider the decomposition \eqref{eq:summ_parts} and further expand the term $\eta(\theta)$ defined in \eqref{eq:eta_k_definition} as 
\[
\eta(\theta) = \psi(\theta) + G(\theta)\eqsp,
\]
where we have defined the following vector-valued functions:
\begin{equation}
\label{eq:psi_theta_definition}
\psi(\theta) = \frac{1}{2} \nabla^{3}f(\theta^*)(\theta - \thetas)^{\otimes 2}\eqsp, \quad 
G(\theta) = \frac{1}{2}\left(\int_{0}^{1} t^2 \nabla^{4}f(t\thetas + (1-t)\theta)\,\rmd t\right) (\theta - \thetas)^{\otimes 3}\eqsp.
\end{equation}
We further rewrite the decomposition \eqref{eq:summ_parts} as 
\begin{multline}
\label{eq:PR-decomposition-error-extended}
\H (\prtheta_{n}^{(\gamma)} - \thetas) = \frac{\theta_{n+1}^{(\gamma)} - \thetas}{\gamma n} - \frac{\theta_{2n+1}^{(\gamma)} - \thetas}{\gamma n} - \frac{1}{n}\sum_{k=n+1}^{2n}\noise{k+1}(\thetas) \\
- \frac{1}{n}\sum_{k=n+1}^{2n}\{\noise{k+1}(\theta_{k}^{(\gamma)}) - \noise{k+1}(\thetas)\} - \frac{1}{n}\sum_{k=n+1}^{2n}\psi(\theta_{k}^{(\gamma)}) - \frac{1}{n}\sum_{k=n+1}^{2n}G (\theta_{k}^{(\gamma)})\eqsp.
\end{multline}
In the decomposition \eqref{eq:PR-decomposition-error-extended}, the linear term $W = n^{-1}\sum_{k=n+1}^{2n}\noise{k+1}(\thetas)$ does not depend upon $\gamma$. Moreover, when setting the step size $\gamma = c_0 n^{-\beta}$ with an appropriate $\beta \in (0,1)$, we can show that the moments of all other terms except for $W$ in the r.h.s. of \eqref{eq:PR-decomposition-error-extended} are small (see \Cref{th:p_moment_RR} for more details). Hence, using the same sequence $\{\xi_k\}_{k \in \nset}$ of noise variables in \eqref{eq:RR-chain-coupled} yields an estimator $\prtheta_{n}^{(RR)}$, such that its leading component of the variance still equals $W$. Hence, using the Richardson-Romberg procedure increases only the second-order (w.r.t. $n$) components of the variance. At the same time, using different random sequences $\{\xi_k\}_{k \in \nset}$ and $\{\xi^{\prime}_k\}_{k \in \nset}$ for $\prtheta_{n}^{(\gamma)}$ and $\prtheta_{n}^{(2\gamma)}$ increases the leading component of the MSE by a constant factor. Hence, it is preferable to use synchronous noise construction as in \eqref{eq:RR-chain-coupled}. \Cref{prop:bias_PR} implies the following improved bound on the bias of $\prtheta_{n}^{(RR)}$:
\begin{proposition}
\label{prop:bias_RR}
Assume \Cref{ass:mu-convex}, \Cref{ass:L-smooth}, \Cref{ass:rand_noise}($6$), and \Cref{ass:stationary_moments_bounds}($6$). Then, for any $\gamma \in (0, 1/(\L \Conststep{6})]$, and any initial distribution $\nu$ on $\rset^{d}$, it holds that  
\begin{equation}
\label{eq:bias_RR_bound}
\PE_{\nu}[\prtheta_{n}^{(RR)}] = \thetas + B_3 \gamma^{3/2} + \Rem_{3}(\theta_0-\thetas,\gamma,n)\eqsp, 
\end{equation}
where $B_3 \in \rset^{d}$ is a vector such that $\norm{B_3} \leq \ConstPR{b,1}$, where $\ConstPR{b,1}$ and is constant independent of $\gamma$ and is defined in \eqref{eq: def_ConstPRb1,2} and 
\begin{equation*}
\textstyle
\norm{\Rem_{3}(\theta_0-\thetas,\gamma,n)} \lesssim \frac{\rme^{-\gamma\mu(n+1)/2}}{n\gamma\mu}\parenthese{\PE_{\nu}^{1/2}\bigl[\norm{\theta_0 -\thetas}^2\bigr] + \frac{\sqrt{\gamma}\tau_2}{\sqrt{\mu}}}\eqsp.
\end{equation*}
\end{proposition}
The proof of \Cref{prop:bias_RR} is provided in \Cref{sec:proof_prop_bias_PR_and_RR}. This result is a simple consequence of \Cref{prop:bias_RR}, since the linear in $\gamma$ component of the bias $\gamma \Delta_1$ from \eqref{eq:bias_PR_bound} cancels out when computing $\prtheta_{n}^{(RR)}$. We are now ready to formulate the main result for the Richardson-Romberg estimate $\prtheta_{n}^{(RR)}$. 

\begin{theorem}
\label{th:RR_second_moment}
Assume \Cref{ass:mu-convex}, \Cref{ass:L-smooth}, \Cref{ass:rand_noise}($6$), and \Cref{ass:stationary_moments_bounds}($6$). Then for any $\gamma \in (0, 1/(\L \Conststep{6}) \wedge 2/(11 \L)]$, initial distribution $\nu$ and $n \in \nset$, the Richardson-Romberg estimator $\prtheta_{n}^{(RR)}$ defined in \eqref{eq:theta_RR_estimator} satisfies
\begin{align*}
\PE_{\nu}^{1/2}[\norm{\H (\prtheta_{n}^{(RR)} -\thetas)}^2] & \leq  \frac{\sqrt{\trace{\noisecov}}}{n^{1/2}} + \frac{\ConstRR{1} \gamma^{1/2}}{n^{1/2}} + \frac{\ConstRR{2}}{\gamma^{1/2}n} + \ConstRR{3} \gamma^{3/2} + \frac{\ConstRR{4} \gamma }{n^{1/2}} \\
&\qquad \qquad + \Rem_4(n, \gamma, \norm{\theta_0-\thetas})\eqsp,
\end{align*}
where the constants $\ConstRR{1}$ to $\ConstRR{4}$ are defined in \Cref{th:RR_second_moment_proof} (equation \eqref{eq:const_rr_1_4_def}), and
\begin{multline*}
\Rem_4(n, \gamma, \norm{\theta_0-\thetas}) = \frac{c_0 \L (1-\gamma\mu)^{(n+1)/2}}{n\gamma \mu} \\ 
\times \left(\PE_{\nu}^{1/2}[\norm{\theta_0-\thetas}^6] + \PE_{\nu}^{1/2}[\norm{\theta_0-\thetas}^4] + \PE_{\nu}^{1/2}[\norm{\theta_0-\thetas}^2] + \frac{\Constlast{4} \gamma\tau_4^2}{\mu}\right)\eqsp,
\end{multline*}
with $c_0$ being an absolute constant.
\end{theorem}
Proof of \Cref{th:RR_second_moment} is provided in \Cref{th:RR_second_moment_proof}. Similarly to \Cref{th:mart_decomposition_pr_error}, we can optimize the above bound setting $\gamma$ depending upon $n$.
\begin{corollary}  
\label{cor:RR_second_moment}
Under the assumptions of \Cref{th:RR_second_moment}, provided that $n \geq \L^2 (\Conststep{6} \vee 11/2)^2$, it holds setting $\gamma = n^{-1/2}$ that 
\begin{equation}
\label{eq:2-moment-bound-optimized-rr}
\PE^{1/2}_{\nu}[\norm{\H (\prtheta_{n}^{(RR)} - \thetas)}^{2}] \leq \frac{\sqrt{\trace{\noisecov}}}{n^{1/2}} + \frac{\mathsf{C}(\L,\mu)}{n^{3/4}} + \Rem_{4}(n, 1/\sqrt{n}, \norm{\theta_0-\thetas})\eqsp,
\end{equation}
where the expression for $\mathsf{C}(\L,\mu)$ can be traced from \Cref{th:RR_second_moment_proof}, eq. \eqref{eq:const_rr_1_4_def}. 
\end{corollary} 
\paragraph{Discussion.} Note that the result of \Cref{cor:RR_second_moment} is a counterpart of \eqref{eq:rate_setting0} with $\delta = 1/4$. This decay rate of the second order term is the same as for the Root-SGD algorithm of \citet{li2022root}. At the same time, we highlight that the assumptions of \Cref{th:RR_second_moment} are stronger compared to the ones imposed by \cite{li2022root}. In particular, in \Cref{ass:L-smooth} we require that $f$ is $4$ times continuously differentiable and uniformly bounded. At the same time, \cite{li2022root} impose Lipschitz continuity of the Hessian of $f$. Our proof of \Cref{th:RR_second_moment} essentially relies on the $4$-th order Taylor expansion, and it is not clear, if this assumption can be relaxed. We leave further investigations of this question for future research.

Now we generalize the previous result for the $p$-th moment bounds with $p \geq 2$. The key technical element of our proof for the $p$-th moment bound is the following statement, which can be viewed as a version of Rosenthal's inequality \citep{Rosenthal1970, pinelis_1994}.

\begin{proposition}
\label{prop:psi_p_moment_bound} 
Let $p \geq 2$ and assume \Cref{ass:mu-convex}, \Cref{ass:L-smooth}, \Cref{ass:rand_noise}($2p$), and \Cref{ass:stationary_moments_bounds}($2p$). Then for $\psi$ defined in \eqref{eq:psi_theta_definition} and any $\gamma \in (0, 1/(\L \Conststep{2p})]$, it holds that 
\begin{equation}
\label{eq:rosenthal_type_bound}
\PE^{1/p}_{\pi_{\gamma}}\bigl[\norm{\sum_{k=0}^{n-1}\{\psi(\theta_{k}^{(\gamma)}) - \pi_{\gamma}(\psi)\}}^p\bigr]\lesssim    \frac{\L \Constlast{2p} p  \tau^2_{2p} \sqrt{n \gamma}}{\mu^{3/2}} + \frac{\L \Constlast{2p} \tau_{2p}}{\mu^2}\eqsp.
\end{equation}
\end{proposition}
\paragraph{Discussion.} Proof of \Cref{prop:psi_p_moment_bound} is provided in \Cref{sec:psi_p_moment_bound_proof}. It is important to acknowledge that there are numerous Rosenthal-type inequalities for dependent sequences in the literature. \Cref{prop:psi_p_moment_bound} can be viewed as an analogue of the classical Rosenthal inequality for strongly mixing sequences, see \citep[Theorem~6.3]{rio2017asymptotic}. However, it should be emphasized that the Markov chain $\{\theta_{k}^{(\gamma)}\}_{k \in \nset}$ is geometrically ergodic under the assumptions \Cref{ass:mu-convex}-\Cref{ass:rand_noise}($p$) only in sense of the weighted Wasserstein semi-metric $\Wass{c}(\xi,\xi')$ with cost function $c$ defined in \eqref{eq:function_c_def}. As a result, the sequence $\{\theta_{k}^{(\gamma)}\}_{k \in \nset}$ does not necessarily satisfy strong mixing conditions. Bounds similar to \eqref{eq:rosenthal_type_bound}  have been explored in \citep{durmus2023rosenthal}, but in \Cref{prop:psi_p_moment_bound} we obtain tighter dependence of the right-hand side upon $\gamma$. Below we provide the $p$-th moment bound together with corollary for the step size $\gamma$ optimized w.r.t. $n$.

\begin{theorem}
\label{th:p_moment_RR}
Let $p \geq 2$ and assume \Cref{ass:mu-convex}, \Cref{ass:L-smooth}, \Cref{ass:rand_noise}($3p$), and \Cref{ass:stationary_moments_bounds}($3p$). Then for any step size $\gamma \in (0, 1/(\L \Conststep{3p}) \wedge p/(4 \cdot 3^p \L)]$, initial distribution $\nu$, and $n \in \nset$, the estimator $\prtheta_{n}^{(RR)}$ defined in \eqref{eq:theta_RR_estimator} satisfies
\begin{equation}
\label{eq:rosenthal_type_p_moment}
\begin{split}
\PE_{\nu}^{1/p}[\norm{\H(\prtheta_{n}^{(RR)} -\thetas)}^p] 
&\leq \frac{c_{1 }\sqrt{\trace{\noisecov}}p^{1/2}}{n^{1/2}} + \frac{\ConstRR{5}}{n \gamma^{1/2}} + \frac{\ConstRR{6} \gamma^{1/2}}{n^{1/2}} + \ConstRR{7} \gamma^{3/2} \\ 
&\qquad \qquad + \frac{c_{2} p \tau_p}{n^{1-1/p}} + \frac{\ConstRR{8}}{n} + \Rem_5(n, \gamma, \norm{\theta_0-\thetas})\eqsp, 
\end{split}
\end{equation}
where $c_1 = 60\rme$ and $c_2 = 60$ are absolute constants from the Pinelis version of Rosenthal inequality \citep[Theorem~4.1]{pinelis_1994}, and problem-specific constants $\ConstRR{5}$ to $\ConstRR{8}$ are defined in \Cref{th:RR_pth_moment_proof} (equation \eqref{eq:const_rr_5_8_def}), and  
\[
\textstyle 
\Rem_5(n, \gamma, \norm{\theta_0-\thetas}) = (1 - \gamma \mu)^{(n+1)/2} C_{f,p}\bigl(\PE_{\nu}^{1/p}\bigl[\norm{\theta_0 - \thetas}^{p}\bigr] + \PE_{\nu}^{1/p}\bigl[\norm{\theta_0 - \thetas}^{2p}\bigr] + \PE_{\nu}^{1/p}\bigl[\norm{\theta_0 - \thetas}^{3p}\bigr]\bigr)\eqsp.
\]
Here constant $C_{f,p}$ can be traced from \Cref{th:RR_pth_moment_proof}, eq. \eqref{eq:exponential_small_term_RR_pth_moment}.
\end{theorem}
\begin{corollary}  
\label{cor:p_moment_RR}
Under the assumptions of \Cref{th:p_moment_RR}, provided that $n \geq \L^2 \left(\Conststep{3p} \vee 4\cdot 3^{p}/p\right)^{2}$, it holds setting $\gamma = n^{-1/2}$ that   
\begin{equation}
\label{eq:p-moment-bound-optimized-rr}
\PE^{1/p}_{\nu}[\norm{\H (\prtheta_{n}^{(RR)} - \thetas)}^{p}] \leq \frac{c_1 \sqrt{\trace{\noisecov}}p^{1/2}}{n^{1/2}} + \frac{\mathsf{C}(\L,\mu,p)}{n^{3/4}} + \Rem_{5}(n, 1/\sqrt{n}, \norm{\theta_0-\thetas})\eqsp,
\end{equation}
where the expression for $\mathsf{C}(\L,\mu,p)$ can be traced from \Cref{th:RR_pth_moment_proof}, eq. \eqref{eq:const_rr_5_8_def}. 
\end{corollary} 

\paragraph{Discussion.} Proof of \Cref{th:p_moment_RR} is provided in \Cref{th:RR_pth_moment_proof}. Note that the result above is a direct generalization of \Cref{th:RR_second_moment}, which reveals the same scaling of the step size $\gamma$ with respect to $n$. To the best of our knowledge, this is the first analysis of a first-order method, which provides a $p$-th moment bound with $p > 2$ and the second-order term of order $\mathcal{O}\bigl(n^{-3/4}\bigr)$ while keeping the precise leading term related to the asymptotically optimal covariance matrix $\noisecov$. Such results were previously obtained for the setting of linear stochastic approximation (LSA), see \citet{mou2020linear, durmus2022finite}. Thus, Richardson-Romberg extrapolation applied to the strongly convex minimization problems allows to mimic the $p$-th moment error bounds that were previously obtained in the LSA setting.

\vspace{-3mm}
\section{Numerical results}
\label{sec:numerics}
\vspace{-3mm}
In this section we numerically illustrate the scale of the second-order terms in equation \eqref{eq:2-moment-bound-optimized-rr} in \Cref{cor:RR_second_moment}. We show that, for a particular minimization problem, setting $\gamma = n^{-1/2}$, we achieve the scaling of the second-order terms in root-MSE bounds of order $\mathcal{O}(n^{-3/4})$. We consider the problem 
\[
\min_{\theta \in \rset} f(\theta)\eqsp, \quad f(\theta) = \theta^2 + \cos \theta\eqsp, 
\]
with the stochastic gradient oracles $\nabla F(\theta,\xi)$ given by 
$\nabla F(\theta,\xi) = 2\theta - \sin \theta + \xi$, and $\xi \sim \mathcal{N}(0,1)$. This example satisfies the assumptions \Cref{ass:mu-convex}, \Cref{ass:L-smooth}, \Cref{ass:rand_noise}($p$) with any $p \geq 2$. We select different sample sizes $n$, choose $ \gamma = 1/\sqrt{n}$, and construct the associated estimates $\bar{\theta}_{n}^{(\gamma)}$ and $\bar{\theta}_{n}^{(2\gamma)}$. Detailed description of the experimental setting is provided in \Cref{sec:exp_details}. Then for each $n$ we compute the Richardson-Romberg estimates $\prtheta_{n}^{(RR)}$ from \eqref{eq:RR-chain-coupled} alongside with its versions without the leading term in $n$, i.e. $\prtheta_{n}^{(RR)} + n^{-1}\sum_{k=n+1}^{2n}\noise{k+1}(\thetas)$. We provide first the plot for $\norm{\prtheta_{n}^{(RR)} - \thetas}^2$ and $\norm{\prtheta_{n}^{(RR)} + n^{-1}\sum_{k=n+1}^{2n}\noise{k+1}(\thetas) - \thetas}^2$, averaged over $M=320$ parallel runs, in \Cref{fig:results-expe-1d}. On the same figure we also provide the plots for rescaled errors
\[
n \norm{\prtheta_{n}^{(RR)} - \thetas}^2 \text{ and } n^{3/2} \norm{\prtheta_{n}^{(RR)} - \thetas + n^{-1}\sum\nolimits_{k=n+1}^{2n}\noise{k+1}(\thetas)}^2\,,
\]
also averaged over $M$ parallel runs. The corresponding plot indicates that the proper scaling of the squared norm of the remainder part is $n^{-3/2}$, that is, the corresponding term in root-MSE bound for $\PE^{1/2}_{\nu}[\norm{\prtheta_{n}^{(RR)} - \thetas}^2]$ scales as $\mathcal{O}(n^{-3/4})$, as predicted by \Cref{cor:RR_second_moment}. 

\begin{figure}[t!]
\centering
\includegraphics[width=0.49\linewidth]{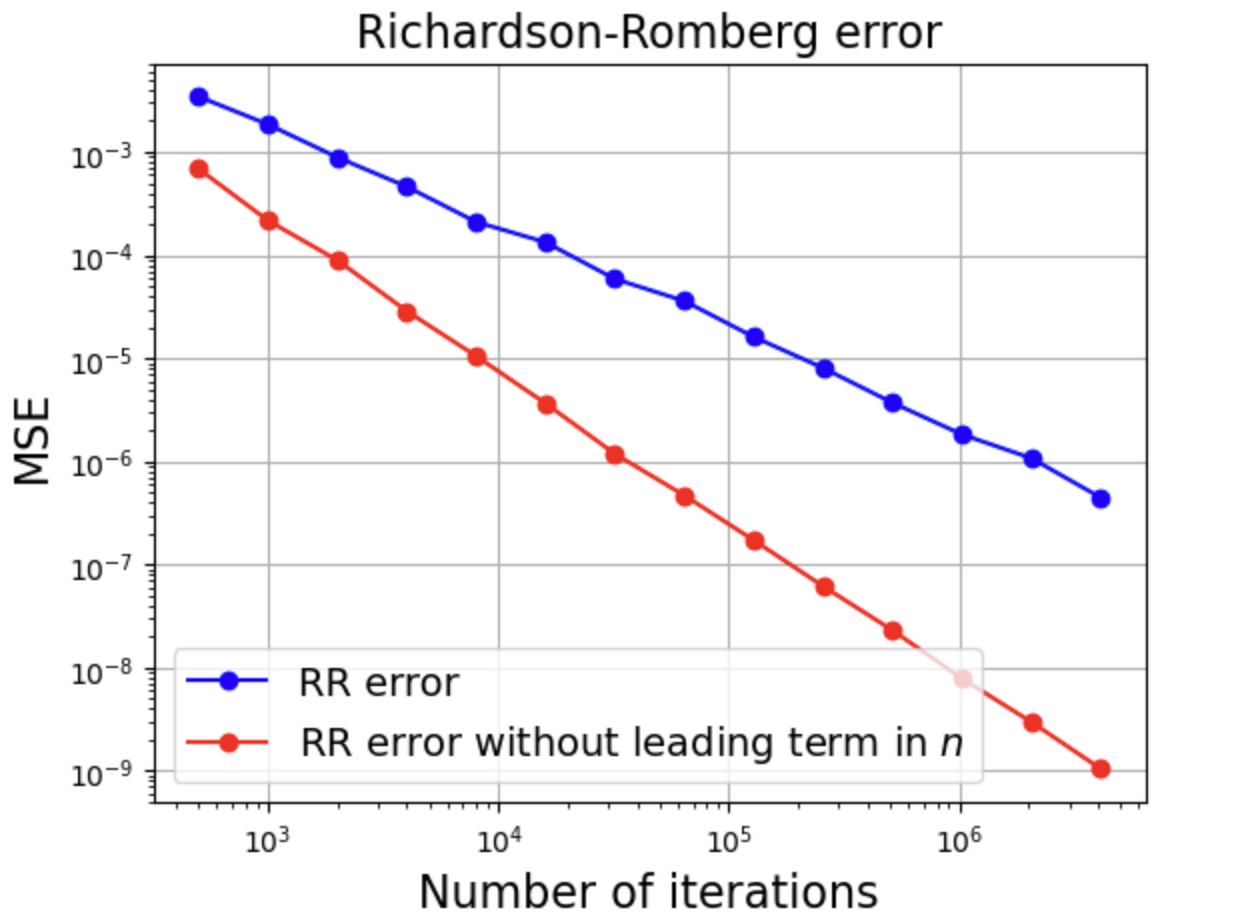}
\includegraphics[width=0.49\linewidth]{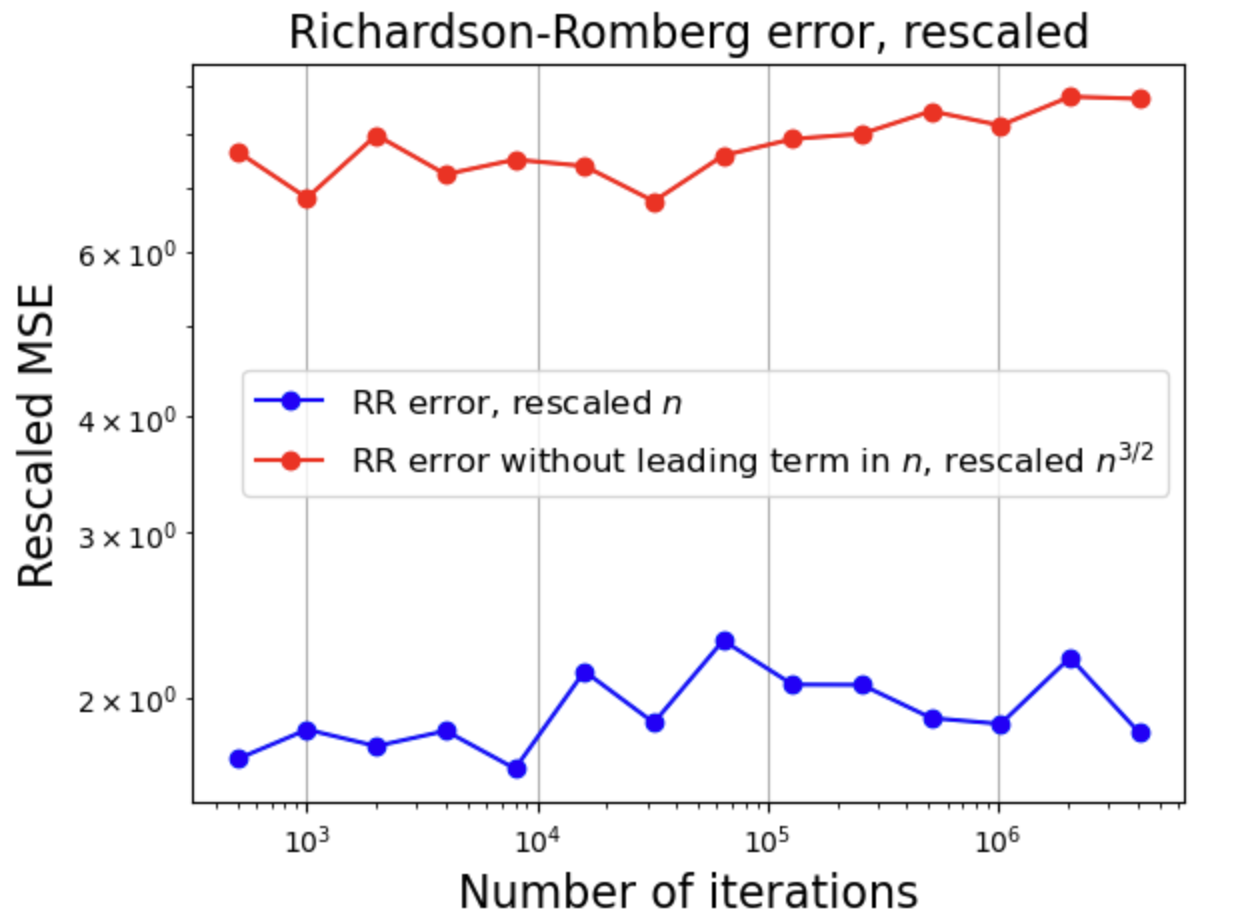}
\caption{Left picture: Richardson-Romberg experimental error with and without the leading term $\frac{1}{n}\sum_{k=n+1}^{2n}\noise{k+1}(\thetas)$. Right picture: same errors after rescaling by $n$ and $n^{3/2}$, respectively.}
\label{fig:results-expe-1d}
\vspace{-3mm}
\end{figure}

\vspace{-3mm}
\section{Conclusion}
\label{sec:conclusion}
In this paper, we study the non-asymptotic error bounds for the Richardson-Romberg estimator built upon the Polyak-Ruppert averaged SGD iterates with a constant step size. In particular, under an appropriate choice of step size, depending on the total number of iterations $n$, the corresponding root-MSE bound admits both a sharp leading term, which aligns with the minimax-optimal covariance matrix, and a second-order term of order $\mathcal{O}(n^{-3/4})$, which is the best known rate among first-order methods. Future research directions include, firstly, generalizing the proposed algorithm to the setting of dependent noise sequences $\{\xi_k\}_{k \in \nset}$ in the stochastic gradients \eqref{eq:stoch_minimization}. Another natural question is to study the properties of $\prtheta_{n}^{(RR)}$ under relaxed assumptions on $f$. In particular, it would be interesting to remove additional smoothness assumptions on $f$ (bounded $3$-rd and $4$-th derivatives), and to relax the strong convexity condition \Cref{ass:mu-convex}. One more research direction is to quantify a relation between the parameter $\delta$ in \eqref{eq:rate_setting0} and convergence rates in \eqref{eq:CLT_fort}, following the approach suggested in  \citet{shao2022berry}.

\newpage 
\section*{Acknowledgement}
The work of M. Sheshukova, D. Belomestny, A. Naumov, and S. Samsonov was prepared within the framework of the HSE University Basic Research Program.  The work of E. Moulines has been partly funded by the European Union (ERC-2022-SYG-OCEAN-101071601). Views and opinions expressed are however those of the author(s) only and do not necessarily reflect those of the European Union or the European Research Council Executive Agency. Neither the European Union nor the granting authority can be held responsible for them. 

\bibliography{bibfile}

\begin{thebibliography}{41}
\providecommand{\natexlab}[1]{#1}
\providecommand{\url}[1]{\texttt{#1}}
\expandafter\ifx\csname urlstyle\endcsname\relax
  \providecommand{\doi}[1]{doi: #1}\else
  \providecommand{\doi}{doi: \begingroup \urlstyle{rm}\Url}\fi

\bibitem[Allmeier \& Gast(2024)Allmeier and Gast]{allmeier2024computing}
Sebastian Allmeier and Nicolas Gast.
\newblock Computing the {B}ias of {C}onstant-step {S}tochastic {A}pproximation
  with {M}arkovian {N}oise.
\newblock In A.~Globerson, L.~Mackey, D.~Belgrave, A.~Fan, U.~Paquet,
  J.~Tomczak, and C.~Zhang (eds.), \emph{Advances in Neural Information
  Processing Systems}, volume~37, pp.\  137873--137902. Curran Associates,
  Inc., 2024.
\newblock URL
  \url{https://proceedings.neurips.cc/paper_files/paper/2024/file/f949c1f490beb42124a267b7476cd353-Paper-Conference.pdf}.

\bibitem[Bally \& Talay(1996)Bally and Talay]{bally1996richardsonnonsmooth}
V.~Bally and D.~Talay.
\newblock The law of the euler scheme for stochastic differential equations.
\newblock \emph{Probability Theory and Related Fields}, 104\penalty0
  (1):\penalty0 43--60, 1996.
\newblock \doi{10.1007/BF01303802}.
\newblock URL \url{https://doi.org/10.1007/BF01303802}.

\bibitem[Defazio et~al.(2014)Defazio, Bach, and
  Lacoste-Julien]{defazio2014saga}
Aaron Defazio, Francis Bach, and Simon Lacoste-Julien.
\newblock {SAGA}: A fast incremental gradient method with support for
  non-strongly convex composite objectives.
\newblock \emph{Advances in neural information processing systems}, 27, 2014.

\bibitem[Dieuleveut et~al.(2020)Dieuleveut, Durmus, and
  Bach]{durmus2020biassgd}
Aymeric Dieuleveut, Alain Durmus, and Francis Bach.
\newblock {Bridging the gap between constant step size stochastic gradient
  descent and Markov chains}.
\newblock \emph{The Annals of Statistics}, 48\penalty0 (3):\penalty0 1348 --
  1382, 2020.
\newblock \doi{10.1214/19-AOS1850}.
\newblock URL \url{https://doi.org/10.1214/19-AOS1850}.

\bibitem[Douc et~al.(2018)Douc, Moulines, Priouret, and
  Soulier]{douc:moulines:priouret:soulier:2018}
R.~Douc, E.~Moulines, P.~Priouret, and P.~Soulier.
\newblock \emph{{M}arkov chains}.
\newblock Springer Series in Operations Research and Financial Engineering.
  Springer, 2018.
\newblock ISBN 978-3-319-97703-4.

\bibitem[Durmus et~al.(2016)Durmus, Simsekli, Moulines, Badeau, and
  Richard]{durmus2016stochastic}
Alain Durmus, Umut Simsekli, Eric Moulines, Roland Badeau, and Ga{\"e}l
  Richard.
\newblock Stochastic {G}radient {R}ichardson-{R}omberg {M}arkov {C}hain {M}onte
  {C}arlo.
\newblock \emph{Advances in neural information processing systems}, 29, 2016.

\bibitem[Durmus et~al.(2023)Durmus, Moulines, Naumov, Samsonov, and
  Sheshukova]{durmus2023rosenthal}
Alain Durmus, Eric Moulines, Alexey Naumov, Sergey Samsonov, and Marina
  Sheshukova.
\newblock Rosenthal-type inequalities for linear statistics of {M}arkov chains.
\newblock \emph{arXiv preprint arXiv:2303.05838}, 2023.

\bibitem[Durmus et~al.(2025)Durmus, Moulines, Naumov, and
  Samsonov]{durmus2022finite}
Alain Durmus, Eric Moulines, Alexey Naumov, and Sergey Samsonov.
\newblock Finite-time high-probability bounds for polyak--ruppert averaged
  iterates of linear stochastic approximation.
\newblock \emph{Mathematics of Operations Research}, 50\penalty0 (2):\penalty0
  935--964, 2025.

\bibitem[Fort(2015)]{fort2015central}
Gersende Fort.
\newblock Central limit theorems for stochastic approximation with controlled
  markov chain dynamics.
\newblock \emph{ESAIM: Probability and Statistics}, 19:\penalty0 60--80, 2015.

\bibitem[Gadat \& Panloup(2023)Gadat and Panloup]{gadat2017optimal}
S{\'e}bastien Gadat and Fabien Panloup.
\newblock Optimal non-asymptotic analysis of the {R}uppert--{P}olyak averaging
  stochastic algorithm.
\newblock \emph{Stochastic Processes and their Applications}, 156:\penalty0
  312--348, 2023.
\newblock ISSN 0304-4149.
\newblock \doi{https://doi.org/10.1016/j.spa.2022.11.012}.
\newblock URL
  \url{https://www.sciencedirect.com/science/article/pii/S0304414922002447}.

\bibitem[Goodfellow et~al.(2014)Goodfellow, Pouget-Abadie, Mirza, Xu,
  Warde-Farley, Ozair, Courville, and Bengio]{goodfellow2014generative}
Ian Goodfellow, Jean Pouget-Abadie, Mehdi Mirza, Bing Xu, David Warde-Farley,
  Sherjil Ozair, Aaron Courville, and Yoshua Bengio.
\newblock Generative adversarial nets.
\newblock In \emph{Advances in Neural Information Processing Systems (NIPS)},
  2014.

\bibitem[Goodfellow et~al.(2016)Goodfellow, Bengio, and
  Courville]{GoodBengCour16}
Ian~J. Goodfellow, Yoshua Bengio, and Aaron Courville.
\newblock \emph{Deep Learning}.
\newblock MIT Press, Cambridge, MA, USA, 2016.
\newblock \url{http://www.deeplearningbook.org}.

\bibitem[Hildebrand(1987)]{hildebrand1987introduction}
Francis~Begnaud Hildebrand.
\newblock \emph{Introduction to numerical analysis}.
\newblock Courier Corporation, 1987.

\bibitem[Huo et~al.(2023)Huo, Chen, and Xie]{huo2023bias}
Dongyan Huo, Yudong Chen, and Qiaomin Xie.
\newblock Bias and extrapolation in {M}arkovian linear stochastic approximation
  with constant stepsizes.
\newblock In \emph{Abstract Proceedings of the 2023 ACM SIGMETRICS
  International Conference on Measurement and Modeling of Computer Systems},
  pp.\  81--82, 2023.

\bibitem[Huo et~al.(2024{\natexlab{a}})Huo, Chen, and
  Xie]{huo2024effectiveness}
Dongyan~Lucy Huo, Yudong Chen, and Qiaomin Xie.
\newblock Effectiveness of constant stepsize in markovian lsa and statistical
  inference.
\newblock In \emph{Proceedings of the AAAI Conference on Artificial
  Intelligence}, volume~38, pp.\  20447--20455, 2024{\natexlab{a}}.

\bibitem[Huo et~al.(2024{\natexlab{b}})Huo, Zhang, Chen, and
  Xie]{huo2024collusion}
Dongyan~Lucy Huo, Yixuan Zhang, Yudong Chen, and Qiaomin Xie.
\newblock The collusion of memory and nonlinearity in stochastic approximation
  with constant stepsize.
\newblock In A.~Globerson, L.~Mackey, D.~Belgrave, A.~Fan, U.~Paquet,
  J.~Tomczak, and C.~Zhang (eds.), \emph{Advances in Neural Information
  Processing Systems}, volume~37, pp.\  21699--21762. Curran Associates, Inc.,
  2024{\natexlab{b}}.

\bibitem[Kwon et~al.(2024)Kwon, Dotson, Chen, and Xie]{kwon2024two}
Jeongyeol Kwon, Luke Dotson, Yudong Chen, and Qiaomin Xie.
\newblock {T}wo-{T}imescale {L}inear {S}tochastic {A}pproximation: {C}onstant
  {S}tepsizes {G}o a {L}ong {W}ay.
\newblock \emph{arXiv preprint arXiv:2410.13067}, 2024.

\bibitem[Li et~al.(2022)Li, Mou, Wainwright, and Jordan]{li2022root}
Chris~Junchi Li, Wenlong Mou, Martin Wainwright, and Michael Jordan.
\newblock Root-sgd: Sharp nonasymptotics and asymptotic efficiency in a single
  algorithm.
\newblock In Po-Ling Loh and Maxim Raginsky (eds.), \emph{Proceedings of Thirty
  Fifth Conference on Learning Theory}, volume 178 of \emph{Proceedings of
  Machine Learning Research}, pp.\  909--981. PMLR, 02--05 Jul 2022.
\newblock URL \url{https://proceedings.mlr.press/v178/li22a.html}.

\bibitem[Merad \& Ga{\"\i}ffas(2023)Merad and
  Ga{\"\i}ffas]{merad2023convergence}
Ibrahim Merad and St{\'e}phane Ga{\"\i}ffas.
\newblock Convergence and concentration properties of constant step-size sgd
  through markov chains.
\newblock \emph{arXiv preprint arXiv:2306.11497}, 2023.

\bibitem[Mnih et~al.(2015)Mnih, Kavukcuoglu, Silver, Rusu, Veness, Bellemare,
  Graves, Riedmiller, Fidjeland, Ostrovski, et~al.]{deeprl}
Volodymyr Mnih, Koray Kavukcuoglu, David Silver, Andrei~A Rusu, Joel Veness,
  Marc~G Bellemare, Alex Graves, Martin Riedmiller, Andreas~K Fidjeland, Georg
  Ostrovski, et~al.
\newblock Human-level control through deep reinforcement learning.
\newblock \emph{nature}, 518\penalty0 (7540):\penalty0 529--533, 2015.

\bibitem[Mou et~al.(2020)Mou, Li, Wainwright, Bartlett, and
  Jordan]{mou2020linear}
Wenlong Mou, Chris~Junchi Li, Martin~J Wainwright, Peter~L Bartlett, and
  Michael~I Jordan.
\newblock On linear stochastic approximation: Fine-grained {P}olyak-{R}uppert
  and non-asymptotic concentration.
\newblock In \emph{Conference on Learning Theory}, pp.\  2947--2997. PMLR,
  2020.

\bibitem[Moulines \& Bach(2011)Moulines and Bach]{moulines2011}
Eric Moulines and Francis Bach.
\newblock Non-asymptotic analysis of stochastic approximation algorithms for
  machine learning.
\newblock \emph{Advances in neural information processing systems}, 24, 2011.

\bibitem[Nguyen et~al.(2017)Nguyen, Liu, Scheinberg, and
  Tak{\'a}{\v{c}}]{nguyen2017sarah}
Lam~M Nguyen, Jie Liu, Katya Scheinberg, and Martin Tak{\'a}{\v{c}}.
\newblock Sarah: A novel method for machine learning problems using stochastic
  recursive gradient.
\newblock In \emph{International Conference on Machine Learning}, pp.\
  2613--2621. PMLR, 2017.

\bibitem[Osekowski(2012)]{osekowski:2012}
A.~Osekowski.
\newblock \emph{Sharp Martingale and Semimartingale Inequalities}.
\newblock Monografie Matematyczne 72. Birkhäuser Basel, 1 edition, 2012.
\newblock ISBN 3034803699,9783034803694.

\bibitem[Pagès(2007)]{pages2007}
Gilles Pagès.
\newblock Multi-step {R}ichardson-{R}omberg {E}xtrapolation: {R}emarks on
  {V}ariance {C}ontrol and {C}omplexity.
\newblock \emph{Monte Carlo Methods and Applications}, 13\penalty0
  (1):\penalty0 37--70, 2007.
\newblock \doi{doi:10.1515/MCMA.2007.003}.
\newblock URL \url{https://doi.org/10.1515/MCMA.2007.003}.

\bibitem[Pinelis(1994)]{pinelis_1994}
Iosif Pinelis.
\newblock {Optimum Bounds for the Distributions of Martingales in Banach
  Spaces}.
\newblock \emph{The Annals of Probability}, 22\penalty0 (4):\penalty0 1679 --
  1706, 1994.
\newblock \doi{10.1214/aop/1176988477}.
\newblock URL \url{https://doi.org/10.1214/aop/1176988477}.

\bibitem[Polyak \& Juditsky(1992)Polyak and Juditsky]{polyak1992acceleration}
Boris~T Polyak and Anatoli~B Juditsky.
\newblock Acceleration of stochastic approximation by averaging.
\newblock \emph{SIAM journal on control and optimization}, 30\penalty0
  (4):\penalty0 838--855, 1992.

\bibitem[Qian(1999)]{qian1999momentum}
Ning Qian.
\newblock On the momentum term in gradient descent learning algorithms.
\newblock \emph{Neural networks}, 12\penalty0 (1):\penalty0 145--151, 1999.

\bibitem[Rio(2017)]{rio2017asymptotic}
Emmanuel Rio.
\newblock \emph{Asymptotic Theory of Weakly Dependent Random Processes},
  volume~80.
\newblock 2017.
\newblock ISBN 978-3-662-54322-1.
\newblock \doi{10.1007/978-3-662-54323-8}.

\bibitem[Robbins \& Monro(1951)Robbins and Monro]{robbins1951stochastic}
Herbert Robbins and Sutton Monro.
\newblock A stochastic approximation method.
\newblock \emph{The annals of mathematical statistics}, pp.\  400--407, 1951.

\bibitem[Rosenthal(1970)]{Rosenthal1970}
Haskell~P. Rosenthal.
\newblock On the subspaces of {$L^{p}$} {$(p>2)$} spanned by sequences of
  independent random variables.
\newblock \emph{Israel J. Math.}, 8:\penalty0 273--303, 1970.
\newblock ISSN 0021-2172.
\newblock \doi{10.1007/BF02771562}.
\newblock URL \url{https://doi.org/10.1007/BF02771562}.

\bibitem[Ruppert(1988)]{ruppert1988efficient}
David Ruppert.
\newblock Efficient estimations from a slowly convergent {R}obbins-{M}onro
  process.
\newblock Technical report, Cornell University Operations Research and
  Industrial Engineering, 1988.

\bibitem[Schulman et~al.(2015)Schulman, Levine, Abbeel, Jordan, and
  Moritz]{schulman2015trust}
John Schulman, Sergey Levine, Pieter Abbeel, Michael Jordan, and Philipp
  Moritz.
\newblock Trust region policy optimization.
\newblock In \emph{International conference on machine learning}, pp.\
  1889--1897. PMLR, 2015.

\bibitem[Shao \& Zhang(2022)Shao and Zhang]{shao2022berry}
Qi-Man Shao and Zhuo-Song Zhang.
\newblock Berry--{E}sseen bounds for multivariate nonlinear statistics with
  applications to {M}-estimators and stochastic gradient descent algorithms.
\newblock \emph{Bernoulli}, 28\penalty0 (3):\penalty0 1548--1576, 2022.

\bibitem[Sutton \& Barto(2018)Sutton and Barto]{Sutton1998}
Richard~S. Sutton and Andrew~G. Barto.
\newblock \emph{Reinforcement Learning: An Introduction}.
\newblock The MIT Press, second edition, 2018.
\newblock URL \url{http://incompleteideas.net/book/the-book-2nd.html}.

\bibitem[Talay \& Tubaro(1990)Talay and Tubaro]{talaytubaro90}
Denis Talay and Luciano Tubaro.
\newblock Expansion of the global error for numerical schemes solving
  stochastic differential equations.
\newblock \emph{Stochastic Analysis and Applications}, 8\penalty0 (4):\penalty0
  483--509, 1990.
\newblock \doi{10.1080/07362999008809220}.
\newblock URL \url{https://doi.org/10.1080/07362999008809220}.

\bibitem[Van~der Vaart(2000)]{van2000asymptotic}
Aad~W Van~der Vaart.
\newblock \emph{Asymptotic statistics}, volume~3.
\newblock Cambridge university press, 2000.

\bibitem[Vlatakis-Gkaragkounis et~al.(2024)Vlatakis-Gkaragkounis, Giannou,
  Chen, and Xie]{vlatakis2023stochastic}
Emmanouil~Vasileios Vlatakis-Gkaragkounis, Angeliki Giannou, Yudong Chen, and
  Qiaomin Xie.
\newblock Stochastic methods in variational inequalities: Ergodicity, bias and
  refinements.
\newblock In \emph{International Conference on Artificial Intelligence and
  Statistics}, pp.\  4123--4131. PMLR, 2024.

\bibitem[Yu et~al.(2021)Yu, Balasubramanian, Volgushev, and
  Erdogdu]{yu2021analysis}
Lu~Yu, Krishnakumar Balasubramanian, Stanislav Volgushev, and Murat~A Erdogdu.
\newblock An analysis of constant step size {SGD} in the non-convex regime:
  {A}symptotic normality and bias.
\newblock \emph{Advances in Neural Information Processing Systems},
  34:\penalty0 4234--4248, 2021.

\bibitem[Zhang \& Xie(2024)Zhang and Xie]{zhang2024constant}
Yixuan Zhang and Qiaomin Xie.
\newblock Constant stepsize q-learning: Distributional convergence, bias and
  extrapolation.
\newblock \emph{arXiv preprint arXiv:2401.13884}, 2024.

\bibitem[Zhang et~al.(2024)Zhang, Huo, Chen, and Xie]{zhang2024prelimit}
Yixuan Zhang, Dongyan Huo, Yudong Chen, and Qiaomin Xie.
\newblock Prelimit coupling and steady-state convergence of constant-stepsize
  nonsmooth contractive {SA}.
\newblock \emph{arXiv preprint arXiv:2303.05838}, 2024.
\newblock URL \url{https://arxiv.org/abs/2404.06023}.

\end{thebibliography}
\bibliographystyle{iclr2025_conference}
\newpage

\appendix 
\section{Proof of \Cref{lem:Wasserstein_ergodicity},  \Cref{prop:bias_PR}, and \Cref{prop:bias_RR}}
\label{sec:proof_prop_bias_PR_and_RR}
Throughout this appendix we use $c_0$ for an absolute constant, which values may vary from line to line. In addition, when the upper index of $\theta_k^{(\gamma)}$ or $\theta_k^{(2\gamma)}$ is omitted, we assume the result applies to iterations of $\theta_k^{(\gamma)}$. The corresponding results for $\theta_k^{(2\gamma)}$ can be obtained by substituting $2\gamma$ instead of $\gamma$. We provide some additional definitions related to the Markov kernels and kernel couplings, particularly useful when considering convergence in Wasserstein semimetric. Detailed exposition can be found in \citep[Chapter 20]{douc:moulines:priouret:soulier:2018}. 
\par 
 Let $\MK(z,A)$ be a Markov kernel on $(\Zset,\Zsigma)$. A Markov kernel $\MKK$ on $(\Zset^2,\Zsigma^{\otimes 2})$ is called a kernel coupling of $(\MK,\MK)$ (that is, of $\MK$ with itself), if for all $(z,z') \in \Zset^2$ and $A \in \Zsigma$, $\MKK((z, z'), A \times \Zset) = \MK(z, A)$ and $\MKK((z,z'), \Zset \times A) = \MK(z',A)$.
If $\MKK$ is a kernel coupling of $(\MK,\MK)$, then for all $n \in \nset$, $\MKK^n$ is a kernel coupling of $(\MK^n,\MK^n)$ and for any $\Pi \in \mathscr{C}(\xi, \xi')$, $\Pi \MKK^n$ is a coupling of $(\xi \MK^n,\xi'\MK^n)$. Moreover, it holds that
\begin{equation}
\label{eq:kernel_coupling_properties}
\Wass{c}(\xi \MK^n,\xi' \MK^n) \leq \int_{\Zset\times\Zset}  \MKK^n c(z,z') \Pi(\rmd z\rmd z')\eqsp,
\end{equation}
see \citep[Corollary 20.1.4]{douc:moulines:priouret:soulier:2018}. For any  probability measure $\Pi$ on $(\Zset^2,\Zsigma^{\otimes 2})$, we denote by $\PP_{\Pi}^\MKK$ and  $\PE_{\Pi}^\MKK$ the probability and  the
expectation on the canonical space $((\Zset^2)^\nset,(\Zsigma^{\otimes 2})^{\otimes \nset})$  such that the canonical process
$\{(Z_n, Z_n'), n \in \nset\}$ is a Markov chain with initial
probability $\Pi$ and Markov kernel $\MKK$. We write
$\PE_{z,z'}^{\MKK}$ instead $\PE_{\delta_{z,z'}}^{\MKK}$.
\par 
To prove \Cref{lem:Wasserstein_ergodicity} we need the following auxiliary lemma about the last iterate of SGD algorithm. It can be found in \citep[Lemma~10]{durmus2020biassgd}, but we provide its proof here for completeness.
\begin{lemma}
\label{lem:second_moment_last_iterate}
Assume \Cref{ass:mu-convex},  \Cref{ass:L-smooth}, and \Cref{ass:rand_noise}($2$). Then for any $\gamma \in (0;1/(2\L)]$ and  any $k,r\in \nset$ it holds that 
\begin{equation}
\PE^{1/2}[\norm{\theta_{k+r}-\thetas}^2|\mathcal{F}_k]\leq (1-\gamma\mu)^{r/2}\norm{\theta_k-\thetas} + \frac{2^{1/2}\gamma^{1/2}\tau_2}{\mu^{1/2}}
\end{equation}
\end{lemma}
\begin{proof}
Using the recurrence \eqref{eq:sgd_recursion}, we get
    \begin{align*}
    &\PE[\norm{\theta_{k+1}-\thetas}^2|\mathcal{F}_k] = 
    \PE\bigl[\norm{\theta_{k}-\thetas - \gamma\nabla F(\theta_k, \xi_{k+1})}^2|\mathcal{F}_k\bigr]\\ & \qquad =
    \PE\bigl[\norm{\theta_k-\thetas}^2 -2\gamma\langle \nabla F(\theta_k, \xi_{k+1}), \theta_k-\thetas \rangle + \gamma^2\norm{\nabla F(\theta_k, \xi_{k+1})}^2|\mathcal{F}_k\bigr]\eqsp.
    \end{align*}
    Applying \Cref{ass:rand_noise}($2$), we get 
    \begin{equation*}
    \PE[\norm{\theta_{k+1}-\thetas}^2|\mathcal{F}_k]\leq \norm{\theta_k-\thetas}^2 -2\gamma \langle\nabla f(\theta_k) - \nabla f(\thetas), \theta_k-\thetas \rangle + \gamma^2\PE[\norm{\nabla F(\theta_k, \xi_{k+1})}^2|\mathcal{F}_k]\eqsp.
    \end{equation*}
    Since $\nabla f(\thetas) = 0$, using \Cref{ass:rand_noise}($2$), we obtain 
    \begin{align*}
        \PE[\norm{\nabla F(\theta_k, \xi_{k+1})}^2|\mathcal{F}_k] &=  \PE[\norm{\nabla F(\theta_k, \xi_{k+1})-\nabla F(\thetas, \xi_{k+1}) + \noise{k+1}(\thetas)}^2|\mathcal{F}_k]\\&\leq 2\L\langle\nabla f(\theta_k) - \nabla f(\thetas), \theta_k-\thetas \rangle  + 2\tau_2^2\eqsp.
    \end{align*}
    Using \Cref{ass:mu-convex}, \Cref{ass:L-smooth}, and the fact that $\gamma \leq 1/(2\L)$, we get 
    \[
        \PE[\norm{\theta_{k+1}-\thetas}^2|\mathcal{F}_k]\leq (1- 2\gamma\mu(1-\L\gamma))\norm{\theta_k-\thetas}^2 + 2\gamma^2\tau_2^2 \leq (1- \gamma\mu)\norm{\theta_k-\thetas} + 2\gamma^2\tau_2^2\eqsp.
    \]
    Hence, applying tower property for conditional expectations, we obtain 
    \begin{equation*}
          \PE[\norm{\theta_{k+r}-\thetas}^2|\mathcal{F}_k] \leq (1-\gamma\mu)^r\norm{\theta_k-\thetas}^2 + 2\gamma^2\tau_2^2\sum_{i=0}^r(1- \gamma\mu)^i \leq(1-\gamma\mu)^r\norm{\theta_k-\thetas}^2 + \frac{2\gamma\tau_2^2}{\mu}\eqsp.
    \end{equation*}
\end{proof}

\subsection{Proof of \Cref{lem:Wasserstein_ergodicity}} 
\label{sec:proof_wass_ergodicity}
Consider the synchronous coupling construction defined by the recursions
\begin{equation}
\label{eq:kernel_coupling_construction}
\begin{split}
\theta^{(\gamma)}_{k+1} &= \theta^{(\gamma)}_{k} - \gamma \nabla F(\theta^{(\gamma)}_{k}, \xi_{k+1})\eqsp, \quad \theta^{(\gamma)}_{0} = \theta \in \rset^{d}\eqsp, \\
\tilde{\theta}^{(\gamma)}_{k+1} &= \tilde{\theta}^{(\gamma)}_{k} - \gamma \nabla F(\tilde{\theta}^{(\gamma)}_{k}, \xi_{k+1})\eqsp, \quad \tilde{\theta}^{(\gamma)}_{0} = \tilde{\theta} \in \rset^{d}\eqsp.
\end{split}
\end{equation}
The pair $(\theta^{(\gamma)}_{k},\tilde{\theta}^{(\gamma)}_{k})_{k \in \nset}$ defines a Markov chain with the Markov kernel $\MKK_{\gamma}(\cdot,\cdot)$, which is a coupling kernel of $(\MK_{\gamma},\MK_{\gamma})$. From now on we omit an upper index $(\gamma)$ and write simply $(\theta_{k},\tilde{\theta}_{k})_{k \in \nset}$. Applying now \Cref{ass:rand_noise}($2$), we get for $\gamma \leq 1/\L$ that
\begin{align}
&\PE[\norm{\theta_{k+1} - \tilde{\theta}_{k+1}}^2|\F_k] = \PE[\norm{\theta_k - \tilde{\theta}_{k} - \gamma(\nabla F(\theta_k, \xi_{k+1})-\nabla F(\tilde{\theta}_{k}, \xi_{k+1}))}^2|\F_k] \nonumber \\ 
& \qquad \qquad  = \norm{\theta_k - \tilde{\theta}_{k}}^2 + \gamma^2 \PE[\norm{\nabla F(\theta_k, \xi_{k+1})-\nabla F(\tilde{\theta}_{k}, \xi_{k+1})}^2|\F_k] \nonumber \\
&\qquad \qquad \qquad \qquad - 2\gamma \langle \nabla f(\theta_k)-\nabla f(\tilde{\theta}_{k}), \theta_k - \tilde{\theta}_{k}\rangle \nonumber \\
& \qquad \qquad \leq (1-\gamma\mu)\norm{\theta_k-\tilde{\theta}_{k}}^2 \label{eq:1-step-contr-second-moment},
\end{align}
where in the last inequality we additionally used $1-2\gamma\mu(1-\gamma \L/2) \leq 1-\gamma\mu$. Similarly, for a cost function $c$ defined in \eqref{eq:function_c_def}, we get using H\"older's and  Minkowski’s inequalities, that for any $r\in \nset$
\begin{multline*}
\PE[c(\theta_{k+r}, \tilde{\theta}_{k+r})|\F_k] 
\leq  \PE^{1/2}[\norm{\theta_{k+r}-\tilde{\theta}_{k+r}}^2|\F_k]\bigl( \PE^{1/2}[\norm{\theta_{k+r} -\thetas}^2|\F_k] \\
+ \PE^{1/2}[\norm{\tilde{\theta}_{k+r} -\thetas}^2|\F_k] + \frac{2^{3/2}\gamma^{1/2}\tau_2}{\mu^{1/2}}\bigr)\eqsp.
\end{multline*}
Combining the above inequalities and applying \Cref{lem:second_moment_last_iterate}, we obtain 
\begin{align*}
\PE[c(\theta_{k+r}, \theta'_{k+r})|\F_k] 
& \leq (1-\gamma\mu)^{r/2}\norm{\theta_k-\tilde{\theta}_k}\bigl((1-\gamma\mu)^{r/2}(\norm{\theta_k - \thetas} + \norm{\tilde{\theta}_k - \thetas}) + \frac{2^{5/2}\gamma^{1/2}\tau_2}{\mu^{1/2}}\bigr) \\ 
& \leq 2(1-\gamma \mu)^{r/2}c(\theta_k, \theta'_k)\eqsp.
\end{align*}
Note that $2(1-\gamma \mu)^{r/2} \leq 2$ for any $r \leq m(\gamma)-1$ and $2(1-\gamma \mu)^{m(\gamma)/2} \leq 1/2$. Hence, applying the result of \citep[Theorem 20.3.4]{douc:moulines:priouret:soulier:2018}, we obtain that the Markov kernel $\MK_\gamma$ admits a unique invariant distribution $\pi_\gamma$. Applying \eqref{eq:kernel_coupling_properties}, we get
\begin{equation}
    \Wass{c}(\nu \MK_{\gamma}^k, \pi_\gamma) \leq  2 (1/2)^{\lfloor k/m(\gamma) \rfloor} \Wass{c}(\nu, \pi_\gamma)\eqsp.
\end{equation}
It remains to note that $(1/2)^{\lfloor k/m(\gamma) \rfloor} \leq 2 (1/2)^{ k/m(\gamma)}$, and the statement follows.

\subsection{Proof of \Cref{prop:bias_PR}} 
We first prove \eqref{eq:bias_expansion_result} and \eqref{eq:bias_expansion_result_2} and introduce some additional notations. Under \Cref{ass:mu-convex} -- \Cref{ass:rand_noise}($2$), we define a matrix-valued function $\mathcal{C}(\theta): \rset^d \rightarrow \rset^{d \times d}$ as
\begin{equation}
\label{eq:C_theta_definition}
\calC(\theta) = \PE[\noise{1}(\theta)^{\otimes 2}]\eqsp.
\end{equation}

The result below is essentially based on an appropriate modification of the bounds presented in \citet[Lemma~18]{durmus2020biassgd}. A careful inspection of its proof reveals that we do not need additional assumptions on $\calC(\theta)$,  instead we use \Cref{lem: noise_covariance_bound}. 
\begin{lemma}
\label{lem:bias_decomposition}
Assume \Cref{ass:mu-convex}, \Cref{ass:L-smooth}, \Cref{ass:rand_noise}($6$), and \Cref{ass:stationary_moments_bounds}($6$). Then, for any $\gamma \in (0, 1/(\L \Conststep{6})]$, it holds
\begin{equation}
\label{eq:firstdev}
\bgammatheta - \thetas = -(\gamma/2) \{\H\}^{-1}
 \{\nabla^{3}f(\thetas)\} \mathbf{T} \calC(\thetas) + B_1 \gamma^{3/2}\eqsp,
\end{equation}
where $\bgammatheta$ is defined in \eqref{eq:bias_expansion_result}, $\calC(\theta)$ is defined in \eqref{eq:C_theta_definition}, and $B_1 \in \rset^{d}$ satisfies $\norm{B_1} \leq \ConstPR{b,1}$, with
\begin{equation}
\begin{split}
\label{eq: def_ConstPRb1,2}
&\ConstPR{b,1} = \frac{\L}{\mu}\ConstPR{b, 2} + \frac{\L\Constlast{6}^{3/2}\tau_6^3}{2\mu^{5/2}}\\
&\ConstPR{b,2} = \frac{\sqrt{\L}\tau_2^2}{\sqrt{\Conststep{6}}\mu} + \frac{1}{2}\left(\left(\frac{\L^2\Constlast{2}}{\mu^{3/2}} + \frac{\L\sqrt{\Constlast{2}}}{\sqrt{\mu}}\right)\tau_2^2 + \frac{\L\Constlast{6}^{3/2}\tau_6^3 }{2\mu^{3/2}} + \frac{\L^{1/2}\Constlast{4}^2  \tau_{4}^4}{4\mu^2\Conststep{6}^{3/2}}\right)
\end{split}
\end{equation}
Moreover, 
\begin{equation}
\label{eq:seconddev}
\covgammatheta = \gamma \mathbf{T} \calC(\thetas) + B_2 \gamma^{3/2}\eqsp,
\end{equation}
where the operator $\mathbf{T}: \rset^{d \times d} \to \rset^{d \times d}$ is defined by the relation 
\[
\vec{\mathbf{T} A} = (\H \otimes \Id + \Id \otimes \H)^{-1} \vec{A}
\]
for any matrix $A \in \rset^{d \times d}$, and $B_2 \in \rset^{d \times d}$ is a matrix, such that $\norm{B_2} \leq \ConstPR{b,2}$.
\end{lemma}

\begin{proof}
Let $(\theta_k^{(\gamma)})_{k \in \nset}$ be a recurrence defined in \eqref{eq:sgd_recursion} with initial distribution $\theta_0 \sim \pi_{\gamma}$. Recall that $\theta_0$ is independent from the noise variables $(\xi_k)_{k \geq 1}$. First, applying a third-order Taylor expansion of $\nabla f(\theta)$ around $\thetas$, we obtain
\begin{equation}
\label{eq:theo:statio_general_1_0}
\nabla f(\theta) = \H (\theta - \thetas) + (1/2) \{\nabla^{3}f(\thetas)\}(\theta - \thetas)^{\otimes 2} + G(\theta)\eqsp,
\end{equation}
where $G(\theta)$ has a form 
\[
G(\theta) = \frac{1}{2}\left(\int_{0}^{1} t^2 \nabla^{4}f(t\thetas + (1-t)\theta)\,dt\right) (\theta - \thetas)^{\otimes 3}\eqsp.
\]
Thus, using \Cref{ass:L-smooth},
\[
\norm{G(\theta)} \leq \frac{\L_{4}}{2} \norm{\theta-\thetas}^{3}\eqsp.
\]
Integrating \eqref{eq:theo:statio_general_1_0} with respect to $\pi_{\gamma}$, we get  
\begin{equation}
\label{eq:integration_pi_gamma_grad}
\H (\bgammatheta - \thetas) + (1/2)\{\nabla^{3}f(\thetas)\} \left[\int_{\rset^{d}} (\theta-\thetas)^{\otimes 2} \pi_{\gamma} (\rmd \theta) \right] = -\int_{\rset^{d}} G(\theta) \pi_{\gamma} (\rmd \theta)\eqsp.
\end{equation}
Moreover, using \Cref{ass:stationary_moments_bounds}($6$), we have
\begin{equation}
    \norm{\int_{\rset^{d}} G(\theta) \pi_{\gamma} (\rmd \theta)} \leq \gamma^{3/2}\frac{\L\Constlast{6}^{3/2}\tau_6^3}{2\mu^{3/2}}\eqsp.
    \end{equation}
Now we provide an explicit expression for the covariance matrix 
\begin{equation}
\covgammatheta = \int_{\rset^{d}} (\theta-\thetas)^{\otimes 2} \pi_{\gamma} (\rmd \theta)\eqsp.
\end{equation}
Using the recurrence \eqref{eq:sgd_recursion}, we obtain that 
\[
\theta_{1} - \thetas = (\Id - \gamma \H)(\theta_0 - \thetas) - \gamma \noise{1}(\theta_0) -\gamma \eta(\theta_0)\eqsp,
\]
where the function $\eta(\cdot)$ is defined in \eqref{eq:eta_k_definition}. Hence, taking second moment w.r.t. $\pi_{\gamma}$ from both sides, we get that 
\begin{multline}
\label{eq:cov_gamma_estimate}
\covgammatheta = (\Id - \gamma \H) \covgammatheta (\Id - \gamma \H) + \gamma^2 \int_{\rset^{d}} \calC(\theta) \pi_{\gamma} (\rmd \theta) + \gamma^2 \int_{\rset^{d}} \{\eta(\theta)\}^{\otimes 2} \pi_{\gamma} (\rmd \theta) \\
- \gamma \int_{\rset^{d}} \left[ (\Id - \gamma \H)(\theta - \thetas)\{\eta(\theta)\}^{\top} + \eta(\theta)(\theta - \thetas)^{\top} (\Id - \gamma \H) \right] \pi_{\gamma} (\rmd \theta)\eqsp. 
\end{multline}
In the above equation $\calC(\theta)$ is defined in \eqref{eq:C_theta_definition}, and we additionally used that $\CPE{\noise{1}(\theta_0)}{\F_{0}} = 0$. Using Taylor's expansion with integral remainder together with \Cref{ass:L-smooth} and \Cref{ass:stationary_moments_bounds}($6$), 
\begin{align*}
    &\gamma^2 \norm{\int_{\rset^{d}} \{\eta(\theta)\}^{\otimes 2} \pi_{\gamma} (\rmd \theta)}[F]\leq \gamma^4\frac{\L^2\Constlast{4}^2  \tau_{4}^4}{4\mu^2}\eqsp,\\
    &\gamma \norm{\int_{\rset^{d}} \left[ (\Id - \gamma \H)(\theta - \thetas)\{\eta(\theta)\}^{\top} + \eta(\theta)(\theta - \thetas)^{\top} (\Id - \gamma \H) \right] \pi_{\gamma} (\rmd \theta)}[F] \leq \gamma^{5/2}\frac{\L\Constlast{6}^{3/2}\tau_6^3 }{2\mu^{3/2}}
\end{align*}
Moreover, \eqref{eq:C_theta_definition} together with \Cref{ass:stationary_moments_bounds}($6$) imply that 
\[
\int_{\rset^{d}} \calC(\theta) \pi_{\gamma} (\rmd \theta) = \calC(\thetas) + B \gamma^{1/2}\eqsp, 
\]
where $B \in \rset^{d \times d}$ satisfies $\norm{B} \leq \ConstPR{b,3}$. Using  \eqref{eq:cov_gamma_estimate} together with \Cref{ass:stationary_moments_bounds}($6$), we obtain that $\covgammatheta$ is a solution to the matrix equation 
\begin{equation}
\label{eq:equation_on_covgamma}
\H \covgammatheta + \covgammatheta \H - \gamma \H \covgammatheta \H =  \gamma \calC(\thetas) + B' \gamma^{3/2}\eqsp,
\end{equation}
where
\begin{equation}
\label{eq:bound_norm_B'}
    \norm{B'}[F] \leq \ConstPR{b,3} + \frac{\L\Constlast{6}^{3/2}\tau_6^3 }{2\mu^{3/2}} + \frac{\L^{1/2}\Constlast{4}^2  \tau_{4}^4}{4\mu^2\Conststep{6}^{3/2}}\eqsp.
\end{equation}
The matrix equation \eqref{eq:equation_on_covgamma} can be written using vectorization operation as 
\begin{multline*}
\vec{\covgammatheta} = \gamma (\H \otimes \Id + \Id \otimes \H - \gamma \H \otimes \H)^{-1} \vec{\calC(\thetas)} \\ + \gamma^{3/2} (\H \otimes \Id + \Id \otimes \H - \gamma \H \otimes \H)^{-1} \vec{B'}\eqsp.
\end{multline*}
Applying \Cref{lem:propirties_operator_prop_2}\ref{eq:decomposition_operator}, we obtain that 
\[(\H \otimes \Id + \Id \otimes \H - \gamma \H \otimes \H)^{-1} = (\H \otimes \Id + \Id \otimes \H)^{-1} + D\eqsp,
\]
where $D \in \rset^{d^2 \times d^2}$ is a matrix which satisfies 
\[
\norm{D} \leq \gamma \L/\mu\eqsp.
\]
Thus, 
\begin{multline*}
    \vec{\covgammatheta} = \gamma (\H \otimes \Id + \Id \otimes \H)^{-1} \vec{\calC(\thetas)} + \gamma^{3/2} (D/\sqrt{\gamma})\vec{\calC(\thetas)} \\ + \gamma^{3/2} (\H \otimes \Id + \Id \otimes \H - \gamma \H \otimes \H)^{-1} \vec{B'}\eqsp.
\end{multline*}
We define the matrix $B_2$ such that 
\begin{equation}
    \vec{B_2} = (D/\sqrt{\gamma})\vec{\calC(\thetas)} +(\H \otimes \Id + \Id \otimes \H - \gamma \H \otimes \H)^{-1} \vec{B'}
\end{equation}
Hence, using \Cref{ass:rand_noise}, \eqref{eq:bound_norm_B'}, and \Cref{lem:propirties_operator_prop_2}, we get
\begin{align*}
    \norm{B_2} &\leq \norm{B_2}[F]= \norm{\vec{B_2}} 
    \\&\leq \norm{D/\sqrt{\gamma}}\norm{\vec{\calC(\thetas)}} +  \norm{(\H \otimes \Id + \Id \otimes \H - \gamma \H \otimes \H)^{-1}}\norm{B'}[F]\\&\leq 
    \frac{\sqrt{\gamma}\L\tau_2^2}{\mu} + \frac{1}{2}\left(\ConstPR{b,3} + \frac{\L\Constlast{6}^{3/2}\tau_6^3 }{2\mu^{3/2}} + \frac{\L^{1/2}\Constlast{4}^2  \tau_{4}^4}{4\mu^2\Conststep{6}^{3/2}}\right)\\ & \leq \frac{\sqrt{\L}\tau_2^2}{\sqrt{\Conststep{6}}\mu} + \frac{1}{2}\left(\ConstPR{b,3} + \frac{\L\Constlast{6}^{3/2}\tau_6^3 }{2\mu^{3/2}} + \frac{L^{1/2}\Constlast{4}^2  \tau_{4}^4}{4\mu^2\Conststep{6}^{3/2}}\right)\eqsp,
\end{align*}
where in the last inequality we use that $\gamma \leq 1/(\L\Conststep{6})$. Combining the above bounds in \eqref{eq:integration_pi_gamma_grad}, we arrive at the expansion formula \eqref{eq:firstdev}.
\end{proof}

\begin{lemma}
\label{lem:propirties_operator_prop_2}
Assume \Cref{ass:mu-convex} and \Cref{ass:L-smooth}. Then for any $\gamma \in (0, 1/(\L \Conststep{6})]$ it holds
\begin{enumerate}[label=(\alph*)]
    \item \label{eq:eig_val_operators} All eigenvalues $\tilde{\lambda}_i$, $i \in \{1,\ldots,d^2\}$ of the matrix $\H \otimes \Id + \Id \otimes \H - \gamma \H \otimes \H$ satisfy
    \[
    2\mu(1-\gamma \L/2)\leq \tilde{\lambda}_i \leq 2\L(1-\gamma\mu/2)\eqsp;
    \]
    \item \label{eq:norm_inverse_operator}$ \norm{(\H \otimes \Id + \Id \otimes \H - \gamma \H \otimes \H)^{-1}} \leq 1/2$;
    \item \label{eq:decomposition_operator}
    In addition, 
    \[(\H \otimes \Id + \Id \otimes \H - \gamma \H \otimes \H)^{-1} = (\H \otimes \Id + \Id \otimes \H)^{-1} + D \text{ where }\norm{D} \leq \gamma \L/\mu\eqsp.\]
\end{enumerate}
\end{lemma}
\begin{proof}
Assumption \Cref{ass:mu-convex} guarantees that the symmetric matrix $\H$ is positive-definite. Let $u_1,\ldots,u_d \in \rset^{d}$ and $\lambda_1 \geq \lambda_2 \geq \ldots \geq \lambda_d \geq \mu > 0$ be its eigenvectors and eigenvalues, respectively. Then we notice that 
\[
\H \otimes \Id + \Id \otimes \H - \gamma \H \otimes \H = \H \otimes (\Id - (\gamma/2) \H) + (\Id - (\gamma/2) \H) \otimes \H\eqsp.
\]
Hence, the latter operator is also diagonalizable in the orthogonal basis $u_i \otimes u_j \in \rset^{d^2}$ with the respective eigenvalues being equal to $\lambda_i (1 - (\gamma/2) \lambda_j) + \lambda_j (1 - (\gamma/2) \lambda_i)$. Hence, we obtain the first part of lemma \ref{eq:eig_val_operators}. To prove \ref{eq:norm_inverse_operator} it remains to note that for $\gamma \leq 1/\L$ it holds $(2\mu(1-\gamma \L/2))^{-1} \leq 1/2$.
Set now 
\begin{equation}
\label{eq:S_R_def}
\begin{split}
S &= \H \otimes \Id + \Id \otimes \H \in \rset^{d^2 \times d^2} \\
R&= \H \otimes \H \in \rset^{d^2 \times d^2}\eqsp.
\end{split}
\end{equation}
Then it is easy to observe that 
\begin{align*}
(S - \gamma R)^{-1} = S^{-1} + S^{-1} \sum_{k=1}^{\infty} \gamma^k (R S^{-1})^{k}\eqsp,
\end{align*}
provided that $\gamma \norm{R S^{-1}} < 1$. Sice $R$ and $S$ are diagonalizable in the same orthogonal basis $\{u_i \otimes u_j\}_{1 \leq i,j \leq d}$ with the eigenvalues $\lambda_i \lambda_j$ and $\lambda_i + \lambda_j$, respectively, the condition $\gamma \norm{R S^{-1}} < 1$ holds provided that $\gamma < 2/\L$. Hence, for $\gamma \leq 1/\L$, it holds that 
\[
(\H \otimes \Id + \Id \otimes \H - \gamma \H \otimes \H)^{-1} = (\H \otimes \Id + \Id \otimes \H)^{-1} + D\eqsp,
\]
where $D \in \rset^{d^2 \times d^2}$ satisfies 
\[
\norm{D} \leq 2\gamma \norm{S^{-1}} \norm{R S^{-1}} \leq \frac{\gamma \L}{\mu}\eqsp.
\]
\end{proof}

We now state an auxiliary lemma about the function $\calC(\theta)$ from \eqref{eq:C_theta_definition}.
\begin{lemma}
\label{lem: noise_covariance_bound}
Assume \Cref{ass:mu-convex}, \Cref{ass:L-smooth}, \Cref{ass:rand_noise}($2$),  and \Cref{ass:stationary_moments_bounds}($2$). Then, for any $\gamma \in (0, 1/(\L \Conststep{2})]$, it holds
\[
\norm{\int_{\rset^{d}} \calC(\theta) \pi_{\gamma} (\rmd \theta) - \calC(\thetas)}[F] \leq \ConstPR{b,3}\gamma^{1/2}\eqsp,  
\]
where the constant $\ConstPR{b,3}$ is given by
\begin{equation}
\label{eq: def_ConstPR2}
\ConstPR{b,3} = \left(\frac{\L^2\Constlast{2}}{\mu^{3/2}} + \frac{\L\sqrt{\Constlast{2}}}{\sqrt{\mu}}\right)\tau_2^2\eqsp.
\end{equation}
\end{lemma}
\begin{proof}
Recall that 
\begin{align*}
\noise{1}(\theta) = \nabla F(\theta,\xi_{1}) - \nabla f(\theta)\eqsp.
\end{align*}
Hence, using the definition of $\calC(\theta)$ in \eqref{eq:C_theta_definition}, we get, with $\theta \in \rset^{d}$, that  
\begin{gather*}
\calC(\theta) - \calC(\thetas) = \PE[(\noise{1}(\theta)-\noise{1}(\thetas))(\noise{1}(\theta)-\noise{1}(\thetas))^T] + \PE[\noise{1}(\thetas)(\noise{1}(\theta)-\noise{1}(\thetas))^T] \\+ \PE[(\noise{1}(\theta)-\noise{1}(\thetas))\noise{1}(\thetas)^T].
\end{gather*}
Using \Cref{ass:rand_noise}($2$), we obtain 
\begin{equation*}
\PE[\norm{\noise{1}(\theta)-\noise{1}(\thetas)}^2] \leq \L \langle \nabla f(\theta) - \nabla f(\thetas), \theta -\thetas \rangle - \norm{\nabla f(\theta)-\nabla f(\thetas)}^2 \leq \L^2 \norm{\theta - \thetas}^2.
\end{equation*}
Hence, combining the previous inequalities and using H\"older's inequality, we obtain for any $\theta \in \rset^{d}$, that 
\begin{equation*}
\norm{\calC(\theta) - \calC(\thetas)}[F]  \leq\L^2\norm{\theta - \thetas}^2 + \tau_2 \L \norm{\theta - \thetas}.
\end{equation*}
Applying now \Cref{ass:stationary_moments_bounds}($2$), we obtain 
\begin{equation*}
    \norm{\int_{\rset^{d}} \calC(\theta)\pi_{\gamma} (\rmd \theta) - \calC(\thetas)}[F] \leq \int_{\rset^{d}} \norm{\calC(\theta) - \calC(\thetas)}[F] \pi_{\gamma} (\rmd \theta) \leq \L^2 \frac{\Constlast{2} \gamma \tau_{2}^2}{\mu}+ \tau_2 \L \sqrt{\frac{\Constlast{2} \gamma \tau_{2}^2}{\mu}}.
\end{equation*}
We conclude the proof by using the fact that $\gamma \mu \leq 1$.
\end{proof}

Now we prove \eqref{eq:bias_PR_bound}. We use synchronous coupling construction defined by the pair of recursions:
\begin{align*}
&\theta_{k+1} = \theta_k - \gamma\nabla F(\theta_k, \xi_{k+1}), \quad \theta_0 \sim \nu \\
&\tilde{\theta}_{k+1} = \tilde{\theta}_{k} - \gamma\nabla F(\tilde{\theta}_{k}, \xi_{k+1}), \quad \tilde{\theta}_{0} \sim \pi_\gamma\eqsp.
\end{align*}
Recall that the corresponding coupling kernel is denoted as $\MKK_{\gamma}(\cdot,\cdot)$. Then we obtain 
\begin{align*}
\PE_{\nu}[\bar{\theta}_{n}] - \thetas 
&= n^{-1} \sum_{k=n+1}^{2n} \PE^{\MKK_{\gamma}}_{\nu, \pi_{\gamma}}[\theta_k - \tilde{\theta}_k] + n^{-1} \sum_{k=n+1}^{2n}\PE_{\pi_{\gamma}}[\tilde{\theta}_k - \thetas] \\
&= n^{-1} \sum_{k=n+1}^{2n} \PE^{\MKK_{\gamma}}_{\nu, \pi_{\gamma}}[\theta_k - \tilde{\theta}_k] + (\bgammatheta - \thetas)\eqsp.
\end{align*}
Using \eqref{eq:1-step-contr-second-moment} and \Cref{ass:stationary_moments_bounds}($2$), we obtain 
\begin{align*}
\norm{\PE^K_{\nu, \pi_{\gamma}}[\theta_k - \tilde{\theta}_k]} 
&\leq (1-\gamma\mu)^{k/2}\{\PE^{\MKK_{\gamma}}_{\nu, \pi_{\gamma}}\norm{\theta_0-\tilde{\theta}_0}^2\}^{1/2} \\
&\leq (1-\gamma\mu)^{k/2} (\PE_{\nu}^{1/2}\bigl[\norm{\theta_0 -\thetas}^2\bigr] + \frac{\sqrt{2\gamma}\tau_2}{\sqrt{\mu}})\eqsp.
\end{align*}
Summing the above bounds for $k$ from $n+1$ to $2n$, we obtain \eqref{eq:bias_PR_bound}.

\subsection{Proof of \Cref{prop:bias_RR}} 
Note that 
\begin{equation*}
\PE_{\nu}[\bar\theta_n^{(RR)}-\thetas] = 2\PE_{\nu}[\bar\theta_n^{\gamma}-\thetas] - \PE_{\nu}[\bar\theta_n^{2\gamma}-\thetas].
\end{equation*}
Applying \eqref{eq:bias_PR_bound}, we obtain 
\begin{equation}
\norm{\PE_{\nu}[\bar\theta_n^{(RR)}-\thetas]} \leq (\frac{L}{\mu}\ConstPR{b,2} + \frac{L\Constlast{6}^{3/2}\tau_6^3}{2\mu^{5/2}})\gamma^{3/2} + \Rem_{3}(\theta_0-\thetas,\gamma,n),
\end{equation}
where 
\begin{equation}
\norm{\Rem_{3}(\theta_0-\thetas,\gamma,n)} \lesssim \frac{(1-\gamma\mu)^{(n+1)/2}}{n\gamma \mu}(\PE_{\nu}^{1/2}\bigl[\norm{\theta_0 -\thetas}^2\bigr] + \frac{\sqrt{\gamma}\tau_2}{\sqrt{\mu}})\eqsp,
\end{equation}
and the statement follows.

\section{Proof of \Cref{th:mart_decomposition_pr_error}}
\label{sec:proof_th_main_no_rr}
\begin{theorem}[Version of \Cref{th:mart_decomposition_pr_error} with explicit constants]
\label{th:mart_decomposition_pr_error_explicit_constants}

Assume \Cref{ass:mu-convex}, \Cref{ass:L-smooth}, \Cref{ass:rand_noise}($4$), and \Cref{ass:stationary_moments_bounds}($4$). Then for any $\gamma \in (0, 1/(\L \Conststep{4})]$, $n \in \nset$, and initial distribution $\nu$ on $\rset^{d}$, the sequence of Polyak-Ruppert estimates \eqref{eq:pr_averaged_est} satisfies
\begin{equation*}
\PE_{\nu}^{1/2}[\norm{\H(\bar{\theta}_n^{(\gamma)} - \thetas)}^{2}] \leq \frac{\sqrt{\trace{\noisecov}}}{\sqrt{n}} + \frac{\ConstPR{2}}{\gamma^{1/2}n} + \ConstPR{3} \gamma + \frac{\ConstPR{4} \gamma^{1/2}}{n^{1/2}} +  \Rem_{2}(n, \gamma, \norm{\theta_0-\thetas})\eqsp,
\end{equation*}
where we have set 
\begin{equation}
\label{eq:Const_2_3_4_def}
\begin{split}
\ConstPR{2} &= c_0 \Constlast{2}^{1/2}\tau_2 \eqsp, \quad 
\ConstPR{3} = c_0 \frac{\L \Constlast{4} \tau_4^2}{2\mu}\eqsp, \quad \ConstPR{4} = c_0 \L \Constlast{2}^{1/2} \tau_2 \eqsp.
\end{split}
\end{equation}
and the remainder term $\Rem_{2}(n, \gamma, \norm{\theta_0-\thetas})$ is given by 
\begin{multline}
\label{eq:remainder_term_2_definition}
\Rem_{2}(n, \gamma, \norm{\theta_0-\thetas}) = \frac{c_0 \L (1-\gamma\mu)^{(n+1)/2}}{\gamma \mu n} \PE_{\nu}^{1/2}\bigl[\norm{\theta_0-\thetas}^2\bigr] \\
+ \frac{\L c_0 (1-\gamma\mu)^{n+1}}{2n\gamma\mu}\PE^{1/2}_{\nu}\bigl[\norm{\theta_0 -\thetas}^4\bigr] \eqsp.
\end{multline}
\end{theorem}
\begin{proof}
Throughout the proof we omit upper index $(\gamma)$ both for the elements of the sequence $\{\theta_k^{(\gamma)}\}_{k \in \nset}$ and Polyak-Ruppert averaged estimates $\bar{\theta}_n^{(\gamma)}$. Instead, we write simply $\theta_k$ and $\bar{\theta}_n$, respectively. Summing the recurrence \eqref{eq:summ_parts}, we obtain that
\begin{equation}
\label{eq:PR-decomposition-error}
\H (\bar{\theta}_n - \thetas) = \frac{\theta_{n+1} - \thetas}{\gamma n} - \frac{\theta_{2n+1} - \thetas}{\gamma n} - \frac{1}{n}\sum_{k=n+1}^{2n}\noise{k+1}(\theta_k) - \frac{1}{n}\sum_{k=n+1}^{2n}\eta(\theta_k)\eqsp.
\end{equation}
Applying the $3$-rd order Taylor expansion with integral remainder, we get that 
\[
\nabla f(\theta_k) = \H (\theta_k - \thetas) + \left(\int_{0}^{1}t \nabla^{3}f(t\thetas + (1-t)\theta_k)\,dt\right) (\theta_k - \thetas)^{\otimes 2}\eqsp,
\]
where $\nabla^{3}f(\cdot) \in \rset^{d \times d \times d}$. Using \Cref{ass:L-smooth}, we thus obtain that 
\[
\norm{\eta(\theta_k)} \leq \frac{1}{2} \L_{3}\norm{\theta_k - \thetas}^2\eqsp. 
\]
Applying Minkowski's inequality to the decomposition \eqref{eq:th_PR_main} and to the last term therein, we get
\begin{gather*}
\PE_{\nu}^{1/2}[\norm{\H (\bar{\theta}_n - \thetas)}^2] \leq \frac{\PE_{\nu}^{1/2}[\norm{\theta_{n+1} - \thetas}^{2}]}{\gamma n} + \frac{\PE_{\nu}^{1/2}[\norm{\theta_{2n+1} - \thetas}^{2}]}{\gamma n} + \frac{1}{n}\PE_{\nu}^{1/2}\bigl[\norm{\sum_{k=n+1}^{2n}\noise{k+1}(\theta_k)}^2\bigr] \\ + \frac{\L_3}{2n}\sum_{k=n+1}^{2n}\PE_{\nu}^{1/2}\bigl[\norm{\theta_k - \thetas}^{4}\bigr]\eqsp.
\end{gather*}
Applying \Cref{ass:stationary_moments_bounds}($2$), we obtain that for $\gamma \in (0;1/(\L \Conststep{2})]$ it holds that
\begin{equation}
\label{eq:last_iterate_2_moment}
\PE_{\nu}{\norm{\theta_k - \thetas}^{2}} \lesssim (1-\gamma\mu)^{k} \PE_{\nu}\bigl[\norm{\theta_0 -\thetas}^2\bigr] + \frac{\Constlast{2} \gamma\tau_2^2}{\mu}\eqsp.
\end{equation}
Moreover, from $\gamma \in (0;1/(\L \Conststep{4})]$ it holds that
\begin{equation}
\label{eq:last_iterate_4_moment}
\PE_{\nu}^{1/2}{\norm{\theta_k - \thetas}^{4}} \lesssim (1-\gamma\mu)^{k} \PE^{1/2}_{\nu}\bigl[\norm{\theta_0 -\thetas}^4\bigr] + \frac{\Constlast{4} \gamma \tau_4^2}{\mu}\eqsp.
\end{equation}
Combining \Cref{lem:error_bound_second_moment} with previous inequalities, we obtain 
\begin{multline*}
\PE_{\nu}^{1/2}[\norm{ \H (\bar{\theta}_n - \thetas)}^2] \lesssim \frac{\sqrt{\trace{\noisecov}}}{\sqrt{n}} + \frac{\Constlast{2}^{1/2}\tau_2}{\gamma^{1/2}n} + \frac{\L \Constlast{4} \gamma\tau_4^2}{2\mu}  + \frac{\L \Constlast{2}^{1/2}\gamma^{1/2}\tau_2}{\mu^{1/2} n^{1/2}} \\ 
+ \frac{(1-\gamma\mu)^{(n+1)/2}}{\gamma n} \left(\frac{\L}{\mu} + 1\right)\PE_{\nu}^{1/2}\bigl[\norm{\theta_0-\thetas}^2\bigr] + \frac{\L (1-\gamma\mu)^{n+1}}{n\gamma\mu}\PE_{\nu}^{1/2}\bigl[\norm{\theta_0-\thetas}^4\bigr]\eqsp,
\end{multline*}
and the result follows.
\end{proof}

Below we provide an auxiliary lemma used in the proof of \Cref{th:mart_decomposition_pr_error}.

\begin{lemma}
\label{lem:error_bound_second_moment}
Assume \Cref{ass:mu-convex}, \Cref{ass:L-smooth}, \Cref{ass:rand_noise}($2$), and \Cref{ass:stationary_moments_bounds}($2$). Then for any $\gamma \in (0;1/(\L \Conststep{2})]$ and any $n \in \nset$, it holds
\begin{equation}
\label{eq:2_nd_moment_bound_eps}
\PE_{\nu}^{1/2}[\norm{\sum_{k=n+1}^{2n}\{\noise{k+1}(\theta_k) - \noise{k+1}(\thetas)\}}^2] \lesssim \frac{\L \Constlast{2}^{1/2} \sqrt{\gamma n} \tau_2}{\mu^{1/2}} + \frac{\L (1-\gamma\mu)^{(n+1)/2}}{\gamma\mu} \PE_{\nu}^{1/2}\bigl[\norm{\theta_0-\thetas}^2\bigr] \eqsp.
\end{equation}
Moreover, let $p \geq 2$, and assume \Cref{ass:mu-convex}, \Cref{ass:L-smooth}, \Cref{ass:rand_noise}($p$), and \Cref{ass:stationary_moments_bounds}($p$). Then for any $\gamma \in (0;1/(\L \Conststep{p})]$ and $n \in \nset$ it holds that 
\begin{equation}
\label{eq:p_th_moment_bound_eps}
\begin{split}
\PE_{\nu}^{1/p}[\norm{\sum_{k=n+1}^{2n}\{\noise{k+1}(\theta_k) - \noise{k+1}(\thetas)\}}^p] &\lesssim \frac{\L \Constlast{p}^{1/2} \sqrt{\gamma n} p \tau_{p}}{\mu^{1/2}} \\
&\qquad + \frac{\L p (1-\gamma \mu)^{(n+1)/2}}{\mu^{1/2} \gamma^{1/2}} \PE^{1/p}_{\nu}[\norm{\theta_0 - \thetas}^{p}]\eqsp.
\end{split}
\end{equation}
\end{lemma}
\begin{proof}
Since $\{ \noise{k+1}(\theta_k) - \noise{k+1}(\thetas)\}$ is a martingale-difference sequence with respect to $\F_k$, we have 
\begin{align*}
\PE_{\nu}\bigl[\norm{\sum_{k=n+1}^{2n}\{\noise{k+1}(\theta_k) - \noise{k+1}(\thetas)\}}^2\bigr] &= \sum_{k=n+1}^{2n} \PE_{\nu}\bigl[\norm{\{\noise{k+1}(\theta_k) - \noise{k+1}(\thetas)\}}^2\bigr]\eqsp.
\end{align*}
where $\noise{k+1}(\thetas) = \nabla F(\thetas,\xi_{k+1})$ uses the same noise variable $\xi_{k+1}$ as $F(\theta_{k},\xi_{k+1})$. 
Note that 
\begin{multline*}
\PE_{\nu}[\norm{\noise{k+1}(\theta_k) - \noise{k+1}(\thetas)}^2]  
= \PE_{\nu}[\norm{\nabla F(\theta_k, \xi_{k+1})-\nabla F(\thetas, \xi_{k+1})}^2 \\
 \qquad - 2  \PE_{\nu} \bigl[ \langle \nabla F(\theta_k, \xi_{k+1})-\nabla F(\thetas, \xi_{k+1}), \nabla f(\theta_k)-\nabla f(\thetas)\rangle \bigr] + \norm{\nabla f(\theta_k)-\nabla f(\thetas)}^2]\eqsp.
\end{multline*}
Using \Cref{ass:L-smooth}, \Cref{ass:rand_noise}($2$), and taking conditional expectation with respect to $\mathcal{F}_k$, we obtain 
\begin{align*}
\PE_{\nu}[\norm{\noise{k+1}(\theta_k) - \noise{k+1}(\thetas)}^2] 
&\leq  \PE_{\nu} [\L \langle \nabla f(\theta_k)-\nabla f(\thetas), \theta_k - \thetas \rangle - \norm{\nabla f(\theta_k)-\nabla f(\thetas)}^2] \\ 
&\leq \L^2 \PE_{\nu}[\norm{\theta_k - \thetas}^2].
\end{align*}
Thus, we obtain that 
\begin{align*}
\PE_{\nu}[\norm{\sum_{k=n+1}^{2n}\{\noise{k+1}(\theta_k) - \noise{k+1}(\thetas)\}}^2] \leq \L^2 \sum_{k=n+1}^{2n}\PE_{\nu}[\norm{\theta_k-\thetas}^2]\eqsp,
\end{align*}
and the statement \eqref{eq:2_nd_moment_bound_eps} follows from the assumption \Cref{ass:stationary_moments_bounds}($2$). In order to prove \eqref{eq:p_th_moment_bound_eps}, we apply Burkholder's inequality \citet[Theorem~8.6]{osekowski:2012} and obtain 
\begin{align*}
&\PE_{\nu}^{1/p}[\norm{\sum_{k=n+1}^{2n}\{\noise{k+1}(\theta_k) - \noise{k+1}(\thetas)\}}^p] \leq p \PE_{\nu}^{1/p}\bigl[ \bigl(\sum_{k=n+1}^{2n}\norm{\noise{k+1}(\theta_k) - \noise{k+1}(\thetas)}^2\bigr)^{p/2} \bigr] \\
& \qquad \leq p \bigl( \sum_{k=n+1}^{2n} \PE_{\nu}^{2/p}\bigl[\norm{\noise{k+1}(\theta_k) - \noise{k+1}(\thetas)}^{p}\bigr] \bigr)^{1/2} \\
& \qquad \lesssim p \L \bigl( \sum_{k=n+1}^{2n} \PE^{2/p}_{\nu}\bigl[\norm{\theta_k - \thetas}^{p}\bigr] \bigr)^{1/2} \\
& \qquad \overset{(a)}{\lesssim} \frac{\L \Constlast{p}^{1/2} \sqrt{\gamma n} p \tau_{p}}{\mu^{1/2}} + \frac{\L p (1-\gamma \mu)^{(n+1)/2}}{\mu^{1/2} \gamma^{1/2}} \PE^{1/p}_{\nu}[\norm{\theta_0 - \thetas}^{p}] \eqsp,
\end{align*}
where in (a) we have additionally used \Cref{ass:stationary_moments_bounds}($p$). 
\end{proof}

\section{Proof of \Cref{th:RR_second_moment}}
\label{th:RR_second_moment_proof}
Within this section we often use the definition of the function $\psi: \rset^{d} \to \rset^{d}$ from \eqref{eq:psi_theta_definition}:
\begin{equation}
\label{eq:psi_function_appendix}
\psi(\theta) = (1/2) \nabla^{3}f(\theta^*)(\theta - \thetas)^{\otimes 2}
\end{equation}

\begin{theorem}[Version of \Cref{th:RR_second_moment} with explicit constants]
\label{th:RR_second_moment_explicit_constants}
Assume \Cref{ass:mu-convex}, \Cref{ass:L-smooth}, \Cref{ass:rand_noise}($6$), and \Cref{ass:stationary_moments_bounds}($6$). Then for any $\gamma \in (0, 1/(\L \Conststep{6}) \wedge 2/(11 \L)]$, initial distribution $\nu$, and $n \in \nset$, the Richardson-Romberg estimator $\prtheta_{n}^{(RR)}$ defined in \eqref{eq:theta_RR_estimator} satisfies
\begin{align*}
\PE_{\nu}^{1/2}[\norm{\H(\prtheta_{n}^{(RR)} -\thetas)}^2] &\leq \frac{\sqrt{\trace{\noisecov}}}{n^{1/2}} + \frac{\ConstRR{1} \gamma^{1/2}}{n^{1/2}} + \frac{\ConstRR{2}}{\gamma^{1/2}n} + \ConstRR{3} \gamma^{3/2} + \frac{\ConstRR{4} \gamma }{n^{1/2}} \\ 
&\qquad \qquad + \Rem_4(n, \gamma, \norm{\theta_0-\thetas})\eqsp,
\end{align*}
where we have set 
\begin{equation}
\label{eq:const_rr_1_4_def}
\begin{split}
\ConstRR{1} &= \frac{c_0 \Constlast{4} \L \tau_4^2}{\mu^{3/2}} + \frac{c_0 \L \Constlast{2}^{1/2} \tau_2}{\mu^{1/2}}\eqsp, \quad 
\ConstRR{2} = \frac{c_0 \Constlast{2}^{1/2} \tau_2}{\mu^{1/2}} \\
\ConstRR{3} &= c_0 \left(\frac{\L \Constlast{6}^{3/2} \tau_6^3}{\mu^{3/2}} + \ConstPR{1}\right)\eqsp, \quad 
\ConstRR{4} = \frac{c_0 \Constlast{4} \L \tau_4^2}{\mu}\eqsp,
\end{split}
\end{equation}
and the remainder term $\Rem_4(n, \gamma, \norm{\theta_0-\thetas})$ is given by 
\begin{multline}
\label{eq:exponential_small_term_RR}
\Rem_4(n, \gamma, \norm{\theta_0-\thetas}) = \frac{c_0 \L (1-\gamma\mu)^{(n+1)/2}}{n\gamma \mu} \\
\times \left(\PE_{\nu}^{1/2}[\norm{\theta_0-\thetas}^6] + \PE_{\nu}^{1/2}[\norm{\theta_0-\thetas}^4] + \PE_{\nu}^{1/2}[\norm{\theta_0-\thetas}^2] + \frac{\Constlast{4} \gamma\tau_4^2}{\mu}\right)\eqsp.
\end{multline}
\end{theorem}
\begin{proof}
Using the recursion \eqref{eq:summ_parts}, we obtain that 
\begin{align}
\label{eq: decomposition_RR}
\H(\prtheta_{n}^{(RR)}-\thetas) = \frac{2(\theta_{n+1}^{(\gamma)}-\thetas)}{\gamma n} - \frac{2(\theta_{2n}^{(\gamma)} -\thetas)}{\gamma n} - \frac{\theta_{n+1}^{(2\gamma)} -\thetas}{2\gamma n} + \frac{\theta_{2n}^{(2\gamma)} -\thetas}{2\gamma n} \notag \\ - \frac{1}{n}\sum_{k=n+1}^{2n}[2\noise{k+1}(\theta_k^{(\gamma)})-\noise{k+1}(\theta_k^{(2\gamma)})] -\frac{1}{n}\sum_{k=n+1}^{2n}[2\eta(\theta_k^{(\gamma)}) - \eta(\theta_k^{(2\gamma)})]\eqsp.
\end{align}
Therefore, applying Minkowski's inequality to the decomposition \eqref{eq: decomposition_RR}, we obtain for any initial distribution $\nu$ that 
\begin{align*}
\PE_{\nu}^{1/2}[\norm{ \H (\prtheta_{n}^{(RR)}-\thetas)}^2] 
&\leq \underbrace{\frac{1}{n}\PE_{\nu}^{1/2}[\norm{\sum_{k=n+1}^{2n}\noise{k+1}(\thetas)}^2]}_{T_1} + \underbrace{\frac{2}{\gamma n}\PE_{\nu}^{1/2}[\norm{\theta_{n+1}^{(\gamma)}-\thetas}^2] + \frac{2}{\gamma n} \PE_{\nu}^{1/2}[\norm{\theta_{2n+1}^{(\gamma)}-\thetas}^2]}_{T_2} \\ 
&+ \underbrace{\frac{1}{2\gamma n}\PE_{\nu}^{1/2}[\norm{\theta_{n+1}^{(2\gamma)}-\thetas}^2] + \frac{1}{2\gamma n}\PE_{\nu}^{1/2}[\norm{\theta_{2n+1}^{(2\gamma)}-\thetas}^2]}_{T_3} \\ 
&+ \underbrace{\frac{2}{n}\PE_{\nu}^{1/2}[\norm{\sum_{k=n+1}^{2n}\noise{k+1}(\theta_k^{(\gamma)}) - \noise{k+1}(\thetas)}^2]}_{T_4} \\
&+  \underbrace{\frac{1}{n}\PE_{\nu}^{1/2}[\norm{\sum_{k=n+1}^{2n}\noise{k+1}(\theta_k^{(2\gamma)}) - \noise{k+1}(\thetas)}^2]}_{T_5} + \underbrace{\norm{2\pi_{\gamma}(\psi) - \pi_{2\gamma}(\psi)}}_{T_6}\\
& + \underbrace{\frac{2}{n}\PE_{\nu}^{1/2}[\norm{\sum_{k=n+1}^{2n}\eta(\theta_k^{(\gamma)}) - \pi_{\gamma}(\psi)}^2] + \frac{1}{n}\PE_{\nu}^{1/2}[\norm{\sum_{k=n+1}^{2n}\eta(\theta_k^{(2\gamma)}) - \pi_{2\gamma}(\psi)}^2]}_{T_7}\eqsp.
\end{align*}
Now we upper bound the terms in the right-hand side of the above bound separately. First, we note that 
\begin{align*}
T_1 = \frac{\sqrt{\trace{\noisecov}}}{\sqrt{n}}\eqsp.
\end{align*}
Using \Cref{ass:stationary_moments_bounds}($2$), we get
\begin{align*}
	T_2 + T_3 \lesssim \frac{(1-\gamma\mu)^{{(n+1)}/2}}{\gamma n}\PE_{\nu}^{1/2}[\norm{\theta_0-\thetas}^2] + \frac{\Constlast{2}^{1/2} \tau_2}{\mu^{1/2}\gamma^{1/2}n}\eqsp.
\end{align*}
Applying \Cref{lem:error_bound_second_moment}, we get 
\begin{align*}
T_4 + T_5 \lesssim \frac{\L \Constlast{2}^{1/2} \gamma^{1/2} \tau_2}{\mu^{1/2} n^{1/2}} + \frac{\L (1-\gamma\mu)^{(n+1)/2}}{\mu \gamma n} \PE_{\nu}^{1/2}\bigl[\norm{\theta_0-\thetas}^2\bigr]\eqsp.
\end{align*}
Now we proceed with the term $T_6$. Applying the recurrence \eqref{eq:sgd_recursion}, we obtain that 
\begin{equation}
\label{eq:one-step-expanded-recur}
\theta_{1}^{(\gamma)} - \thetas = (\Id - \gamma \H)(\theta_0^{(\gamma)} - \thetas) - \gamma \noise{1}(\theta_0^{(\gamma)}) -\gamma \eta(\theta_0^{(\gamma)})\eqsp.
\end{equation}
Thus, taking expectation w.r.t. $\pi_{\gamma}$ in both sides above, we get 
\[
\H (\bar{\theta}_{\gamma} - \thetas) = \PE_{\pi_{\gamma}}\bigl[\eta(\theta_0^{(\gamma)})\bigr] = \pi_{\gamma}(\psi) + \pi_{\gamma}(G)\eqsp,
\]
where $G(\theta)$ is defined in \eqref{eq:psi_theta_definition} and writes as
\[
G(\theta) =  \frac{1}{2} \left(\int_{0}^{1}t^2\nabla^{4}f(t\thetas + (1-t)\theta)\,dt\right) (\theta - \thetas)^{\otimes 3}\eqsp.
\]
Hence, applying \Cref{ass:L-smooth} together with \Cref{prop:bias_PR}, we obtain that
\begin{align}
\label{eq:T_6_bound}
T_6 = \norm{ 2\pi_{\gamma}(\psi) -  \pi_{2\gamma}(\psi)} \lesssim \ConstPR{1} \gamma^{3/2} \eqsp.
\end{align}
Finally, using \Cref{lem:eta_second_moment_bound}, \Cref{lem:arbitrary_init_second_moment_bound_residial_term}, and \Cref{lem:stationary_second_moment_bound_residial_term}, we obtain that 
\begin{align*}
T_7 & \lesssim \frac{\Constlast{4} \L \gamma \tau_4^2}{\mu n^{1/2}} +  \frac{\Constlast{4} \L \gamma^{1/2}\tau_4^2}{\mu^{3/2}n^{1/2}} + \frac{\L \Constlast{6}^{3/2} \gamma^{3/2}\tau_6^3}{\mu^{3/2}} \\
& \qquad + \frac{\L(1-\gamma\mu)^{(n+1)/2}}{n\gamma\mu}\left(\PE_{\nu}^{1/2}[\norm{\theta_0-\thetas}^6] + \PE^{1/2}_{\nu}[\norm{\theta_0-\thetas}^4] + \frac{\Constlast{4} \gamma \tau_4^2}{\mu}\right) \eqsp.
\end{align*}
Combining the bounds above completes the proof.
\end{proof}

Below we provide some auxiliary technical lemmas. 

\begin{lemma}
\label{lem:stationary_second_moment_bound_residial_term}
Assume \Cref{ass:mu-convex}, \Cref{ass:L-smooth}, \Cref{ass:rand_noise}($4$), and \Cref{ass:stationary_moments_bounds}($4$). Then for any $\gamma \in (0; 1/(\L \Conststep{4})]$ and any $n \in \nset$ it holds
\begin{equation}
n^{-1}\PE^{1/2}_{\pi_\gamma}\bigl[\norm{\sum_{k=n+1}^{2n}\{\psi(\theta_k) - \pi_{\gamma}(\psi)\}}^2\bigr] \lesssim  \frac{\Constlast{4} \L_3\gamma \tau_4^2}{\mu n^{1/2}} +  \frac{\Constlast{4} \L_3 \gamma^{1/2}\tau_4^2}{\mu^{3/2}n^{1/2}}\eqsp.
\end{equation}
\end{lemma}
\begin{proof}
Using the fact that $\pi_{\gamma}$ is a stationary distribution, we obtain that 
\begin{multline*}
\PE_{\pi_\gamma}\bigl[\norm{\sum_{k=n+1}^{2n}\{\psi(\theta_k) - \pi_{\gamma}(\psi)\}}^2\bigr] = n\PE_{\pi_{\gamma}}[\norm{\psi(\theta_0) - \pi_{\gamma}(\psi)}^2] \\
+ \sum_{k=1}^{n-1}(n-k)\PE_{\pi_{\gamma}}[(\psi(\theta_0) - \pi_{\gamma}(\psi))^T(\psi(\theta_{k}) - \pi_{\gamma}(\psi))]
\end{multline*}
Using the Markov property, Cauchy–Schwartz inequality, \Cref{lem:Wasserstein_ergodicity}, and \Cref{lem: Lipschitzness_remainder}, we obtain 
\begin{equation}
\label{eq: cov_decomposition}
\begin{split}
& \PE_{\pi_{\gamma}}[(\psi(\theta_0) - \pi_{\gamma}(\psi))^T(\psi(\theta_{k}) - \pi_{\gamma}(\psi))] \\
&\qquad = \PE_{\pi_{\gamma}}[(\psi(\theta_0) - \pi_{\gamma}(\psi))^T(\MK_{\gamma}^k\psi(\theta_0) - \pi_{\gamma}(\psi))]  \\ 
& \qquad \overset{(a)}{\lesssim} (1/2)^{k/m(\gamma)} \L_3 \PE_{\pi_{\gamma}}\bigl[\norm{\psi(\theta_0) - \pi_{\gamma}(\psi)}\int c(\theta_0, \vartheta)\rmd \pi_{\gamma}(  \vartheta)\bigr]\eqsp,
\end{split}
\end{equation}
where in (a) we additionally used the fact that
\[
\Wass{c}(\delta_{\theta_0}, \pi_{\gamma}) = \int c(\theta_0, \vartheta) \rmd \pi_{\gamma}(\vartheta)\eqsp.
\]
Using \Cref{ass:stationary_moments_bounds}($4$), we get 
\begin{equation}
    \label{eq: variance_bound}
    \PE_{\pi_{\gamma}}[\norm{\psi(\theta_0) - \pi_{\gamma}}^2] \leq \PE_{\pi_{\gamma}}[\norm{\psi(\theta_0)}^2] \leq \L_3^2 \PE_{\pi_{\gamma}}[\norm{\theta_0 - \thetas}^4] \leq
    \frac{\L_3^2 \Constlast{4} \gamma^2 \tau_{4}^4}{\mu^2}\eqsp,
\end{equation}
and, using \Cref{ass:stationary_moments_bounds}($2$) and \Cref{ass:stationary_moments_bounds}($4$), we get
\begin{align}
\label{eq: moment_cost_function_bound}
& \int \int c^2(\theta_0, \vartheta)d\pi_{\gamma}(\vartheta)d\pi_{\gamma}(\theta_0)  \\ & \leq \int \int \norm{\theta_0 -\vartheta}^2\left(\norm{\theta_0-\thetas} + \norm{\vartheta-\thetas} + \frac{2^{3/2}\gamma^{1/2}\tau_2}{\mu^{1/2}}\right)^2 d\pi_{\gamma}(\vartheta)d\pi_{\gamma}(\theta_0) \\ 
& \lesssim \int \int (\norm{\theta_0-\thetas}^4 + \norm{\vartheta-\thetas}^4) + \frac{\gamma\tau_2^2}{\mu}(\norm{\theta_0-\thetas}^2 + \norm{\vartheta-\thetas}^2) d\pi_{\gamma}(\vartheta)d\pi_{\gamma}(\theta_0) \\
& \lesssim \frac{\Constlast{4} \gamma^2\tau_4^2}{\mu^2} + \frac{\Constlast{2}\gamma^2\tau_2^4}{\mu^2} \lesssim \frac{\Constlast{4} \gamma^2\tau_4^4}{\mu^2}\eqsp.
\end{align}
Using \eqref{eq: variance_bound}, \eqref{eq: moment_cost_function_bound}, and Cauchy–Schwartz inequality for \eqref{eq: cov_decomposition}, we obtain 
\begin{align*}
\PE_{\pi_{\gamma}}[(\psi(\theta_0) - \pi_{\gamma}(\psi))^T(\psi(\theta_{k}) - \pi_{\gamma}(\psi))] \lesssim (1/2)^{k/m(\gamma)} \frac{\L_{3} \Constlast{4} \gamma^2\tau_4^4}{\mu^2}\eqsp.
\end{align*}
Combining the inequalities above and using that $m(\gamma) = \lceil 2 \frac{\log 4}{\gamma \mu}\rceil \leq \frac{2\log 4 + 1}{\gamma \mu}$, we get 
\begin{align*}
n^{-1}\PE^{1/2}_{\pi_\gamma}\bigl[\norm{\sum_{k=n+1}^{2n}\{\psi(\theta_k) - \pi_{\gamma}(\psi)\}}^2\bigr]  
&\leq  \left(\frac{\Constlast{4} \L_3^2\gamma^2 \tau_4^4}{\mu^2n} + \frac{\Constlast{4} m(\gamma)\L_3^{2} \gamma^2\tau_4^4}{\mu^2n}\right)^{1/2} \\
&\lesssim \frac{\Constlast{4} \L_3\gamma \tau_4^2}{\mu n^{1/2}} +  \frac{\Constlast{4} \L_3 \gamma^{1/2}\tau_4^2}{\mu^{3/2}n^{1/2}}\eqsp.
\end{align*}
\end{proof}

\begin{lemma}
\label{lem: coupling_4_moment}
Assume \Cref{ass:mu-convex}, \Cref{ass:L-smooth}, \Cref{ass:rand_noise}($4$). Then for any $\gamma \in (0; \frac{2}{11 \L}]$, and any $k \in \nset$ it holds that 
\begin{equation}
\PE[\norm{\theta_{k+1} - \tilde{\theta}_{k+1}}^4|\F_k] \leq (1-\gamma \mu)^2\norm{\theta_k-\tilde{\theta}_k}^4.
\end{equation}
Moreover, let $p \geq 2$ and assume  \Cref{ass:mu-convex}, \Cref{ass:L-smooth}, and \Cref{ass:rand_noise}($2p$). Then for any $\gamma \in (0, \frac{p}{4\cdot 3^p \L}]$ and any $k \in \nset$ it holds that 
\begin{equation}
    \PE[\norm{\theta_{k+1} - \tilde{\theta}_{k+1}}^{2p}|\F_k] \leq (1-\gamma\mu)^p\norm{\theta_k -\tilde{\theta}_k}^{2p}\eqsp.
\end{equation}
\end{lemma}
\begin{proof}
Recall that the sequences $\{\theta_k\}_{k \in \nset}$ and $\{\tilde{\theta}_k\}_{k \in \nset}$ are defined by the recurrences 
\begin{align}
\theta_{k+1} &= \theta_{k} - \gamma \nabla F(\theta_{k}, \xi_{k+1})\eqsp, \quad \theta_{0} = \theta \in \rset^{d}\eqsp, \\
\tilde{\theta}_{k+1} &= \tilde{\theta}_{k} - \gamma \nabla F(\tilde{\theta}_{k}, \xi_{k+1})\eqsp, \quad \tilde{\theta}_{0} = \tilde{\theta} \in \rset^{d}\eqsp.
\end{align}

Expanding the brackets, we obtain that 
\begin{align*}
&\norm{\theta_{k+1} - \tilde{\theta}_{k+1}}^4 = \norm{\theta_k-\tilde{\theta}_k}^4 + \gamma^4\norm{\nabla F(\theta_k, \xi_{k+1}) - \nabla F(\tilde{\theta}_k, \xi_{k+1})}^4 \\& + 4\gamma^2\langle \nabla F(\theta_k, \xi_{k+1}) - \nabla F(\tilde{\theta}_k, \xi_{k+1}), \theta_k - \tilde{\theta}_k \rangle^2  \\& + 2\gamma^2 \norm{\nabla F(\theta_k, \xi_{k+1}) - \nabla F(\tilde{\theta}_k, \xi_{k+1})}^2\norm{\theta_k - \tilde{\theta}_k}^2 \\& -4\gamma\langle \nabla F(\theta_k, \xi_{k+1}) - \nabla F(\tilde{\theta}_k, \xi_{k+1}), \theta_k - \tilde{\theta}_k \rangle\norm{\theta_k-\tilde{\theta}_k}^2 \\& - 4\gamma^3\langle \nabla F(\theta_k, \xi_{k+1}) - \nabla F(\tilde{\theta}_k, \xi_{k+1}), \theta_k - \tilde{\theta}_k \rangle \norm{\nabla F(\theta_k, \xi_{k+1}) - \nabla F(\tilde{\theta}_k, \xi_{k+1})}^2
\end{align*}
Using \Cref{ass:rand_noise}($4$) and Cauchy–Schwartz inequality, we get
\begin{align*}
\PE[\norm{\nabla F(\theta_k, \xi_{k+1}) - \nabla F(\tilde{\theta}_k, \xi_{k+1})}^4| \F_k]  &\leq \L^3 \langle \nabla f (\theta_k) - \nabla f (\tilde{\theta}_k), \theta_k - \tilde{\theta}_k \rangle \norm{\theta_k - \tilde{\theta}_k}^2, \\
\PE[\langle \nabla F(\theta_k, \xi_{k+1}) - \nabla F(\tilde{\theta}_k, \xi_{k+1}), \theta_k - \tilde{\theta}_k \rangle^2| \F_k]  &\leq \L\langle \nabla f (\theta_k) - \nabla f (\tilde{\theta}_k), \theta_k - \tilde{\theta}_k \rangle \norm{\theta_k - \tilde{\theta}_k}^2, \\
\PE[\norm{\nabla F(\theta_k, \xi_{k+1}) - \nabla F(\tilde{\theta}_k, \xi_{k+1})}^2\norm{\theta_k - \theta'_k}^2|\F_k ]  &\leq  \L \langle \nabla f (\theta_k) - \nabla f (\tilde{\theta}_k), \theta_k - \tilde{\theta}_k \rangle \norm{\theta_k - \tilde{\theta}_k}^2 \\  
\PE[\langle \nabla F(\theta_k, \xi_{k+1}) - \nabla F(\tilde{\theta}_k, \xi_{k+1}), \theta_k - \tilde{\theta}_k \rangle\norm{\theta_k-\tilde{\theta}_k}^2| \F_k ] &= \langle \nabla f (\theta_k) - \nabla f (\tilde{\theta}_k), \theta_k - \tilde{\theta}_k \rangle \norm{\theta_k - \tilde{\theta}_k}^2\eqsp.
\end{align*}
Similarly, 
\begin{align*}
&\PE[\langle \nabla F(\theta_k, \xi_{k+1}) - \nabla F(\tilde{\theta}_k, \xi_{k+1}), \theta_k - \tilde{\theta}_k \rangle \norm{\nabla F(\theta_k, \xi_{k+1}) - \nabla F(\tilde{\theta}_k, \xi_{k+1})}^2|\F_k] \\
&\qquad \qquad \qquad \leq \L^2 \langle \nabla f (\theta_k) - \nabla f (\tilde{\theta}_k), \theta_k - \tilde{\theta}_k \rangle \norm{\theta_k - \tilde{\theta}_k}^2
\end{align*}
Combining all inequalities above, we obtain 
\begin{multline*}
\PE[\norm{\theta_{k+1} - \theta'_{k+1}}^4|\F_k] \leq \norm{\theta_k-\tilde{\theta}_k}^4 \\
- (4\gamma - \gamma^4 \L^3 -4\gamma^2 \L  - 2\gamma^2 \L -4\gamma^3 \L^2) \langle \nabla f (\theta_k) - \nabla f (\tilde{\theta}_k), \theta_k - \tilde{\theta}_k \rangle \norm{\theta_k - \tilde{\theta}_k}^2
\end{multline*}
Using \Cref{ass:mu-convex} and since $1-\gamma^3 \L^3/4 - 3\gamma \L/2-\gamma^2 \L^2\geq 1- 11\gamma \L/4$, we get 
\begin{align*}
\PE[\norm{\theta_{k+1} - \tilde{\theta}_{k+1}}^4|\F_k] 
&\leq (1-4\gamma\mu(1 - 11\gamma \L/4))\norm{\theta_k-\tilde{\theta}_k}^4 \\ 
& \leq (1-2\gamma\mu(1 - 11\gamma \L/4))^2 \norm{\theta_k-\tilde{\theta}_k}^4\eqsp.
\end{align*}
Since $1 - 11\gamma \L/4 \geq 1/2$ for $\gamma \leq 2/(11\L)$, we complete the proof.

For simplicity of proof for second part of lemma we define $\delta_{k+1} = \norm{\theta_{k+1} - \theta'_{k+1}}$. Then we have 
\begin{align*}
    &\PE[\delta_{k+1}^{2p}|\F_k] \\&= \PE[(\delta_k^2 -2\gamma\langle \nabla F(\theta_k, \xi_{k+1}) - \nabla F(\tilde{\theta}_k, \xi_{k+1}), \theta_k - \tilde{\theta}_k \rangle + \gamma^2 \norm{\nabla F(\theta_k, \xi_{k+1}) - \nabla F(\tilde{\theta}_k, \xi_{k+1})}^2)^{p}|\F_k] \\&= 
\PE[\underset{\mathclap{\substack{i+j+l=p;\\ i,j,l \in \{0,\ldots p\}}}}{\sum}\frac{p!}{i!j!l!}\delta_k^{2i}(-2\gamma\langle \nabla F(\theta_k, \xi_{k+1}) - \nabla F(\tilde{\theta}_k, \xi_{k+1}), \theta_k - \tilde{\theta}_k \rangle)^j\gamma^{2l}\norm{\nabla F(\theta_k, \xi_{k+1}) - \nabla F(\tilde{\theta}_k, \xi_{k+1})}^{2l}|\F_k]\eqsp.
\end{align*}
Now we bound each term in the sum above.
\begin{enumerate}
    \item First, for $i = p, j = 0, l=0$ the corresponding term in the sum is equal to $\delta_k^{2p}$.
    \item Second, for $i = p-1, j = 1, l = 0$, we have
    \begin{equation*}
        \PE[(-2\gamma\langle \nabla F(\theta_k, \xi_{k+1}) - \nabla F(\tilde{\theta}_k, \xi_{k+1}), \theta_k - \tilde{\theta}_k \rangle)|\F_k] = -2\gamma\langle\nabla f(\theta_k) - \nabla f(\tilde{\theta}_k),  \theta_k - \tilde{\theta}_k \rangle\eqsp.
    \end{equation*}
    \item Third, for $l\geq 1$ or $j\geq 2$ we use Cauchy-Schwartz inequality and get 
    \begin{align*}
        &(2\gamma\langle \nabla F(\theta_k, \xi_{k+1}) - \nabla F(\tilde{\theta}_k, \xi_{k+1}), \theta_k - \tilde{\theta}_k \rangle)^j\gamma^{2l}\norm{\nabla F(\theta_k, \xi_{k+1}) - \nabla F(\tilde{\theta}_k, \xi_{k+1})}^{2l}\\&\qquad \qquad \leq 2^j\gamma^{j+2l}\delta_k^j\norm{\nabla F(\theta_k, \xi_{k+1}) - \nabla F(\tilde{\theta}_k, \xi_{k+1})}^{2l+j}\eqsp.
    \end{align*}
    Moreover using \Cref{ass:rand_noise}($2p$), we get 
    \begin{align*}
&\PE[2^j\gamma^{j+2l}\delta_k^{2i+j}\norm{\nabla F(\theta_k, \xi_{k+1}) - \nabla F(\tilde{\theta}_k, \xi_{k+1})}^{2l+j}|\F_k] \\& \leq 2^j\gamma^{j+2l}\delta_k^{2p-2}L^{2l+j-1}\langle \nabla f(\theta_k) - \nabla f(\tilde{\theta}_k),  \theta_k - \tilde{\theta}_k \rangle\eqsp.
    \end{align*}
\end{enumerate}
Combining all inequalities above, we obtain 
\begin{multline*}
    \PE[\delta_{k+1}^{2p}|\F_k] \leq \delta_k^{2p}-2p\gamma \langle\nabla f(\theta_k) - \nabla f(\tilde{\theta}_k),  \theta_k - \tilde{\theta}_k \rangle\delta_k^{2p-2} \\+ \bigl(\underset{\mathclap{\substack{i+j+l=p;\\ i,j,l \in \{0,\ldots p\}\\ j+2l \geq 2}}}{\sum}\frac{p!}{i!j!l!} 2^j\gamma^{j+2l}L^{2l+j-1}\bigr)\langle \nabla f(\theta_k) - \nabla f(\tilde{\theta}_k),  \theta_k - \tilde{\theta}_k \rangle\delta_k^{2p-2}\eqsp.
\end{multline*}
 Since $\gamma \leq \frac{p}{3^p4L}$, we have 
 \begin{equation*}
    \PE[\delta_{k+1}^{2p}|\F_k] \leq \delta_k^{2p}-\gamma p\langle \nabla f(\theta_k) - \nabla f(\tilde{\theta}_k),  \theta_k - \tilde{\theta}_k \rangle\delta_k^{2p-2}\eqsp.
\end{equation*}
It remains to apply \Cref{ass:mu-convex} together with an elementary bound $(1-p\mu\gamma) \leq (1-\gamma\mu)^p$. 
\end{proof}
\begin{lemma}
\label{lem:arbitrary_init_second_moment_bound_residial_term}
Assume \Cref{ass:mu-convex}, \Cref{ass:L-smooth}, \Cref{ass:rand_noise}($4$), and \Cref{ass:stationary_moments_bounds}($4$). Then for any $\gamma \in (0;1/(\L \Conststep{4})]$, any $n \in \nset$ and initial distribution $\nu$ it holds 
\begin{align*}
n^{-1}\PE^{1/2}_{\nu}[\norm{\sum_{k=n+1}^{2n}\{\psi(\theta_k) - \pi_{\gamma}(\psi)\}}^2]  
&\lesssim  n^{-1} \PE^{1/2}_{\pi_{\gamma}}[\norm{\sum_{k=n+1}^{2n}\{\psi(\theta_k) - \pi_{\gamma}(\psi)\}}^2] \\
& \qquad + \frac{\L_3(1-\gamma\mu)^{(n+1)/2}}{n\gamma\mu}\left(\PE^{1/2}_{\nu}[\norm{\theta_0-\thetas}^4] + \frac{\Constlast{4} \gamma \tau_4^2}{\mu}\right)\eqsp.
\end{align*}
\end{lemma}
\begin{proof}
Using the synchronous coupling construction defined in \eqref{eq:kernel_coupling_construction} and the corresponding coupling kernel $\MKK_{\gamma}$, we obtain that
\begin{equation}
\begin{aligned}
\label{eq: coupling_arbitrary_eta}
&\PE_{\nu}^{1/2}[\norm{\sum_{k=n+1}^{2n}\{\psi(\theta_k) - \pi_{\gamma}(\psi)\}}^2] =  (\PE^{\MKK_{\gamma}}_{\nu, \pi_{\gamma}}[\norm{\sum_{k=n+1}^{2n}\{\psi(\theta_k) - \pi_{\gamma}(\psi)\}}^2])^{1/2} \\ &\quad \quad \leq \PE^{1/2}_{\pi_{\gamma}}[\norm{\sum_{k=n+1}^{2n}\{\psi(\tilde{\theta}_k) - \pi_{\gamma}(\psi)\}}^2] + (\PE^{\MKK_{\gamma}}_{\nu, \pi_{\gamma}}[\norm{\sum_{k=n+1}^{2n}\{\psi(\theta_k) - \psi(\tilde{\theta}_k)\}}^2])^{1/2}
\end{aligned}
\end{equation}
Applying Minkowski's inequality to the last term and using \Cref{lem: Lipschitzness_remainder}, we get 
\begin{align*}
(\PE^{\MKK_{\gamma}}_{\nu, \pi_{\gamma}}[\norm{\sum_{k=n+1}^{2n}\{\psi(\theta_k) - \psi(\tilde{\theta}_k)\}}^2])^{1/2} &\leq \sum_{k=n+1}^{2n}(\PE^{K}_{\nu, \pi_{\gamma}}[\norm{\{\psi(\theta_k) - \psi(\tilde{\theta}_k)\}}^2])^{1/2} \\ 
& \leq \frac{\L_3}{2}\sum_{k=n+1}^{2n}(\PE^{\MKK_{\gamma}}_{\nu, \pi_{\gamma}}[c^2(\theta_k, \tilde{\theta}_k)])^{1/2}\eqsp.
\end{align*}
Using H\"older's and Minkowski’s inequality and applying \Cref{lem: coupling_4_moment} , \eqref{eq:last_iterate_2_moment} and \eqref{eq:last_iterate_4_moment}, we obtain 
\begin{align*}
&(\PE^{\MKK_{\gamma}}_{\nu, \pi_{\gamma}}[c^2(\theta_k, \tilde{\theta}_k)])^{1/2} \\
& \qquad \leq (\PE^{\MKK_{\gamma}}_{\nu, \pi_{\gamma}}[\norm{\theta_k- \tilde{\theta}_k}^4])^{1/4}\bigl(\PE_{\pi_{\gamma}}^{1/4}[\norm{\tilde{\theta}_k -\thetas}^4] + \PE_{\nu}^{1/4}[\norm{\theta_k -\thetas}^4 + \frac{\gamma^{1/2}\tau_2}{\mu^{1/2}}]\bigr) \\ 
& \qquad \leq (1-\gamma\mu)^{k/2}(\PE^{\MKK_{\gamma}}_{\nu, \pi_{\gamma}}[\norm{\theta_0- \tilde{\theta}_0}^4])^{1/4}(\PE_{\nu}^{1/4}[\norm{\theta_0-\thetas}^4] + \frac{\Constlast{4}^{1/2}\gamma^{1/2}\tau_4}{\mu^{1/2}}+ \frac{\gamma^{1/2}\tau_2}{\mu^{1/2}}) \\
&\qquad \lesssim (1-\gamma\mu)^{k/2}\left(\frac{\Constlast{4} \gamma\tau_4^2}{\mu} + \PE_{\nu}^{1/2}\norm{\theta_0-\thetas}^4\right)
\end{align*}
Combining all inequalities above, we get
\begin{equation*}
(\PE^{\MKK_{\gamma}}_{\nu, \pi_{\gamma}}[\norm{\sum_{k=n+1}^{2n}\{\psi(\theta_k) - \psi(\theta_k')\}}^2])^{1/2} \lesssim \frac{\L_3(1-\gamma\mu)^{(n+1)/2}}{\gamma\mu}\left(\PE^{1/2}_{\nu}[\norm{\theta_0-\thetas}^4] + \frac{\Constlast{4} \gamma\tau_4^2}{\mu}\right)\eqsp.
\end{equation*}
Substituting the last inequality into \eqref{eq: coupling_arbitrary_eta} we complete the proof.
\end{proof}

\begin{lemma}
\label{lem:eta_second_moment_bound}
Assume \Cref{ass:mu-convex}, \Cref{ass:L-smooth}, \Cref{ass:rand_noise}($6$), and \Cref{ass:stationary_moments_bounds}($6$). Then for any $\gamma \in (0;1/(\L \Conststep{6})]$, $n \in \nset$, and initial distribution $\nu$, it holds that 
\begin{equation}
\begin{split}
n^{-1}\PE_{\nu}^{1/2}[\sum_{k=n+1}^{2n}\norm{\eta(\theta_k) - \pi_{\gamma}(\psi)}^2] 
&\leq n^{-1} \PE_{\nu}^{1/2}\bigl[\sum_{k=n+1}^{2n}\norm{\psi(\theta_k) - \pi_{\gamma}(\psi)}^2\bigr] \\
& + \frac{\L_4 (1-\gamma\mu)^{(n+1)/2}}{n \gamma\mu}\PE_{\nu}^{1/2}[\norm{\theta_0-\thetas}^6] + \frac{\L_4 \Constlast{6}^{3/2} \gamma^{3/2}\tau_6^3}{3\mu^{3/2}}\eqsp.
\end{split}
\end{equation}
\end{lemma}
\begin{proof}
Applying the 4-rd order Taylor expansion with integral remainder, we get that
\begin{equation}
\label{eq: eta_taylor}
\eta(\theta) = \psi(\theta) + \frac{1}{2}\left(\int_{0}^{1}t^2 \nabla^{4}f(t\thetas + (1-t)\theta)\,dt\right) (\theta - \thetas)^{\otimes 3}\eqsp,
\end{equation} 
and using \Cref{ass:L-smooth}, we obtain
\begin{equation}
\label{eq: bound_taylor_remainder}
(1/2) \norm{\left(\int_{0}^{1}t^2 \nabla^{4}f(t\thetas + (1-t)\theta)\,dt\right) (\theta - \thetas)^{\otimes 3}}\leq \L_4\norm{\theta-\thetas}^3\eqsp.
\end{equation}
Therefore, combining \eqref{eq: eta_taylor}, \Cref{ass:L-smooth}, and applying Minkowski's inequality, we get 
\begin{multline}
\label{eq: eta_decomposotion} 
\PE_{\nu}^{1/2}\bigl[\sum_{k=n+1}^{2n}\norm{\eta(\theta_k)-\pi_{\gamma}(\psi)}^2\bigr] \leq  \PE_{\nu}^{1/2}\bigl[\sum_{k=n+1}^{2n}\norm{\psi(\theta_k) - \pi_{\gamma}(\psi)}^2\bigr] \\ 
+ \frac{\L_4}{6}\sum_{k=n+1}^{2n}\PE_{\nu}^{1/2}[\norm{\theta_k -\thetas}^6]
\end{multline}
Applying \Cref{ass:stationary_moments_bounds}($6$) for the last term of \eqref{eq: eta_decomposotion}, we get 
\begin{equation}
\begin{aligned}
\PE_{\nu}^{1/2}\bigl[\sum_{k=n+1}^{2n}\norm{\eta(\theta_k) - \pi_{\gamma}(\psi)}^2\bigr] &\lesssim \PE_{\nu}^{1/2}\bigl[\sum_{k=n+1}^{2n}\norm{\psi(\theta_k) - \pi_{\gamma}(\psi)}^2\bigr] + \frac{\L_4 n \Constlast{6}^{3/2}\gamma^{3/2}\tau_6^3}{\mu^{3/2}} \\ & \qquad + \frac{\L_4 (1-\gamma\mu)^{3(n+1)/2}}{1-(1-\gamma\mu)^{3/2}}\PE_{\nu}^{1/2}[\norm{\theta_0-\thetas}^6]\eqsp.
\end{aligned}
\end{equation}
In remains to notice that $(1-\gamma\mu)^{3/2} \leq (1-\gamma\mu)$, and the statement follows.
\end{proof}

We conclude this section with a technical statement on the properties of the function $\psi$ from \eqref{eq:psi_function_appendix}. 
\begin{lemma}
\label{lem: Lipschitzness_remainder}
Let $\psi(\cdot)$ be a function defined in \eqref{eq:psi_function_appendix}. Then for any $\theta, \theta' \in \rset^d$, it holds that
\begin{equation*}
\norm{\psi(\theta) - \psi(\theta')} \leq \frac{1}{2} \L_3 c(\theta,\theta').
\end{equation*}
\end{lemma}
\begin{proof}
     For simplicity, let us denote $T = \nabla^{3}f(\theta^*)$. 
Hence, 
\begin{equation}
     \norm{\psi(\theta) - \psi(\theta')} \leq \frac{1}{2}\norm{T(\theta - \thetas)^{\otimes 2} - T(\theta' - \thetas)^{\otimes 2}}.
\end{equation}
Note that 
\begin{equation}
\norm{T} = \underset{x \neq 0, y \neq 0, z \neq 0}{\sup} \frac{\underset{i, j, k}{\sum}{T_{ijk}x_iy_jz_k}}{\norm{x}\norm{y}\norm{z}} \geq \underset{x\neq 0, y \neq 0}{\sup}\underset{z\neq 0}{\sup}\frac{\underset{k}{\sum} z_k \underset{i,j}{\sum}T_{ijk}x_iy_j}{\norm{z}\norm{y}\norm{x}} = \underset{x\neq 0, y \neq 0}{\sup}\frac{\norm{t(x,y)}}{\norm{y}\norm{x}},
\end{equation}
where $t(x,y)_k = \underset{i, j}{\sum}T_{ijk}x_iy_j$.
Therefore, for any $x,y \in \rset^d$, it holds that 
\begin{equation}
    \label{eq: tensor_norm_inequality}
    \norm{t(x,y)} \leq \norm{x}\norm{y}\norm{T}
\end{equation}
We denote $v = Tx^{\otimes 2} - Ty^{\otimes 2}$. Then 
\begin{gather}
    \label{eq: decomposition_lipschts_part}
   v_k = \underset{i,j}{\sum}T_{ijk}(x_ix_j - y_iy_j) = \underset{i,j}{\sum}T_{ijk}((x_i-y_i)x_j + (x_i-y_i)y_j) = \notag\\ \underset{i,j}{\sum}T_{ijk}(x_i-y_i)x_j + \underset{i,j}{\sum}T_{ijk}(x_i-y_i)y_j,
\end{gather}
where the first inequality is true since $T_{ijk}=T_{jik}$ by definition of $T$.
Combining  \eqref{eq: tensor_norm_inequality} and \eqref{eq: decomposition_lipschts_part} and using triangle inequality, we obtain 
\begin{equation*}
    \norm{v} \leq \norm{T}\norm{x-y}(\norm{x}+\norm{y}) \leq \norm{T}\norm{x-y}(\norm{x}+\norm{y}+\frac{2\sqrt{2}
    \tau_2\sqrt{\gamma}}{\sqrt{\mu}}).
\end{equation*}
We complete the proof setting $x = \theta - \thetas, y = \theta'-\thetas$
\end{proof}

\section{Proof of \Cref{th:p_moment_RR}}
\label{th:RR_pth_moment_proof}
\begin{theorem}[Version of \Cref{th:p_moment_RR} with explicit constants]
\label{th:RR_pth_moment_proof_explicit}
Let $p \geq 2$ and assume \Cref{ass:mu-convex}, \Cref{ass:L-smooth}, \Cref{ass:rand_noise}($3p$), and \Cref{ass:stationary_moments_bounds}($3p$). Then for any $\gamma \in (0, 1/(\L \Conststep{3p}) \wedge p/(4 \cdot 3^p \L)]$, initial distribution $\nu$, and $n \in \nset$, the estimator $\prtheta_{n}^{(RR)}$ defined in \eqref{eq:theta_RR_estimator} satisfies
\begin{align*}
\PE_{\nu}^{1/p}[\norm{\H(\prtheta_{n}^{(RR)} -\thetas)}^p] 
&\leq \frac{c_1 \sqrt{\trace{\noisecov}}p^{1/2}}{n^{1/2}} + \frac{c_2 p \tau_p}{n^{1-1/p}} + \frac{\ConstRR{5}}{n \gamma^{1/2}} + \frac{\ConstRR{6} \gamma^{1/2}}{n^{1/2}} + \ConstRR{7} \gamma^{3/2} \\ 
&\qquad \qquad + \frac{\ConstRR{8}}{n} + \Rem_5(n, \gamma, \norm{\theta_0-\thetas})\eqsp,
\end{align*}
where we have set 
\begin{equation}
\label{eq:const_rr_5_8_def}
\begin{split}
\ConstRR{5} &= \frac{c_0 \Constlast{p}^{1/2} \tau_{p}}{\mu^{1/2}} \eqsp, \quad \ConstRR{6} = \frac{c_0 \L \Constlast{p}^{1/2} p \tau_{p}}{\mu^{1/2}} + \frac{c_0 \L \Constlast{2p} p  \tau^2_{2p}}{\mu^{3/2}} \eqsp, \\
\ConstRR{7} &= c_0 \left(\ConstPR{1} + \frac{ \L \Constlast{3p}^{3/2}  \tau_{3p}^3}{\mu^{3/2}}\right) \eqsp, \quad \ConstRR{8} = \frac{c_0 \L \Constlast{2p} \tau_{2p}}{\mu^2}\eqsp,
\end{split}
\end{equation}
and the remainder term $\Rem_5(n, \gamma, \norm{\theta_0-\thetas})$ is given by 
\begin{multline}
\label{eq:exponential_small_term_RR_pth_moment}
\Rem_5(n, \gamma, \norm{\theta_0-\thetas}) = \frac{c_0 (1 - \gamma \mu)^{(n+1)/2}}{\gamma n} \PE_{\nu}^{1/p}\bigl[\norm{\theta_0 - \thetas}^{p}\bigr] + \frac{c_0 \L p (1-\gamma \mu)^{(n+1)/2}}{\mu^{1/2} \gamma^{1/2} n} \PE^{1/p}_{\nu}[\norm{\theta_0 - \thetas}^{p}] \\
+ \frac{c_0\L(1-\gamma\mu)^{(n+1)/2}}{\gamma\mu n}\left(\PE^{1/p}_{\nu}[\norm{\theta_0-\thetas}^{2p}] + \frac{\Constlast{2p} \gamma\tau_{2p}^2}{\mu}\right)+ \frac{c_0 \L (1-\gamma \mu)^{(3/2)n}}{\gamma \mu } \PE_{\nu}^{1/p}\bigl[\norm{\theta_0-\thetas}^{3p}\bigr]
\end{multline}
\end{theorem}
\begin{proof}
Using the decomposition \eqref{eq: decomposition_RR}, we obtain that for any $p \geq 2$, it holds that 
\begin{align*}
\PE_{\nu}^{1/p}[\norm{ \H (\prtheta_{n}^{(RR)}-\thetas)}^p] 
&\lesssim \underbrace{\frac{1}{n}\PE_{\nu}^{1/p}[\norm{\sum_{k=n+1}^{2n}\noise{k+1}(\thetas)}^p]}_{T_1} + \underbrace{\frac{1}{\gamma n}\PE_{\nu}^{1/p}[\norm{\theta_{n+1}^{(\gamma)}-\thetas}^p] + \frac{1}{\gamma n} \PE_{\nu}^{1/p}[\norm{\theta_{2n+1}^{(\gamma)}-\thetas}^p]}_{T_2} \\ 
&+ \underbrace{\frac{1}{\gamma n}\PE_{\nu}^{1/p}[\norm{\theta_{n+1}^{(2\gamma)}-\thetas}^p] + \frac{1}{\gamma n}\PE_{\nu}^{1/p}[\norm{\theta_{2n+1}^{(2\gamma)}-\thetas}^p]}_{T_3} \\ 
&+ \underbrace{\frac{1}{n}\PE_{\nu}^{1/p}[\norm{\sum_{k=n+1}^{2n}\noise{k+1}(\theta_k^{(\gamma)}) - \noise{k+1}(\thetas)}^p]}_{T_4} \\
&+  \underbrace{\frac{1}{n}\PE_{\nu}^{1/p}[\norm{\sum_{k=n+1}^{2n}\noise{k+1}(\theta_k^{(2\gamma)}) - \noise{k+1}(\thetas)}^p]}_{T_5} + \underbrace{\norm{2\pi_{\gamma}(\psi) - \pi_{2\gamma}(\psi)}}_{T_6}\\
& + \underbrace{\frac{1}{n}\PE_{\nu}^{1/p}[\norm{\sum_{k=n+1}^{2n}\psi(\theta_k^{(\gamma)}) - \pi_{\gamma}(\psi)}^p] + \frac{1}{n}\PE_{\nu}^{1/p}[\norm{\sum_{k=n+1}^{2n}\psi(\theta_k^{(2\gamma)}) - \pi_{2\gamma}(\psi)}^p]}_{T_7} \\
& + \underbrace{\frac{1}{n}\sum_{k=n+1}^{2n}\PE^{1/p}_{\nu}\bigl[\norm{G(\theta_k^{(\gamma)})}^{p}\bigr] + \frac{1}{n}\sum_{k=n+1}^{2n}\PE^{1/p}_{\nu}\bigl[\norm{G(\theta_k^{(2\gamma)})}^{p}\bigr]}_{T_8}\eqsp.
\end{align*}
Now we upper bounds the terms above separately. Applying first the Pinelis version of Rosenthal inequality \citep[Theorem 4.1]{pinelis_1994} together with \Cref{ass:rand_noise}($p$), we obtain that 
\begin{align*}
T_1 \leq \frac{c_1 \sqrt{\trace{\noisecov}}p^{1/2}}{n^{1/2}} + \frac{c_2 p \tau_p}{n^{1-1/p}}\eqsp.
\end{align*}
Applying \Cref{ass:stationary_moments_bounds}($p$) (which is implied by \Cref{ass:stationary_moments_bounds}($3p$)), we obtain that 
\begin{align*}
T_2 + T_3 \lesssim \frac{\Constlast{p}^{1/2} \tau_{p}}{\mu^{1/2} n \gamma^{1/2}} + \frac{(1 - \gamma \mu)^{(n+1)/2}}{\gamma n} \PE_{\nu}^{1/p}\bigl[\norm{\theta_0 - \thetas}^{p}\bigr]\eqsp.
\end{align*}
Applying \Cref{lem:error_bound_second_moment} (see the bound \eqref{eq:p_th_moment_bound_eps}), we get that 
\begin{align*}
T_4 + T_5 \lesssim \frac{\L \Constlast{p}^{1/2} \gamma^{1/2} p \tau_{p}}{\mu^{1/2} n^{1/2}} + \frac{\L p (1-\gamma \mu)^{(n+1)/2}}{\mu^{1/2} \gamma^{1/2} n} \PE^{1/p}_{\nu}[\norm{\theta_0 - \thetas}^{p}] \eqsp.
\end{align*}
Using the bounds \eqref{eq:one-step-expanded-recur} and \eqref{eq:T_6_bound}, we obtain 
\begin{align*}
T_6 \lesssim \ConstPR{1} \gamma^{3/2}\eqsp.
\end{align*}
Applying \Cref{prop:psi_p_moment_bound}, we get 
\begin{align*}
\frac{1}{n}\PE_{\pi_\gamma}^{1/p}[\norm{\sum_{k=n+1}^{2n}\psi(\theta_k^{(\gamma)}) - \pi_{\gamma}(\psi)}^p] \lesssim \frac{\L \Constlast{2p} p  \tau^2_{2p} \gamma^{1/2}}{\mu^{3/2} n^{1/2}} +  \frac{\L \Constlast{2p} \tau_{2p}}{\mu^2 n}\eqsp.
\end{align*}
Using this bound and \Cref{lem:Generalization_Rosenthal}, we obtain that 
\begin{align*}
T_7 \lesssim \frac{\L \Constlast{2p} p  \tau^2_{2p} \gamma^{1/2}}{\mu^{3/2} n^{1/2}} +  \frac{\L \Constlast{2p} \tau_{2p}}{\mu^2 n} + \frac{\L(1-\gamma\mu)^{(n+1)/2}}{\gamma\mu n}\left(\PE^{1/p}_{\nu}[\norm{\theta_0-\thetas}^{2p}] + \frac{\Constlast{2p} \gamma\tau_{2p}^2}{\mu}\right)
\end{align*}
Finally, applying the definition of $G(\theta)$ in \eqref{eq:psi_theta_definition} together with \Cref{ass:stationary_moments_bounds}($3p$), we obtain that 
\begin{align*}
T_8 
&\lesssim \frac{\L \Constlast{3p}^{3/2} \gamma^{3/2} \tau_{3p}^3}{\mu^{3/2}} + \frac{\L}{n} \sum_{k=n+1}^{2n} (1-\gamma \mu)^{(3/2)k} \PE_{\nu}^{1/p}\bigl[\norm{\theta_0-\thetas}^{3p}\bigr] \\
&\lesssim \frac{\L \Constlast{3p}^{3/2} \gamma^{3/2} \tau_{3p}^3}{\mu^{3/2}} + \frac{\L (1-\gamma \mu)^{(3/2)n}}{\gamma \mu } \PE_{\nu}^{1/p}\bigl[\norm{\theta_0-\thetas}^{3p}\bigr]\eqsp.
\end{align*}
To complete the proof it remains to combine the bounds for $T_1$ to $T_8$.
\end{proof}

\subsection{Proof of \Cref{prop:psi_p_moment_bound}}
\label{sec:psi_p_moment_bound_proof}
In the proof below we use the notation 
\[
\barpsi(\theta) = \psi(\theta) - \pi_{\gamma}(\psi)\eqsp.
\]
We proceed with the blocking technique. Indeed, let us set the parameter 
\begin{equation}
\label{eq:m_gamma_def}
m = m(\gamma) = \biggl \lceil \frac{2 \log 4}{\gamma \mu} \biggr \rceil \eqsp.
\end{equation}
Our choice of parameter $m(\gamma)$ is due to \Cref{lem:Wasserstein_ergodicity}. For notation conciseness we write it simply as $m$, dropping its dependence upon $\gamma$. Using Minkowski's inequality, we obtain that
\begin{equation}
\label{eq:whole_part}
\PE_{\pi_{\gamma}}^{1/p}\bigl[\norm{\sum_{k=0}^{n-1}\barpsi(\theta_k)}^p\bigr] \leq  \PE_{\pi_{\gamma}}^{1/p}\bigl[\norm{\sum_{k=0}^{\lfloor n/m \rfloor m-1}\barpsi(\theta_k)}^p\bigr] + m \PE_{\pi_{\gamma}}^{1/p}\bigl[\norm{\barpsi(\theta_0)}^p\bigr] \eqsp.
\end{equation}
Now we consider the Poisson equation, associated with $\MKQ_{\gamma}^{m}$ and function $\barpsi$, that is, 
\begin{equation}
\label{eq:Poisson_eq_iterates}
g_m(\theta) - \MKQ_{\gamma}^{m} g_m(\theta) = \barpsi(\theta)\eqsp.
\end{equation}
The function
\begin{equation}
\label{eq:g_m_theta_def}
g_m(\theta) = \sum_{k = 0}^{\infty}\MKQ_{\gamma}^{km}\barpsi(\theta)
\end{equation}
is well-defined under the assumptions \Cref{ass:mu-convex}, \Cref{ass:L-smooth}, \Cref{ass:rand_noise}($2p$), and \Cref{ass:stationary_moments_bounds}($2p$). Moreover, $g_m$ is a solution of the Poisson equation \eqref{eq:Poisson_eq_iterates}. Define $q := \lfloor n/m \rfloor$, then we have
\begin{equation}
\label{eq:decomposition}
\sum_{k=0}^{q m - 1}\barpsi(\theta_k) = \sum_{r = 0}^{m-1} B_{m,r}\eqsp, \quad \text {with} \quad B_{m,r}= \sum_{k = 0}^{q-1} \bigl\{g_m(\theta_{km+r}) - \MKQ_{\gamma}^{m} g_m(\theta_{km+r})\bigr\}
\eqsp.
\end{equation}
Using Minkowski's inequality, we get from \eqref{eq:whole_part}, that 
\begin{equation}
\label{eq:block_bound}
\PE_{\pi_{\gamma}}^{1/p}\bigl[\norm{\sum_{k=0}^{n-1}\barpsi(\theta_k)}^p\bigr] \leq 
m \PE_{\pi_{\gamma}}^{1/p}\bigl[\norm{\sum_{k=1}^{q} \left\{ g_m(\theta_{km}) - \MKQ_{\gamma}^{m} g_m(\theta_{(k-1)m}) \right\}}^p \bigr] + 2m \PE_{\pi_{\gamma}}^{1/p}\bigl[\norm{\psi(\theta_0}^p\bigr] 
\end{equation}
Now we upper bound both terms of \eqref{eq:block_bound} separately. Under assumption \Cref{ass:L-smooth}, and applying \Cref{ass:stationary_moments_bounds}($2p$), we get
\begin{equation}
\PE_{\pi_{\gamma}}^{1/p}\bigl[\norm{\psi(\theta_0}^p\bigr] \leq \frac{\L}{2}\PE_{\pi_{\gamma}}^{1/p}\bigl[\norm{\theta_0 - \thetas}^{2p}\bigr] \leq \frac{\L \Constlast{2p} \gamma \tau_{2p}^2}{2 \mu} \eqsp.
\end{equation}
To proceed with the first term, we apply Burkholder's inequality  \citep[Theorem~8.6]{osekowski:2012}, and obtain that 
\begin{multline}
\label{eq:burkholder_applied}
\PE_{\pi_{\gamma}}^{1/p}\bigl[\norm{\sum_{k=1}^{q} \left\{ g_m(\theta_{km}) - \MKQ_{\gamma}^{m}g_m(\theta_{(k-1)m}) \right\}}^p \bigr] \\
\leq p \PE_{\pi_{\gamma}}^{1/p}\bigr[\bigl(\sum_{k=1}^{q}\norm{\left\{ g_m(\theta_{km}) - \MKQ_{\gamma}^{m} g_m(\theta_{(k-1)m}) \right\}}^{2}\bigr)^{p/2}\bigr]\eqsp.
\end{multline}
Applying now Minkowski's inequality again, we get 
\begin{align*}
&\PE_{\pi_{\gamma}}^{2/p}\bigl[\bigl(\sum_{k=1}^{q}\norm{\left\{ g_m(\theta_{km}) - \MKQ_{\gamma}^{m} g_m(\theta_{(k-1)m}) \right\}}^{2}\bigr)^{p/2}\bigr] \leq q \PE_{\pi_{\gamma}}^{2/p}\bigl[\norm{\left\{ g_m(\theta_{km}) - \MKQ_{\gamma}^{m} g_m(\theta_{(k-1)m}) \right\}}^{p}\bigr] \\
& \qquad\qquad\qquad\qquad \lesssim q \left( \PE^{2/p}_{\pi_{\gamma}}[\norm{g_m(\theta_{0})}^{p}] + \PE^{2/p}_{\pi_{\gamma}}[\norm{ \MKQ_{\gamma}^{m} g_m(\theta_{0})}^{p}] \right) \\
& \qquad\qquad\qquad\qquad \lesssim q \PE^{2/p}_{\pi_{\gamma}}[\norm{g_m(\theta_{0})}^{p}]\eqsp.
\end{align*}
It remains to upper bound the moment $\PE^{2/p}_{\pi_{\gamma}}[\norm{g_m(\theta_{0})}^{p}]$. In order to do this, we first note that due to the duality theorem \citep[Theorem~20.1.2.]{douc:moulines:priouret:soulier:2018}, we get that for any $k \in \nset$,  
\begin{align*}
\norm{\MKQ_{\gamma}^{mk}\psi(\theta) - \pi_{\gamma}(\psi)} 
&= \sup_{u \in \rset^{d}:\norm{u}=1} |\MKQ_{\gamma}^{mk}(u^{\top} \psi(\theta)) - \pi_{\gamma}(u^{\top}\psi)| \\
&\leq \frac{1}{2} \L_3 \Wass{c}(\delta_{\theta}\MK_{\gamma}^{km}, \pi_{\gamma}) \\
&\leq 2 \L_3 (1/2)^{k}\Wass{c}(\delta_{\theta}, \pi_{\gamma})\eqsp,   
\end{align*}
where the last inequality is due to \Cref{lem:Wasserstein_ergodicity}. Hence, applying the definition of $g_m(\theta)$ in \eqref{eq:g_m_theta_def}, we obtain that 
\begin{align*}
\PE^{1/p}_{\pi_{\gamma}}[\norm{g_m(\theta_{0})}^{p}] \leq  \sum_{k = 0}^{\infty} \PE^{1/p}_{\pi_{\gamma}}\bigl[\norm{\MK_{\gamma}^{k m} \barpsi(\theta)}^p\bigr] \leq 2 \L_{3} \sum_{k = 0}^{\infty} (1/2)^{k}\PE^{1/p}_{\pi_{\gamma}}\bigl[\{\Wass{c}(\delta_{\theta}, \pi_{\gamma})\}^p\bigr]\eqsp.
\end{align*}
To control the latter term, we simply apply the definition of $\Wass{c}(\delta_{\theta}, \pi_{\gamma})$ and a cost function $c(\theta,\theta^{\prime})$ together with \Cref{ass:stationary_moments_bounds}($2p$), we get 
\begin{align*}
&\PE^{1/p}_{\pi_{\gamma}}\bigl[\{\Wass{c}(\delta_{\theta}, \pi_{\gamma})\}^p\bigr] \lesssim \left(\int_{\rset^{d} \times \rset^{d}} \norm{\theta - \theta^{\prime}}^{p}\left(\norm{\theta - \thetas} + \norm{\theta^{\prime} - \thetas} + \frac{\tau_2 \sqrt{\gamma}}{\sqrt{\mu}}\right)^{p} \pi_{\gamma}(\rmd \theta) \pi_{\gamma}(\rmd \theta^{\prime})\right)^{1/p} \\
&\qquad \leq \bigl(\int \norm{\theta - \theta^{\prime}}^{2p} \pi_{\gamma}(\rmd \theta) \pi_{\gamma}(\rmd \theta^{\prime})\bigr)^{1/2p} \bigl(\int \left(\norm{\theta - \thetas} + \norm{\theta^{\prime} - \thetas} +  \frac{\tau_2 \sqrt{\gamma}}{\sqrt{\mu}}\right)^{2p} \pi_{\gamma}(\rmd \theta) \pi_{\gamma}(\rmd \theta^{\prime}) \bigr)^{1/2p} \\
&\qquad \lesssim \frac{\Constlast{2p} \tau^2_{2p} \gamma}{\mu} \eqsp.
\end{align*}
Combining now the bounds above in \eqref{eq:burkholder_applied}, we get that 
\begin{equation}
\PE_{\pi_{\gamma}}^{1/p}\bigl[\norm{\sum_{k=1}^{q} \left\{ g_m(\theta_{km}) - \MKQ_{\gamma}^{m}g_m(\theta_{(k-1)m}) \right\}}^p \bigr] \lesssim \frac{\Constlast{2p} p \L_{3} \tau^2_{2p} \gamma \sqrt{q}}{\mu}\eqsp,
\end{equation}
and, hence, substituting into \eqref{eq:whole_part}, we get
\begin{equation}
\PE_{\pi_{\gamma}}^{1/p}\bigl[\norm{\sum_{k=0}^{n-1}\barpsi(\theta_k)}^p\bigr] \lesssim \frac{\Constlast{2p} p \L_{3} \tau^2_{2p} \gamma \sqrt{q} m}{\mu} + \frac{\L \Constlast{2p}\tau_{2p}^2 \gamma m}{2 \mu}\eqsp.
\end{equation}
Now the statement follows from the definition of $m = m(\gamma)$ in \eqref{eq:m_gamma_def} and $q = \lfloor n/m \rfloor \leq n/m$.

\subsection{Version of \Cref{prop:psi_p_moment_bound} for arbitrary initial distribution $\nu$.}

In order to prove \Cref{th:RR_pth_moment_proof_explicit}, we need a generalization of \Cref{prop:psi_p_moment_bound} for arbitrary initial distribution $\nu$. The following result holds:
\begin{lemma}
\label{lem:Generalization_Rosenthal}
Let $\{\tilde{\theta}^{(\gamma)}_{k}\}_{k \in \nset}$ and $\{\theta^{(\gamma)}_{k}\}_{k \in \nset}$ be defined by the synchronous coupling construction \eqref{eq:kernel_coupling_construction}, where $\tilde{\theta}^{(\gamma)}_{0} \sim \pi_{\gamma}$ and $\theta^{(\gamma)}_{0} \sim \nu$. Then, under assumptions of \Cref{prop:psi_p_moment_bound}, for any $\gamma \in (0, 1/(L\Conststep{6})], n \in \nset$ and initial distribution $\nu$, it holds that 
\begin{multline*}
\PE^{1/p}_{\nu}\bigl[\norm{\sum_{k=n+1}^{2n}\{\psi(\theta_{k}^{(\gamma)}) - \pi_{\gamma}(\psi)\}}^p\bigr] \leq \PE^{1/p}_{\pi_\gamma}\bigl[\norm{\sum_{k=n+1}^{2n}\{\psi(\tilde{\theta}^{(\gamma)}_{k}) - \pi_{\gamma}(\psi)\}}^p\bigr] \\+ \frac{c_0 \L_3(1-\gamma\mu)^{(n+1)/2}}{\gamma\mu}\left(\PE^{1/p}_{\nu}[\norm{\theta^{(\gamma)}_0-\thetas}^{2p}] + \frac{\Constlast{2p} \gamma\tau_{2p}^2}{\mu}\right)\eqsp,
\end{multline*}
where $c_0$ is an absolute constant.
\end{lemma}
\begin{proof}
We consider the synchronous coupling contraction defined in \eqref{eq:kernel_coupling_construction} and denote by $\MKK_{\gamma}$ the corresponding coupling kernel. Hence, we have 
\begin{align*}
&\PE^{1/p}_{\nu}\bigl[\norm{\sum_{k=n+1}^{2n}\{\psi(\theta_{k}) - \pi_{\gamma}(\psi)\}}^p\bigr] = \bigl(\PE^{\MKK_\gamma}_{\nu, \pi_\gamma}\bigl[\norm{\sum_{k=n+1}^{2n}\{\psi(\theta_{k}) - \pi_{\gamma}(\psi)\}}^p\bigr]\bigr)^{1/p} \\
&\qquad \qquad \leq  \PE^{1/p}_{\pi_\gamma}\bigl[\norm{\sum_{k=n+1}^{2n}\{\psi(\tilde{\theta}_{k}) - \pi_{\gamma}(\psi)\}}^p\bigr] + \bigl(\PE^{\MKK_\gamma}_{\nu, \pi_\gamma}\bigl[\norm{\sum_{k=n+1}^{2n}\{\psi(\theta_{k}) - \psi(\tilde{\theta}_{k})\}}^p\bigr]\bigr)^{1/p} \eqsp.
\end{align*}
It remains to bound the last term in the inequality above. Applying Minkowski's inequality together with \Cref{lem: Lipschitzness_remainder}, we get 
\begin{equation*}
\bigl(\PE^{\MKK_\gamma}_{\nu, \pi_\gamma}\bigl[\norm{\sum_{k=n+1}^{2n}\{\psi(\theta_{k}) - \psi(\tilde{\theta}_{k})\}}^p\bigr]\bigr)^{1/p}  \leq \frac{\L_3}{2}\sum_{k=n+1}^{2n}\bigl(\PE^{\MKK_\gamma}_{\nu, \pi_\gamma}[c^p(\theta_k, \tilde{\theta}_k)]\bigr)^{1/p}\eqsp.
\end{equation*}
Using H\"older's and Minkowski's inequalities together with \Cref{ass:stationary_moments_bounds}($2p$) and \Cref{lem: coupling_4_moment}, we obtain that 
\begin{align*}
&\bigl(\PE^{\MKK_{\gamma}}_{\nu, \pi_{\gamma}}[c^p(\theta_k, \tilde{\theta}_k)]\bigr)^{1/p} \\
& \qquad \leq (\PE^{\MKK_{\gamma}}_{\nu, \pi_{\gamma}}[\norm{\theta_k- \tilde{\theta}_k}^{2p}])^{1/(2p)}\bigl(\PE_{\pi_{\gamma}}^{1/(2p)}[\norm{\tilde{\theta}_k -\thetas}^{2p}] + \PE_{\nu}^{1/(2p)}[\norm{\theta_k -\thetas}^{2p} + \frac{2^{3/2}\gamma^{1/2}\tau_2}{\mu^{1/2}}]\bigr) \\ 
& \qquad \leq (1-\gamma\mu)^{k/2}(\PE^{\MKK_{\gamma}}_{\nu, \pi_{\gamma}}[\norm{\theta_0- \tilde{\theta}_0}^{2p}])^{1/(2p)}(\PE_{\nu}^{1/(2p)}[\norm{\theta_0-\thetas}^{2p}] + \frac{2\Constlast{2p}^{1/2}\gamma^{1/2}\tau_2p}{\mu^{1/2}}+ \frac{2^{3/2}\gamma^{1/2}\tau_2}{\mu^{1/2}}) \\
&\qquad \lesssim (1-\gamma\mu)^{k/2}\left(\frac{\Constlast{2p} \gamma\tau_{2p}^2}{\mu} + \PE_{\nu}^{1/p}\norm{\theta_0-\thetas}^{2p}\right)
\end{align*}

Combining all inequalities above, we get
\begin{equation*}
\biggl(\PE^{\MKK_{\gamma}}_{\nu, \pi_{\gamma}}[\norm{\sum_{k=n+1}^{2n}\{\psi(\theta_k) - \psi(\tilde{\theta}_k)\}}^p]\biggr)^{1/p} \lesssim \frac{\L_3(1-\gamma\mu)^{(n+1)/2}}{\gamma\mu}\biggl(\PE^{1/p}_{\nu}[\norm{\theta_0-\thetas}^{2p}] + \frac{\Constlast{2p} \gamma\tau_{2p}^2}{\mu}\biggr)\eqsp,
\end{equation*}
and the statement follows.
\end{proof}

\section{Experimental details}
\label{sec:exp_details}
We recall the error representation \eqref{eq:PR-decomposition-error-extended}, and obtain with simple algebra:
\begin{multline}
\label{eq:PR-decomposition-error-rhs}
\H (\prtheta_{n}^{(\gamma)} - \thetas) + n^{-1}\sum_{k=n+1}^{2n}\noise{k+1}(\thetas) = \frac{\theta_{n+1}^{(\gamma)} - \thetas}{\gamma n} - \frac{\theta_{2n+1}^{(\gamma)} - \thetas}{\gamma n}  \\
- \frac{1}{n}\sum_{k=n+1}^{2n}\{\noise{k+1}(\theta_{k}^{(\gamma)}) - \noise{k+1}(\thetas)\} - \frac{1}{n}\sum_{k=n+1}^{2n}\psi(\theta_{k}^{(\gamma)}) - \frac{1}{n}\sum_{k=n+1}^{2n}G (\theta_{k}^{(\gamma)})\eqsp.
\end{multline}
Under \Cref{ass:rand_noise}($6$), the statistics $\frac{1}{n}\sum_{k=n+1}^{2n}\noise{k+1}(\thetas)$ is a sum of independent random variables, and 
\[
n^{-2}\PE[\norm{\sum_{k=n+1}^{2n}\noise{k+1}(\thetas)}^2] = \frac{\trace{\noisecov}}{n}\,.
\]
Hence, in order to trace the rate of the second-order terms in \eqref{eq:2-moment-bound-optimized-rr}, it is enough to find  the decay rate of the right-hand side in \eqref{eq:PR-decomposition-error-rhs}. We select different sample sizes $n = 250 \times 2^{k}$, where $k = 0,\ldots,14$, and run the SGD procedure \eqref{eq:sgd_recursion_main} based on the constant step sizes $\gamma$ and $2\gamma$, selecting $ \gamma = 1/\sqrt{n}$. Then we construct the associated estimates $\bar{\theta}_{n}^{(\gamma)}$ and $\bar{\theta}_{n}^{(2\gamma)}$. We conduct $M = 320$ independent parallel runs to approximate the expectations. Code to reproduce experiments is provided at \href{https://github.com/svsamsonov/richardson_romberg_example}{https://github.com/svsamsonov/richardson\_romberg\_example}.

\end{document}